\documentclass{amsart}
\usepackage{amsfonts,latexsym,amssymb,amsmath,amscd,amsthm}
\usepackage[dvips]{graphicx}
\usepackage{enumitem}
\usepackage{amscd}
\usepackage[usenames]{color}
\usepackage{srcltx}
\usepackage[normalem]{ulem}
\usepackage{epsfig}
\newcommand{\red}{}
\newcommand{\green}[1]{#1}

\numberwithin{equation}{section}
\newtheorem{teorema}{Theorem}[section]
\newtheorem{propo}[teorema]{Proposition}
\newtheorem{lema}[teorema]{Lemma}
\newtheorem{coro}[teorema]{Corollary}
\newtheorem{defprop}[teorema]{Definition/Proposition}

\theoremstyle{definition}
\newtheorem{defin}[teorema]{Definition}
\newtheorem{nota}[teorema]{Remark}

\theoremstyle{remark}
\newtheorem{ejemplo}[teorema]{Example}

\newcommand{\be}{\begin{enumerate}}
\newcommand{\benot}{\begin{enumerate}[leftmargin=0em,itemindent=2em,itemsep=0em]}
\newcommand{\ee}{\end{enumerate}}

\newcommand{\bi}{\begin{itemize}}
\newcommand{\ei}{\end{itemize}}

\newcommand{\dis}{\displaystyle}

\newcommand{\RR}{{\mathbb R}}
\newcommand{\ZZ}{{\mathbb Z}}
\newcommand{\NN}{{\mathbb N}}

\newcommand{\CC}{{\mathbb C}}
\newcommand{\PC}{{\mathbb P}^1_{\CC}} 
\newcommand{\OO}{{\mathcal O}}
\newcommand{\BB}{{\mathcal B}}
\renewcommand{\AA}{{\mathcal A}}
\newcommand{\MM}{{\mathcal M}}
\newcommand{\mcN}{{\mathcal N}}
\newcommand{\LL}{{\mathcal L}}

\newcommand{\EE}{{\mathcal E}}

\newcommand{\mcR}{{\mathcal R}}
\newcommand{\eps}{\varepsilon}
\renewcommand{\a}{\alpha}

\newcommand{\bal}{{\boldsymbol\alpha}}
\newcommand{\bbeta}{{\boldsymbol\beta}}
\newcommand{\bv}{\mathbf{v}}
\newcommand{\bx}{{\boldsymbol x}}
\newcommand{\bof}{{\boldsymbol f}}
\newcommand{\ba}{{\boldsymbol a}}
\newcommand{\bt}{{\boldsymbol t}}
\newcommand{\by}{\mathbf{y}}
\newcommand{\bz}{{\boldsymbol z}}
\newcommand{\bo}{{\boldsymbol 0}}
\newcommand{\bm}{{\boldsymbol m}}
\newcommand{\bn}{{\boldsymbol n}}
\newcommand{\br}{{\boldsymbol r}}
\newcommand{\bR}{{\boldsymbol R}}
\newcommand{\bH}{{\boldsymbol H}}

\newcommand{\norm}[1]{\left | #1 \right |}
\newcommand{\norma}[1]{\left |\left| #1 \right|\right |}
\newcommand{\con}{\CC \{\bx\}}
\newcommand{\formal}{\CC [\!\!\:[\bx]\!\!\:]}
\newcommand{\lf}{[\hspace*{-1pt}[}
\newcommand{\rf}{]\hspace*{-1pt}]}
\newcommand{\lista}[2]{#1_1,#1_2,\ldots ,#1_{#2}}

\DeclareMathOperator{\dist}{dist}
\DeclareMathOperator{\cl}{cl}

\newcommand{\bibart}[5]{\textsc{#1}. \textit{#2}. #3 (#4), #5.} 
\newcommand{\biblibro}[4]{\textsc{#1}. \textrm{#2}. #3, #4.} 
\newcommand{\labl}[1]{\label{#1}}
\newcommand{\reff}[1]{\ref{#1}}

\begin{document}

\title[Asymptotic expansions with respect to an analytic germ]{Asymptotic expansions and summability\\with respect to an analytic germ}

\author{Jorge Mozo Fern\'andez}
\address[Jorge Mozo Fern\'andez]{Dpto. \'{A}lgebra, An\'alisis Matem\'atico, Geometr\'\i{}a y Topolog\'\i{}a \\
Facultad de Ciencias \\ Campus Miguel Delibes \\
Universidad de Valladolid\\
Paseo de Bel\'en, 7 \\
47011 Valladolid\\
Spain}
\email{jmozo@maf.uva.es}
\thanks{The first author partially supported by the Spanish national project MTM2010-15471}

\author{Reinhard Sch\"afke}
\address[Reinhard Sch\"afke]{
Institut de Recherche Math\'ematique Avanc\'ee\\
U.F.R.\ de Ma\-th\'e\-ma\-ti\-ques et Informatique\\
Universit\'e de  Strasbourg et C.N.R.S.\\
7, rue Ren\'e Descartes\\
67084 Strasbourg cedex\\France
}
\email{schaefke@unistra.fr}
\thanks{The second author was supported in part by grants of the French National
Research Agency (ref.\ ANR-10-BLAN 0102 and ANR-11-BS01-0009)}



\date{\today}



\begin{abstract}
In a previous article \cite{CDMS}, monomial asymptotic expansions, Gevrey asymptotic expansions and monomial summability were introduced and applied to  certain systems of singularly perturbed differential equations.
In the present work, we extend this concept, introducing (Gevrey) asymptotic expansions and summability with respect to a germ of an analytic function in several variables - this includes polynomials.
The reduction theory of sin\-gu\-la\-rities of curves and monomialization of germs of analytic functions are crucial to establish properties of the new notions, for example a generalization of the Ramis-Sibuya theorem for the existence of Gevrey asymptotic expansions.
{Two examples of singular differential equations are presented for which the formal solutions 
are shown to be summable with respect to a polynomial: one ordinary and one partial differential equation.
}
\end{abstract}

\maketitle

\section{Introduction}

The concept of asymptotic expansion for complex functions in one variable is well established and widely used since Poincar\'e, in order to give a meaning to { divergent} formal power series that appear as solutions of different functional equations, and to understand the behavior near singular points of { analytic} solutions and other special functions. { We mention only the books of W.\ Wasow \cite{Wasow65} and F.W.J.\ Olver \cite{O}.}

{ The closely related notion of summability was introduced to provide unique analytic functions
having certain asymptotic expansions. In one complex variable, it has been extensively used
in such different fields as ordinary differential equations, the analytic classification of formal objects and some classes of singularly perturbed differential equations and partial differential equations.}

For {asymptotic expansions in} several variables, different approaches exist in the literature. Let us mention the approaches of R.~G\'erard and Y.~Sibuya \cite{GS}, who treated some class of Pfaffian systems, and a more powerful one by H.~Majima \cite{Majima84}. This last author introduced the concept of strong asymptotic expansion in polysectors in order to study several classes of singularly perturbed differential equations and integrable connections.

{Several problems suggested asymptotic expansions in several variables in which monomials
$x^py^q$ are crucial: W.~Wasow \cite{Wasow79} studied equations of the form
$$
\eps^h x^{-k} \frac{d\by}{dx} =A(x,\eps )\by,
$$
where $h$, $k\in \NN$ and $A(x,\eps )$ is a matrix holomorphic near $\eps=0$ and  $x=\infty$. They are singular both in the variable $x$ and in the parameter $\eps$.
J.~Martinet and J.-P.~Ramis \cite{MR2} studied the analytic classification of resonant singularities
of holomorphic foliations in two variables. The formal normal form
involved the monomial $u=x^py^q$. The normalizing transformations are $(k,p,q)$-summable
in the following sense: they are locally defined as
$$(x,y)\mapsto (x\exp\{q\,h(x,y)\}, y\exp\{-p\,h(x,y)\}),\ \mbox{where}\ h(x,y)=f(x^py^q)(x,y)$$
and $f$ is obtained as the $k$-sum of an
element  $\tilde{f}\in \CC \{ x,y\} \lf u\rf$, $k$-summable in the variable $u$ with coefficients
in $\CC\{x,y\}$. L.\ Stolovitch \cite{Stolo} used a similar construction in $n\geq2$ variables.

These examples lead M.~Canalis-Durand and the authors \cite{CDMS} to a detailed investigation of
the concept of \emph{monomial asymptotic expansion} in two variables.
One possible definition is the above one given by \cite{MR2}; \cite{CDMS} gives a more algorithmical
definition. In the case of $p=q=1$, a power series $f(x,y)$ is   $(k,1,1)$-summable if rewritten
$f(x,y)=\sum_{n=0}^\infty (a_n(x)+b_n(y))(xy)^n$ all series $a_n(x)$, $b_n(y)$ have a common radius of convergence
$R>0$ and for sufficienly small $r>0$ the series $Tf(u)=\sum_n c_nu^n\in \OO_b(D(0,r)^2)\lf u\rf $
is  $k$-summable where $c_n(x,y)=a_n(x)+b_n(y), |x|<r,|y|<r$ are elements of the Banach space $\OO_b(D(0,r)^2)$
of bounded holomorphic functions on $D(0,r)^2$. \cite{CDMS} applied this definition to doubly singular
ordinary differential equations of the form
$$
\eps^\sigma x^{r+1} \frac{d\by}{dx}= f(x,\eps, \by ),
$$
where $f(0,0,\bo) =\bo$ and $\frac{\partial f}{\partial \by}(0,0,\bo)$ is invertible.

{\red Returning to the problem of classification of holomorphic foliations,
observe that  resonant ones are some of the models that appear as a final step in the reduction
of singularities of holomorphic foliations in dimension two, in a situation in which the set of
separatrices and the divisor have normal crossings.
Some examples of vector fields with nilpotent linear part leading to foliations without normal crossings
have been previously studied. For instance, F.~Loray \cite{Loray} 
considered 
generic perturbations of
the system $\dot x  =2y,\ \dot y  =3x^2$ with hamiltonian $h=y^2-x^3$.
He obtained a formal normal form 
involving formal power series in $h$.
In \cite{CDS}, it was shown that the normal form always contains summable series in $h$. 
In this context an extension of monomial summability to summability with respect to a polynomial
will be useful. It is conjectured that there is a normalizing transformation that is
summable with respect to $y^2-x^3$.}

We choose to 
study asymptotic expansions with respect to a germ of an analytic function as it is no more
difficult than asymptotic expansions with respect to a polynomial. 
Such a concept
needs to behave properly with respect to blow-ups in both
directions: properties of the asymptotic expansions must be preserved when you blow-up, and properties
of the blow-up must give properties of the original asymptotic expansion.%
\footnote{Let us remark here that in \cite{tesisSergio,CarrilloMozo}, the authors
study the behaviour of monomial asymptotics under blow-ups in
order to establish certain Tauberian theorems useful to study properties of Pfaffian systems.}

The purpose of the present work is to introduce such concepts of asymptotic expansions and
summability with respect to a germ of  an analytic function in an arbitrary number of variables.
We use blow-ups of centers of codimension two and ramifications, in the style of
Rolin, Speissegger and Wilkie \cite{RSW}, who, at their turn, follow the ideas of Bierstone and Milman \cite{BM}. The reader should note that we haven't used the full power of desingularization techniques -- this
undoubtedly deserves a further study. Throughout this work, different techniques are used, among them,
induction on the number of steps needed to monomialize the analytic germ, a
Generalized Weierstrass Division Theorem, and Ramis-Sibuya Theorems.}

The structure of this work is as follows: In Section 2, we present some tools needed in the present work.
Among them are a normalization result adapted from \cite{RSW} and a Generalized Weierstrass Division
Theorem from \cite{AHV}. Both are proved here in a simplified version adapted to our needs. They are
used to deal with bounded quotients of germs of analytic functions (i.e., for the elimination of indeterminacies,
Lemma \reff{lemadivision}), and to rewrite a formal power series in terms of powers of 
a germ of an analytic function 
(Corollary \reff{coro-P-germ}). In Section 3,  we recall some properties of classical  and monomial
asymptotic expansions and present the latter in a more general setting than in \cite{CDMS}.
In particular, monomial asymptotic expansions are given for an arbitrary number of variables. 
Certain operators are defined, that transform, both in the formal
and in the analytic setting, monomial asymptotic expansions into asymptotic 
expansion in one variable with coefficients in a Banach space -- the inverse transformations are
given by simple substitution operators. 

Asymptotic expansions with respect to an
arbitrary germ of an analytic function, which are the main object of this work, are defined in
Section 4, and their main properties are established. In order to study them, we construct new operators
analogous to the previous section
that
transform asymptotic expansions with respect to a germ into asymptotic expansions
with respect to one variable 
with coefficients in some Banach space (see Theorem \reff{st-form}).
In the analytic setting, these operators are constructed in Theorem 
\reff{T-sector-germ}; this is one of the main results of this section. Its rather technical 
proof is given in Section 5. It uses induction with respect to the number of steps needed to monomialize the
analytic germ.

In Section 6, the behaviour with respect to blow-ups with 
centers of codimension two is established for the new concept of
asymptotic expansion with respect to a germ of an analytic function. 
While a function $f$ having an 
asymptotic expansion with respect to a germ $P$ clearly also has a corresponding asymptotic expansion 
after blow-up, the converse is more interesting and it is proved in  Theorem \reff{asymp-blu+}. 
For $P$-asymptotic series, i.e.\ series appearing as asymptotic expansions with respect to a germ $P$
for certain functions, we prove an analogous result (Theorem \reff{aserie-blu+}).

In Section 7, 
Gevrey asymptotic expansions
with respect to an analytic germ are defined  and investigated, and subsequently, summability with respect 
to such a germ. 
 We also study the behavior of these concepts with respect to blow-ups.

{Finally, in Section 8, we present two examples of singular differential equations for which the formal solutions 
are summable with respect to a polynomial: one ordinary and one partial differential equation.
The examples suggest that asymptotic expansions with respect to an analytic germ
will play an important in these theories and in the theory of foliations. They also illustrate 
the application of our results, in particular concerning blow-ups, in  proving summability.}

%

\subsection*{Acknowledgments}
The first author wishes to thank the University of Strasbourg and the second author wishes to thank the University of Valladolid  
for the hospitality during their visits while preparing this article.

\section{Preliminaries}
\subsection{Notation} \labl{notation}
The following notation will be used throughout this work.
We fix an integer $d\geq2$.
$$
D(\bo; (\lista rd)):=\{ \bx=(\lista xd)\in \CC^d;\ \norm{x_j}<r_j,\
\mbox{ for }j=1,\ldots,d\}
$$
is a polydisk around the origin. As an abbreviation, $D(\bo; \rho):=
D(\bo; (\rho,\ldots, \rho))$. If $U$ is an open set, in $\CC$ or in
$\CC^d$, $\OO (U)=\OO(U;\CC)$ is the set of complex valued
functions holomorphic on $U$,
and $\OO_b (U)$ the subset of bounded holomorphic functions.  Analogously, if $E$ is a Banach space, $\OO (U;E)$ will denote the set of $E$-valued holomorphic functions on $U$, and $\OO_b (U;E)$ the set of bounded $E$-valued holomorphic functions.
$\OO=\con=\CC\{\lista xd\}$ is the ring of germs of analytic functions at the
origin (convergent power series), and $\hat{\OO}=\formal$ the ring
of formal power series. We denote by $\mathfrak{m}$ the maximal
ideal of both (local) rings. There are natural inclusions
$$
\OO_b (D(\bo; \rho))\subseteq \OO( D(\bo; \rho)) \subseteq  \con
\subseteq  \formal,
$$
as well as the relation $\OO=\bigcup_{\rho>0}\OO_b(D(\bo;\rho))$
that we will not detail. For an element $f$ of one of those
rings, $J(f)$ will be its power series at the origin, and
$J_m(f)$ its $m$-jet, i.e., the
polynomial of degree at most $m$
obtained from $J(f)$ deleting the terms of degree greater than $m$.

We use $\PC=\CC\cup\{\infty\}$ with the usual topology.
For $\bx=(\lista xd)\in\CC^d$ put $\bx'=(x_2,\ldots ,x_d)$ and $\bx''=(x_3,\ldots ,x_d)$.

\newcommand{\blu}{blow-up}
\newcommand{\blus}{blow-ups}
\subsection{Normalisation}\labl{normal}
Our approach uses \blus{} at several essential points and we would like to recall
the well known result we use. Our presentation follows that of \cite{RSW} (who base their work
on \cite{BM}); the results
themselves are classical.

Following \cite{RSW}, we will only use \blus{} of codimension two smooth varieties
and so we only recall this case. Assume that the center of the \blu{}
is $x_1=x_2=0$ and define
$$M=\{([u_1,u_2],\bt)\in\PC\times\CC^d;u_1t_2=u_2t_1\}$$
the \blu{} variety and
$$\begin{array}{rcl}b:M&\to&\CC^d\\
    ([u_1,u_2],\bt)&\mapsto&\bt
\end{array}  $$
the \blu{} map (shortly \blu).
$M$ is covered by affine charts, each
one analytically equivalent to $\CC^d$. In fact, identifying
$\PC\cong\bar\CC=\CC\cup\{\infty\}$ as $[1,\xi]\equiv\xi\in\CC$,
$[0,1]\equiv\infty$, we use the charts centered in $\xi\in\CC$,
$$\phi_\xi:\left\{\begin{array}{rcl}M_\xi&\to&\CC^d\\
  ([u_1,u_2],\bt)&\mapsto&(\frac{u_2}{u_1}-\xi,t_1,\bt''),\end{array}\right.$$
where $M_\xi=M_0=\{([u_1,u_2],\bt)\in M;u_1\neq0\}$ and the chart at $\infty$
$$\phi_\infty:\left\{\begin{array}{rcl}M_\infty&\to&\CC^d\\
  ([u_1,u_2],\bt)&\mapsto&(\frac{u_1}{u_2},t_2,\bt''),\end{array}\right.$$
where $M_\infty=\{([u_1,u_2],\bt)\in M;u_2\neq0\}$.
Then the map $b$ is described by
$$b_\xi=b\circ\phi_\xi^{-1}:\left\{\begin{array}{rcl}\CC^d&\rightarrow&\CC^d\\
             \bv&\mapsto&(v_2,(\xi+v_1)v_2,\bv'')\end{array}\right.$$
in the chart at $\xi$ and by
$$b_\infty=b\circ\phi_\infty^{-1}:\left\{\begin{array}{rcl}\CC^d&\rightarrow&\CC^d\\
             \bv&\mapsto&(v_1v_2,v_2,\bv'').\end{array}\right.$$
in the chart at $\infty$. The reason for our somewhat unusual choice of
$\phi_\xi$ and hence $b_\xi$ will become clear in section 
\reff{blu}; otherwise the choice of charts is not important.
For $k\in\NN$, $k\geq 2$, we introduce the ramification
$$r_k:\left\{\begin{array}{rcl}\CC^d&\rightarrow&\CC^d\\
             \bt&\mapsto&(t_1^k,\bt').\end{array}\right.$$
We say that $f\in\OO$ has {\em normal crossings} (at the origin)
if there is a diffeomorphism\footnote{
i.e.\ its Jacobian at the origin is an invertible matrix.}
$D\in\mbox{Diff}(\CC^d,\bo)$ such that
$$(f\circ D)(\bx)=x_1^{\ell_1}\cdot \ldots\cdot x_d^{\ell_d}\cdot U(\bx)$$
with non-negative integers $\ell_j$ and a unit $U\in\OO$,
i.e.\ a germ satisfying $U(\bo)\neq0$.

It has been shown in \cite{RSW}
\begin{lema}\labl{reduction}
There exists a function $h:\OO\setminus\{0\}\to\NN$ with the following properties:
\be\item If $h(f)=0$ then $f$ has normal crossings.
\item If $h(f)>0$ then there exists a diffeomorphism  $D\in\mbox{Diff}(\CC^d,\bo)$
such that either for all $\xi\in\PC$
$$h(f\circ D\circ b_\xi)<h(f)$$
or there exists $k\in\NN$, $k\geq2$ such that $h(f\circ D\circ r_k)<h(f)$.
\ee\end{lema}

Observe that this Lemma can be applied simultaneously to a finite number of germs
$f_1,f_2,\ldots, f_r$. It suffices to consider their product $f=f_1\cdot f_2\cdot \ldots \cdot f_r$.
Moreover, it is elementary (see \cite{RSW}) that $f_1\cdot f_2\cdot \ldots \cdot  f_r$ has normal crossings
if and only if every $f_i$ has normal crossings.

In a more general setting (quasi-analytic classes), this result is proved in \cite{RSW}, following some general ideas adapted from \cite{BM}. In the analytic situation, the result is much easier. We sketch a proof for completeness, omitting most technical details.

\begin{proof}

Use induction on $d$. Every $f\in \CC \{ x\}$ has normal crossings, so  $h(f)=0$ and the result is trivial if $d=1$. Assume now that $d>1$, and let $n$ denote the order of $f$.
By a linear change of variables,
$f$ can be made $x_d$-regular of order $n$, that is $f(0,\ldots,0,x_d)$ has valuation $n$.
Indeed, as $f_n(\a_1,\ldots,\a_d)\neq0$ (where $f_n$ is the homogeneous component of degree $n$ of $f$) for some sufficiently small vector $\bal$ with $\a_d\neq0$, the function
$\tilde f(t_1,\ldots,t_d)=f(t_1+\a_1 t_d,\ldots ,t_{d-1}+\a_{d-1} t_d, \a_dt_d)$ is regular.

If $f$ is $x_d$-regular, the Weierstrass Preparation Theorem  allows us to
write $f (\bx )=(x_d^n+c_1(\bx') x_d^{n-1}+\cdots + c_n (\bx '))U(\bx)$ where $U\in\CC\{\bx\}$,
$U(\bo)\neq0$ and $c_1,\ldots,c_n$ vanish at the origin. Hence it is sufficient to continue with the
polynomial factor. Another change of variable ($t_d=x_d+c_1(\bx')/n$, $\bt'=\bx'$)
eliminates $c_1 (\bx ')$, so we consider only the case $c_1 (\bx ')\equiv 0$.
Let $\MM$ denote the set of $j\in\{2,\ldots,n\}$ such that $c_j\neq 0$ which we assume to be non-empty.

Now apply  the induction hypothesis to the product
$$\dis\prod_{k\in\MM} c_k (\bx ') 
\prod_{i,j\in\MM,i<j} (c_i (\bx')^{n!/i}-c_j (\bx ')^{n!/j}).$$

By a sequence of right compositions with diffeomorphisms,
blow-ups and ramifications, we arrive at a situation where the product has normal crossings.
Thus every $c_j (\bx ')$ and every $c_k (\bx')^{d!/i}-c_j (\bx ')^{d!/j}$ in the product
has normal crossings.  It is easy to deduce
(see Lemma 4.7 in \cite{BM}) that
if $c_j (\bx ')=\bx'^{\gamma_j} \cdot U_j (\bx ')$ with $U_j (\bo )\neq 0$,
then the set $\{\frac{n!}j\gamma_j\mid j\in\MM\}$ is totally ordered with
respect to the relation $\preceq$
defined by $\mathbf{\alpha}\preceq \mathbf{\beta}$ if $\alpha_i\leq \beta_i$ for all $i$.
A ramification in the first $d-1$ variables allows us to suppose that, moreover,
$\gamma_j$ is divisible by $j$ for all $j\in\MM$ and thus the subset
$\{\frac{1}j\gamma_j\mid j\in\MM\}$ of $\NN^{d-1}$  is totally ordered.
Let $l\in\MM$ be such $\frac1l\gamma_l \leq \frac1j\gamma_j$ for every $j\in\MM$.


Denote $\gamma_j= (\gamma_{j1},\ldots , \gamma_{j,d-1} )$, $j\in\MM$.
For the above $l$, let $k$ be the largest index such that $\gamma_{lk}\neq 0$
(and consequently,  $\gamma_{jk}/j\geq \gamma_{lk}/l\geq 1$ for every $j\in\MM$ ).

Now, blow-up with center $x_k=x_d=0$.  For $\xi=\infty$, this means replacing
$x_k$ by $x_kx_d$. Then for $j\in\MM$,
the term $\bx'^{\gamma_{j}} \cdot x_d^{n-j}$ is transformed into $\bx'^{\gamma_j}\cdot
  x_d^{\gamma_{jk}+n-j}$
which can be divided by $x_d^n$, because $\gamma_{jk}\geq j$. Hence $f$ is transformed
into $\tilde f(\bx)=x_d^{n}\left(1+\sum_{j\in\MM} \bx'^{\gamma_j}q_j(\bx)x_d^{\gamma_{jk}-j}\right)$
with some $q_j$ analytic at the origin. Clearly, $\tilde f$ has normal crossings.

For $\xi\in\CC$, this blow-up means to
replace $x_d$ by $x_k (x_d+\xi )$. Then for $j\in\MM$,
the term $\bx'^{\gamma_{j}} \cdot x_d^{n-j}$ is transformed into
$$
\left( \prod_{i\neq k,d} x_i^{\gamma_{ji}} \right) \cdot x_k^{\gamma_{jk}+n-j} (x_d+\xi )^{n-j},
$$
which can be divided by $x_k^n$, because $\gamma_{jk}\geq j$.
Thus $f(\bx)$ is transformed into
$$\tilde f(\bx)=x_k^n\left((\xi+x_d)^n+\tilde c_2(\bx')(\xi+x_d)^{n-2}+\ldots +\tilde c_n(\bx')\right),$$
where
$\tilde c_j(\bx')=x_k^{-j}c_j(\bx')$ is analytic at the origin. We continue then with the second
factor $\bar f(\bx)=x_k^{-n}\tilde f(\bx)$.

If all $\tilde c_j(\bo)$ vanish, we have normal crossings for $\xi\neq0$, but have to
blow-up another time if $\xi=0$ (we do not use ramifications as $\gamma_j$
is already divisible by $j$ for $j\in\MM$). After a finite number of such blow-ups, at least one of the
$\tilde c_j(\bo)\neq 0$. The second
factor $\bar f(\bx)=x_k^{-n}\tilde f(\bx)$
is then of lower order than $f$ for all $\xi$ and also $x_d$-regular. Indeed,
$\bar f(0,\ldots ,0,x_d)=(\xi+x_d)^n+\tilde c_2(\bo)(\xi+x_d)^{n-2}+\ldots+\tilde c_n(\bo)$
might vanish at $x_d=0$, but because it has no term with $(\xi+x_d)^{n-1}$,
it cannot be equal to $x_d^n$ and hence
cannot vanish of order $n$. As $\bar f(0,\ldots ,0,x_d)$ contains some term $x_d^m$, $m<n$,
this term is also in $\bar f$ and hence it is at most of order $m$.
\end{proof}

A first application of this monomialisation Lemma is the following statement.
It may be well-known to specialists, but lacking precise references we include a proof
for the sake of completeness.
\begin{lema} \labl{lemadivision}
Let $f,g\in\OO$, $f(\bo)=g(\bo)=0$, $g\neq0$ be germs of holomorphic functions.
Assume that both are defined on $D(\bo;r)$ and that
$\norm{\frac{f(\bx)}{g(\bx)}}$ is bounded on the set $A_{\theta,r}$
of all $\bx\in D(\bo;r)$ such that $g(\bx)\neq 0$ and $\arg g(\bx)=\theta$.
Then $g$ divides $f$, i.e.\ there exists $q\in\OO$ such that $f=q\cdot g$.
\end{lema}
\begin{proof}
Note first that this result is much easier in dimension one.
Indeed, consider  functions $f,g$ holomorphic in some
neighbourhood of $0\in \CC$ such that $f(0)=g(0)=0$, $g\neq0$. Then
the quotient $f/g$ has at most a pole at 0, if we restrict
ourselves to a small enough neighbourhood of the origin. If there exist a
sequence $\{ t_n\}_{n=1}^\infty $ such that $\lim_{n\rightarrow
\infty }t_n=0$ for which $\norm{f(t_n)/g(t_n)}$ is bounded,
the origin cannot be a pole, so it is a removable
singularity: there exists a germ $q\in\OO$ such
that $f=g\cdot q$.

In arbitrary dimension, we proceed by induction on $h(f\cdot g)$.
At some points it might be necessary to reduce $r$ but we do not
always mention this.
If $h(f\cdot g)=0$, $f$ and $g$ have normal crossings and we can assume that
$$f(\bx)=x_1^{m_1}\cdot\ldots\cdot x_d^{m_d}U_1(\bx);\
    g(\bx)=x_1^{\ell_1}\cdot\ldots\cdot x_d^{\ell_d}U_2(\bx)$$
with  germs $U_1,U_2$, $U_1(\bo)\cdot U_2(\bo)\neq0$. If all $\ell_j$ vanish, there is
nothing to show. Otherwise, we can assume that $\ell_1>0$.

For fixed $x_2,\ldots,x_d\neq0$ and if $r$ is sufficiently small,
we can apply the one-dimensional result with respect to the variable
$x_1$ and appropriate $\theta_1$.
This implies that $m_1\geq \ell_1$. Similarly, we obtain that
$m_j\geq\ell_j$ for all $j$ and the statement follows.

Assume now that the statement is true for all couples $(f,g)$ with
$h(f\cdot g)\leq m$ for some $m\in\NN$.
Consider some couple $(f,g)$ satisfying the assumptions of the lemma
and $h(f\cdot g)=m+1$. As the statement is stable with respect to right composition by
diffeomorphisms, we can assume that $h((f\circ b_\xi)\cdot(g\circ b_\xi))\leq m$
for all $\xi\in\PC$ or that there is a $k\in\NN$ with $h((f\circ r_k)\cdot(g\circ r_k))\leq m$.
In the latter case, the assumption of the lemma implies that $f(r_k(\bt))/g(r_k(\bt))$
is bounded on the set of all sufficiently small $\bt\in\CC^d$
such that $g(r_k(\bt))\neq0$ and $\arg(g(r_k(\bt)))=\theta$.
Hence there exists $q\in\OO$ such that
\begin{equation}\labl{qrk} f\circ r_k = (g \circ r_k)\cdot q.
\end{equation}
As $f\circ r_k$ and $g\circ r_k$ are invariant under the rotation
$\bt\mapsto (e^{2\pi i/k}t_1,\bt')$, so is $q$ and hence there is a germ
$Q\in\OO$ such that $q=Q\circ r_k$. We obtain that $f=g\cdot Q$ and the statement
follows.

Consider now the first case that $h((f\circ b_\xi)\cdot(g\circ b_\xi))\leq m$ for all $\xi\in\PC$.
The assumption of the lemma implies that $f(b_\xi(\bt))/g(b_\xi(\bt))$
is bounded on the set of all sufficiently small $\bt\in\CC^d$
such that $g(b_\xi(\bt))\neq0$ and $\arg(g(b_\xi(\bt)))=\theta$.
Hence for every $\xi\in\PC$ there exists $q_\xi\in\OO$ such that
\begin{equation}\labl{qxi} f\circ b_\xi = (g \circ b_\xi)\cdot q_\xi.
\end{equation}

Thus for every $\xi\in\PC$, there is an open neighborhood $W_\xi$ of $\bo$,
where $f\circ b_\xi$, $g\circ b_\xi$ and $q_\xi$ are defined.
With these we consider the open neighborhoods $U_\xi=\phi_\xi^{-1}(W_\xi)$ of
$(\xi,\bo)\in M$ (see Subsection \reff{normal} for notation) and
the holomorphic $Q_\xi:U_\xi\to\CC$ defined by $Q_\xi=q_\xi\circ \phi_\xi$.
By definition, for $\xi,\zeta\in\PC$, we have
$Q_\xi (p)= Q_\zeta (p)$ for $p\in U_\xi\cap U_\zeta$ with $g(b(p))\neq 0$.
As $g$ is not identically zero, this means that $Q_\xi$ and $Q_\zeta$ coincide on an open and dense subset
of $U_\xi\cap U_\zeta$. Therefore $Q_\xi=Q_\zeta$ on this intersection and thus all
$Q_\xi$, $\xi\in\PC$ define a holomorphic function $Q:U\rightarrow \CC$,
where $U$ is some neighbourhood of $\PC\times \{ \bo \}\subseteq M$.
$\PC$ being compact, $Q$ is constant over it, so there exists a holomorphic $q:V\rightarrow \CC$,
$V$ a neighbourhood of $\bo\in \CC^d$, such that $q\circ b= Q$ (apply Hartogs' Theorem).
By construction, we have $f\circ b_\xi = (g\circ b_\xi )\cdot (q\circ b_\xi)$ for all $\xi\in\PC$
and thus we obtain that  $f= g\cdot q$, as desired.
\end{proof}

Lemma \reff{lemadivision} is basic for our article, but especially for the study of
Gevrey asymptotics with respect to an analytic function, we need more quantitative information
about division.
\begin{lema} \labl{lemadivision2}
Let $D',D$ denote two open connected subsets of $\CC^d$ such that the closure
of $D'$ is contained in $D$ and compact. Suppose that $P$ is holomorphic in
$D$ and not identically vanishing. Let $\theta\in\RR$ and
$A_{\theta,D}=\{\bx\in D\mid P(\bx)\neq0,\,\arg(P(\bx))=\theta\}$. For $n\in\NN$ let
$M_n$ denote the set of all functions $f\in\OO_b(D)$ such that
$f/P^n$ is bounded on $A_{\theta,D}.$

Then $f/P^n$ can be analytically continued to $D'$ if $f\in M_n$. Furthermore,
there is a constant $K>0$ such that the mappings $q_n:M_n\to\OO_b(D')$
associating to $f\in M_n$ the analytic continuation of $f/P^n$ to $D'$
have norms $\leq K^n$ (provided $M_n$ and $\OO_b(D')$ are equipped
with the maximum norm).\end{lema}
\begin{proof} The first statement is proved by applying Lemma
\reff{lemadivision} in the
neighborhood of each point of $D'$.

For the second statement, it is sufficient to prove the existence of such a
constant in some neighborhood of every point of the closure of $D'$ and then use
the compactness of the latter.

In order to show it in some neighborhood of some point $\bx_0$ of the closure of $D'$,
one uses induction on $h(P\circ T_{x_0})$, $T_{x_0}$ the translation $\bx\mapsto\bx+\bx_0$,
similarly to the proof of Lemma \reff{lemadivision}.
If $P(\bx_0)\neq0$, the statement is trivial, if $P\circ T_{x_0}$ has normal crossings,
the existence of such a constant follows by Schwarz's Lemma.
Indeed, assume that $D(\bx_0,\br)\subset D'$ and that $(P\circ T_{x_0})(\bx)=\bx^\bal U(\bx)$,
$U$ a unit, i.e.\ $U(\bo)\neq0$, and without loss of generality (reduce $\br$ otherwise)
there is a constant $\mu>0$ such that $\norm{U(\bx)}\geq\mu$ for $\bx\in D(\bx_0,\br)$.
Given $f\in M_n$, we can apply Schwarz's Lemma repeatedly to $f/U^n$ and obtain that
$$\sup_{\bx\in D(\bx_0,\br)}\norm{q_n(f)(\bx)}\leq \sup_{\bx\in D(\bx_0,\br)}\norm{f(\bx)}
  \br^{-n\bal}\mu^{-n}.$$
Hence $K=\br^{-\bal}\mu^{-1}$ satisfies the wanted estimates in $D(\bx_0,\br)$.

If $h(P\circ T_{x_0})=m>0$, we can assume that either $h(P\circ T_{x_0}\circ b_\xi)<m$
for all $\xi\in\PC$  or there is a $k\in\NN$ such that $h(P\circ T_{x_0}\circ r_k)<m$.
In the former case there exists such a constant $K_\xi$ and a
neighborhood $V_\xi$ of the origin for $P\circ T_{x_0}\circ b_\xi$ ; $\xi\in\PC$
arbitrary. Using the compactness of $\PC$ as in the previous proof, the existence
of such a constant $K$ for $P$ and some neighborhood of $\bo$ follows.

In the latter case ($h(P\circ T_{x_0}\circ r_k)<m$),
we can use the same constant before and after ramification and
only adjust the neighborhood.
\end{proof}

\subsection{Generalised Weierstrass Division}\labl{GWD}
We will present here a version adapted and simplified by Stevens \cite{Stevens} of a generalised Weierstrass Division Theorem, who attributes it to Galligo \cite{Galligo}, but whose original version is due to Aroca, Hironaka and Vicente  \cite{AHV}.
The version (see Lemma \reff{WDT-Ob}) for functions bounded and holomorphic on
certain special neighborhoods of the origin is particularly useful in the sequel.


Let $S$ be either the ring $\OO=\con=\CC\{\lista xd\}$ or $\hat\OO=\mcR\lf \bx\rf $ where $\mcR$ is an integral domain; in most cases we will use $\mcR=\CC$. Each $f\in S$
can be written as a (formal) power series
\begin{equation}\labl{serie}f=\sum_{\bal\in\NN^d}f_\bal\,\bx^\bal\end{equation}
where the monomials $\bx^\bal=x_1^{\a_1}\cdot\ldots\cdot x_d^{\a_d}$ are defined as usual.
Let $\ell:\NN^d\to\RR_+$ be an injective linear form, $\ell(\bal)=\ell_1\a_1+\ldots +\ell_d\a_d$.
As in \cite{Stevens}, we define a total ordering on the monomials by
$$\bx^\bal<_\ell \bx^\bbeta\mbox{ if and only if }\ell(\bal)<\ell(\bbeta).$$
For $f\in S\setminus\{0\}$  written according to (\reff{serie}), we will say that $\bal$ is the
minimal exponent of $f$, and we will denote $ v_{\ell} (f)=\bal$ if
$$\bx^\bal=\min\{\bx^\bbeta\mid \bbeta\in\NN^d\ ;\ f_\bbeta\neq0\},$$
where the minimum is taken according to the ordering $<_\ell$.
Observe that $v_\ell$ is compatible with the multiplication: 
$v_{\ell}(fg)=v_{\ell}(f)+v_{\ell}(g)$ if $f,g\neq0$.
Therefore the multiples of some nonzero $f\in S$ (i.e.\ the elements of $f\cdot S$)
have minimal exponents in $v_{\ell}(f)+\NN^d$. The converse is false in general, of course.

Given a nonzero $P\in S$ with $v_\ell(P)>_\ell\bo$, we introduce the set
\begin{equation}\labl{delta}
\Delta_{\ell} (P)=\{g=\sum g_\bal\,x^\bal\ ;\ g_\bal=0\mbox{ if }\bal\in v_\ell(P)+\NN^d\}.
\end{equation}
In the case of two variables and $v_{\ell}(P)=(a_1,a_2)$, $a_1,a_2>0$, this set can be written as
$$\bigoplus_{j=0}^{a_1-1}x_1^j\CC\lf x_2\rf \oplus x_1^{a_1}
  \bigoplus_{k=0}^{a_2-1}x_2^k\CC\lf x_1\rf .$$
In the general case it is possible to express $\Delta_\ell(P)$ in a similar way, but
we do not write down this cumbersome formula.
In the rest of subsection \reff{GWD} we omit the index $\ell$ for the sake of simplicity.

\begin{lema}\labl{WDT}Let $P\in S$, $P\neq0$ with $\ell(v(P))>0$ and let $\Delta(P)$
be defined by (\reff{delta}). In the case $S=\mcR\lf \bx\rf $, we assume that the coefficient $P_{v(P)}$ of $\bx^{v(P)}$ in $P$
is a unit in $\mcR$. Then for every $g\in S$, there exist unique $q\in S$
and $r\in\Delta(P)$ such that
\begin{equation}\labl{qr}g=q\cdot P + r.\end{equation}
\end{lema}\newcommand{\Sc}{{\mathcal S}}
\begin{proof}[Proof in the case $S=\CC\{\bx\}$] We even prove it in the case of the Banach space
$\Sc_\mu=\OO_b(D_\mu)$, with norm $\norma{\cdot}_{\infty}$, where $D_\mu=D(\bo;(\mu^{\ell_1},\ldots ,\mu^{\ell_d}))$, if $\mu>0$ is sufficiently small.
The set $D_\mu$ has been chosen so that $\norm{x^\bal}\leq \mu^{\ell(\bal)}$ for $x\in D_\mu.$

If we define
$$T_jg(\bx)=\left\{\begin{array}{cl}
  (g(\bx)-g(x_1,\ldots ,x_{j-1},0,x_{j+1},\ldots ,x_d))/x_j&\mbox{if }x_j\neq0\\
  \frac{\partial g}{\partial x_j}(\bx)&\mbox{if }x_j=0
  \end{array}\right.,$$
then with $g\in\Sc_\mu$, also $T_jg\in\Sc_\mu$ and $\norma{T_jg}_\infty\leq
2\mu^{-\ell_j}\norma{g}_\infty$.

We put $\ba=v(P)$.
Clearly, every $g\in\Sc_\mu$ can be written uniquely
\begin{equation}\labl{QR} g=Q_0(g)\bx^\ba +R_0(g)\mbox{ where }Q_0(g)\in\OO,\,R_0(g)\in
\Delta(P)\cap\OO.
\end{equation}
Rewriting $Q_0(g)=T_1^{a_1}\ldots T_d^{a_d}g$ we find that $Q_0$ is a linear operator from $\Sc_\mu$
to itself and satisfies $\norma{Q_0(g)}_\infty\leq2^{\norm\ba}\mu^{-\ell(\ba)}\norma g_\infty$
for all $g\in\Sc_\mu$.

We can suppose $P_\ba=1$ without loss of generality. Then $P=\bx^\ba+\tilde P$
with some $\tilde P\in\Sc_\mu$, $v(\tilde P)>_\ell \ba$. We can rewrite the equation (\reff{qr})
as $q \bx^\ba+r=g-\tilde P q$ which is equivalent to the fixed point equation
\begin{equation}\labl{FP}
q=Q_0(g-\tilde P q)\end{equation}
together with $r=R_0(g-\tilde P q)$.
We can find a constant $K>0$ such that $\norma{\tilde P}_\infty\leq
K \mu^{\ell(v(\tilde P))}$ if $\mu$ is sufficiently small.
This finally yields
$$\norma{Q_0(\tilde P\cdot h)}_\infty\leq K 2^{\norm\ba} \mu^{\ell(v(\tilde P))-\ell(\ba)}\norma h_\infty$$
for all $h\in\Sc_\mu$, $\mu$ sufficiently small. Therefore the right hand side of (\reff{FP})
defines a contraction on $\Sc_\mu$ and hence it has a unique solution if $\mu>0$
is sufficiently small.
As (\reff{FP}) together with $r=R_0(g-\tilde P q)$ is equivalent to (\reff{qr}),
this implies the statement of the lemma in the case of $S=\Sc_\mu$ and also in the case
of $S=\OO=\CC\{\bx\}$.

Concerning the proof in the case of $S=\hat\OO=\CC\lf \bx\rf $,
we define $w(f)=\ell(v(f))$ for $f\in S$. Then $w$ is a discrete valuation on $S$:
$w(f+g)\geq\min(w(f),w(g))$ and $w(fg)=w(f)+w(g)$ and $\delta(g,h)=2^{-w(g-h)}$
makes $S$ into a complete metric space. In the same way as above, the equation (\reff{qr})
of the lemma is equivalent to the fixed point equation (\reff{FP}) and it can be shown that it
has exactly one solution.
\end{proof}
An immediate consequence of the above lemma is
\begin{coro}\labl{coro-P-serie}Under the assumptions of Lemma \reff{WDT},
every $f\in \mcR\lf \bx\rf $ can be written uniquely in the form
\begin{equation}\labl{P-serie}
\hat f =\sum_{n=0}^\infty g_n\cdot P^n
\end{equation}
where $g_n\in\Delta(P)$ for all $n\in\NN$.
\end{coro}
\begin{proof} Using Lemma \reff{WDT} repeatedly, we can write (uniquely)
$$\hat f=g_0+g_1P+g_2P^2+\cdots+g_{N-1}P^{N-1}+q_NP^N$$
where $N\in\NN$, all $g_n\in\Delta(P)$ and $q_N\in \mcR\lf \bx\rf $.
As $N\to\infty$, the statement follows by $\mathfrak m$-adic convergence.
\end{proof}

For later use, we note the statement proved in the first part of the proof of
Lemma \reff{WDT} and prove the analogue of Corollary \reff{coro-P-serie} for $\CC\{\bx\}$.
\begin{lema}\labl{WDT-Ob} Let $\ell$, $P$ as above Lemma \reff{WDT} and let $\Delta(P)$
be defined by (\reff{delta}). For $s>0$ let
$D_s=D(\bo;(s^{\ell_1},\ldots,s^{\ell_d}))$. If $s$ is sufficiently small,
then for every $g\in\OO_b(D_s)$
there exist unique $r\in\OO_b(D_s)$ with $J(r)\in\Delta(P)$ and $q\in\OO_b(D_s)$
such that $g=q\cdot P + r.$

The corresponding operators $Q,R:\OO_b(D_s)\to\OO_b(D_s)$ defined by
$g\mapsto q$ (respectively $g\mapsto r$) are linear and continuous.
\end{lema}
\begin{coro}\labl{coro-P-germ} Under the assumptions of Lemma \reff{WDT}, for every
$f\in\CC\{\bx\}$ there exist $\rho>0$ and a sequence $\{g_n\}_{n\in\NN}$ in $\OO_b(D(\bo;\rho))$
with $J(g_n)\in\Delta(P)$ for all $n$ such that $f$ can be written in the form
\begin{equation}\labl{P-germ}
f(\bx) =\sum_{n=0}^\infty g_n(\bx)\cdot P(\bx)^n\mbox{ for }\norm{\bx}\leq\rho.
\end{equation}
\end{coro}
The functions $g_n$ are uniquely determined by Corollary
\reff{coro-P-serie}.
\begin{proof} For $s>0$ sufficiently small, $f\in\OO_b(D_s)$ and the operators $Q,R$
of the preceding lemma are defined on $\OO(D_s)$. For $N\in\NN$, we obtain
$$f(\bx)=\sum_{n=0}^{N-1}((RQ^n)f)(\bx)P(\bx)^n+(Q^Nf)(\bx)P(\bx)^N$$
by repeated application of Lemma
\reff{WDT-Ob}.

If $\rho\in]0,s]$ is so small that $M:=\sup\{\norm{P(\bx)}\mid \bx\in D(\bo,\rho)\}
  <\frac1{\norma{Q}}$
then we can estimate
$$\sup_{\norm\bx<\rho}\norm{f(\bx)-\sum_{n=0}^{N-1}[((RQ^n)f)\cdot P^n](\bx)}\leq
         (M\norma{Q})^N\sup_{\by\in D_s}\norm{f(\by)}.$$
This proves the statement.\end{proof}

\subsection{A Cousin Problem}\newcommand{\U}{{\mathcal U}}
In the sequel we use the following lemma solving a certain Cousin problem for $\PC$.
Let $I$ denote some finite set, $\infty\in I$, $\U=(U_i)_{i\in I}$ a finite cover
of $\PC$ by open sets. Assume $\infty\in U_\infty$ for simplicity of notation.
Let $(U_{ij})_{i,j\in I}$ denote the collection of intersections
$U_{ij}=U_i\cap U_j$.
\begin{lema} \labl{cousin}
Let ${\mathcal C}^0(\U)$ denote the Banach space of collections $f=(f_i)_{i\in I}$ of
bounded holomorphic $f_i:U_i\to\CC$ such that $f_\infty(\infty)=0$, equipped
with the maximum norm. Let ${\mathcal Z}^1(\U)$ denote the Banach space of collections
$d=(d_{ij})_{i,j\in I}$ of bounded holomorphic $d_{ij}:U_{ij}\to\CC$  satisfying
the cocycle condition
$$d_{ij}(z)+d_{jk}(z)=d_{ik}(z),\mbox{ if }z\in U_i\cap U_j\cap U_k,$$
equipped with the maximum norm.

Then the boundary mapping $\delta:{\mathcal C}^0(\U)\to{\mathcal Z}^1(\U)$ defined by
$$\delta\left((f_i)_{i\in I}\right)=(f_i-f_j)_{i,j\in I}$$
is bijective, linear and continuous and its inverse,
denoted by $\Sigma$, is also continuous.
\end{lema}
\begin{proof} $\delta$ is surjective. Since $H^1(\U,\OO)=0$ as is well
known, there exist, for every $(d_{ij})_{i,j\in I}\in{\mathcal Z}^1(\U)$, a family
$(f_i)_{i\in I}$, $f_i:U_i\to\CC$ holomorphic, such that $d_{ij}=f_i-f_j$ for $i,j\in I$.
The additional condition $f_\infty(\infty)=0$ is achieved by adding the
same suitable constant to each $f_i$. It remains to show that each $f_i$ is bounded.

As $\U$ is a cover of $\PC$, each point $c\in\PC$ is contained in some $U_{i(c)}$; hence
there exists a neighborhood $V_c$ of $c$ the closure of which is contained in $U_{i(c)}$.
Therefore $f_{i(c)}$ is bounded on $V_c$. For $j\in I$, $j\neq {i(c)}$, the function $f_j$ can be written
$f_j=f_{i(c)}+d_{j{i(c)}}$ on $V_c\cap U_j\subset U_{{i(c)}j}$, provided that $U_{i(c)j}\neq  \emptyset$, and therefore $f_j$ is also bounded on
$V_c\cap U_j$ as $d_{{i(c)}j}$
and $f_{i(c)}$ are. We have shown that every $c\in\PC$ has a neighborhood $V_c$ such that
for all $j\in I$, the function $f_j$ is bounded on $V_c\cap U_j$.
By the compactness of $\PC$, a finite number of such neighborhoods $V_c$ covers $\PC$ and the boundedness
of all $f_j$, $j\in I$, follows. This completes the proof that $\delta$ is surjective.

$\delta$ is injective,
because its kernel is $\{0\}$. Indeed, if $\delta((f_i)_{i\in I})=(0)_{i,j\in I},$
then $f_i(z)=f_j(z)$ whenever $z\in U_i\cap U_j$.
Hence $(f_i)_{i\in I}$ is actually the collection of restrictions of some analytic
function $f:\PC\to \CC$ to the $(U_i)_{i\in I}$.
By Liouville's theorem, $f$ is then a constant. The condition
$f_\infty(\infty)=0$ now implies that $f=0$; hence $(f_i)_{i\in I}=0$.

Obviously $\delta$ is linear and continuous. Therefore, by the theorem of the bounded inverse,
its inverse is also continuous.
\end{proof}
In the sequel, we need an extension of the above Lemma to functions depending holomorphically upon parameters.
\begin{lema}\labl{rem-cousin} Consider a collection $(d_{ij})_{i,j\in I}$ of functions holomorphic on 
$B\times U_{ij}$, where $B$ is some open subset of $\CC^m$, $m\geq1$,
satisfying the cocycle condition with respect to the second variable. There exists
a collection $(f_i)_{i\in I}$ of holomorphic functions  on
$B\times U_i$, $i\in I$, such that $f_i(b,t)-f_j(b,t)=d_{ij}(b,t)$ for all $b\in B$
and $t\in U_{ij}$. 

If there exists a function $K:B\to\RR_+$ such that
for every $b\in B$, the collection of functions $(z\mapsto d_{ij}(b,z))_{i,j\in I}$
is bounded by $K(b)$, then the collection of functions $(z\mapsto f_i(b,z))_{i\in I}$ is bounded by $\norma\Sigma\,K(b)$ for
every $b\in B$. 
\end{lema}
\begin{proof} We define the collection $(f_i)_{i\in I}$ by $f_i(b,z)=\phi_i^b(z)$, where
$$(\phi_i^b)_{i\in I}=\Sigma((z\mapsto d_{ij}(b,z))_{i,j\in I})\mbox{ for }b\in B$$
where $\Sigma$ is the operator of Lemma \reff{cousin}. Then the second statement of our Lemma
follows from \reff{cousin}. It is not clear, however, that the functions $f_i$ are holomorphic with repect to
both variables.

In order to prove this, we first choose an  open cover
$V_i$, $i\in I$, of $\PC$ such that for every $i$, the closure $\cl(V_i)$ is a subset of $U_i$.
The set $V_\infty$ can be chosen such that, additionnally, it contains
$\infty$.
Let $\tilde \Sigma$ denote the operator of Lemma \reff{cousin} applied to the cover ${\mathcal V}=(V_i)_{i\in I}$.

Now fix any $b_0\in B$ and choose $\rho>0$ such that $\cl(D(b_0,\rho))\subset B$. Then the collection
$(\tilde d_{ij})_{i,j\in I}$, $\tilde d_{ij}=d_{ij}\mid_{D(b_0,\rho)\times(V_i\cap V_j)}$, consists
of bounded holomorphic functions.
Then we use that for all open subsets $D\subset \CC^m$  the Banach space $\OO_b(D\times {\mathcal V},\CC)$ is canonically
isometrically isomorphic to
$\OO_b(D,\OO_b({\mathcal V},\CC))$ and that Cauchy's formula with respect to the
first variable commutes with the continuous linear operator $\tilde\Sigma$, applied with respect to the second variable.
We obtain that the collection $(\tilde f_i)_{i\in I}$ defined by $\tilde f_i(b,z)=\psi^b_i(z)$, 
$$(\psi_i^b)_{i\in I}=\tilde \Sigma((z\mapsto \tilde d_{ij}(b,z))_{i,j\in I})\mbox{ for }b\in D(b_0,\rho),$$
consists of bounded holomorphic functions on $D(b_0,\rho)\times V_i$, $i\in I$.

We define now a collection of holomorphic functions $(F_i)_{i\in I}$ on $D(b_0,\rho)\times U_i$ by
$F_i(b,z)=\tilde f_k(b,z)+d_{ik}(b,z)$ for $b\in D(b_0,\rho)$, $z\in U_i$ if $z\in V_k$ for some $k\in I$. 
Observe that here $k=i$ is allowed in which case $d_{ii}(b,z)=0$ and hence $F_i(b,z)=\tilde f_i(b,z)$.
Since $V_k$, $k\in I$, form a cover of $\PC$, we can always find some $k$ such that $z\in V_k$.
In that case $z\in U_i\cap V_k\subset U_{ik}$ and $F_i(b,z)$ is defined.
The cocycle condition and the definition of $\tilde f_k$, $k\in I$, imply that the definition
is independent of the choice of $k$ with $z\in V_k$.
In a similar way, we obtain that $F_i(b,z)-F_j(b,z)=d_{ij}(b,z)$ for all $b\in D(b_0,\rho)$ and $z\in U_{ij}$.

As $F_\infty(b,\infty)=\tilde f_\infty(b,\infty)=f_\infty(b,\infty)=0$ for all $b\in D(b_0,\infty)$, we obtain
that $(z\mapsto F_i(b,z))_{i\in I}=\Sigma((z\mapsto d_{ij}(b,z))_{i,j\in I})=(z\mapsto f_i(b,z))_{i\in I}$
for all $b\in D(b_0,\rho)$. Therefore all functions $f_i$ are holomorphic with respect to $(b,z)$.
%
%
%
%
\end{proof}

We will use a consequence of this Lemma for covers of the exceptional divisor in subsequent sections.
See Subsection \reff{normal} for notation.
\begin{lema}\labl{cover-E}Consider an open cover $U_j$, $j=0,\ldots ,N$ of the exceptional divisor
$E=\PC\times \{\bo\}$ in the blow-up variety $M$
. Then there exist a positive constant $C$ and an open cover
$\tilde U$,  $j=0,\ldots ,N$, of E with $\tilde U_j\subset U_j$, $j=0,\ldots ,N$, 
and the following property.

Given a collection of holomorphic functions $D_{i,j}:U_i\cap U_j\to\CC$, $i,j=0,\ldots ,N$
satisfying the cocycle condition $D_{i,j}(p)+ D_{j,k}(p)= D_{i,k}(p)$ for all $i,j,k$ and
$p\in U_i\cap U_j\cap U_k$ there exists a collection
of holomorphic functions $F_i:\tilde U_i\to\CC$ such that ${D_{ij}}(p)=F_i(p)-F_j(p)$
for all $i,j\in\{0,\ldots ,N\}$, $p\in \tilde U_i\cap \tilde U_j$, and
$$\max_{j\in\{0,\ldots ,N\}}\,\sup_{p\in \tilde U_j}\norm{F_j(p)}\, \leq\, C\,
\max_{i,j\in\{0,\ldots ,N\}}\,\sup_{p\in U_i\cap U_j}\norm{D_{i,j}(p)}\leq\infty.$$
\end{lema}\begin{proof}
Without loss in generality, we can assume that $(\infty,\bo)\in U_0$. Then it is sufficient to
prove the lemma under the additional assumption that $(\infty,\bo)$ is not an element of
the other $U_k$.

Consider the projection $R:M\to\PC\times \CC^{d-1}$ defined by
$R((\xi,\bt))=(\xi,(t_1,\bt''))$.
Its restriction to the chart $M_0$ is an analytic diffeomorphism onto its image. 
We will use the geodesic distance $d$ on $\PC\simeq S^2$ and denote for $\xi\in\PC$, $\mu>0$
by $B(\xi,\mu)$ the set of all $\zeta\in\PC$ with $d(\xi,\zeta)<\mu$.
Also let $V_i\subset\PC$ denote the open (in $\PC$) set such that $V_i\times\{\bo\}=U_i\cap E$.
As in the proof of Lemma \reff{rem-cousin} we choose an open cover
$\tilde V_k$, $k=0,\ldots,N$, of $\PC$ such that for every $k$, the closure $\cl(\tilde V_k)$ is a subset of $V_k$.

Fix now some $k\in\{1,\ldots,N\}$. By assumption, for all $\xi\in V_k$, there exists $\rho_\xi>0$
such that $B(\xi,\rho_\xi)\times D'(\bo,\rho_\xi)\subset R(U_k)$ \footnote{Here it is used that $(\infty,\bo)\not\in U_k$
and hence $\infty\not\in V_k$.}. By the compactness of
$\cl(\tilde V_k)$, a finite number of $B(\xi,\rho_\xi)$, $\xi\in\cl(\tilde V_k)$, is sufficient
to cover  $\cl(\tilde V_k)$. Taking the minimum of these $\rho_\xi$ implies that
there exists  $\rho^{(k)}>0$ such that $\tilde V_k\times D'(\bo,\rho^{(k)})\subset R(U_k)$.

In a similar manner, we obtain $\rho_0^{(k)}$, $k=1,\ldots,N$ such that 
$(\tilde V_k\cap\tilde V_0)\times D'(\bo,\rho_0^{(k)})\subset R(U_k\cap U_0)$.
Then let $\rho>0$ denote the minimum of the $2N$ numbers  $\rho^{(k)}$, $\rho_0^{(k)}$, $k=1,\ldots,N$.
It has the property that for $k=1,\ldots,N$ we have
$\tilde V_{k}\times D'(\bo,\rho)\subset R(U_{k})$
and $(\tilde V_0\cap\tilde V_{k})\times D'(\bo,\rho)\subset R(U_{0}\cap U_k)$.

Therefore we can define bounded holomorphic
functions $\tilde D_{j,k}:= D_{j,k} \circ R^{-1}$ on
$(\tilde V_{j}\cap\tilde V_{k})\times D'(\bo,\rho)$ for
$j,k\in\{0,\ldots ,N\}$ not both equal to $0$. For completeness put
$\tilde D_{0,0}=0$ on $\tilde V_0\times D'(\bo,\rho)$.
Now we can apply Lemma  \reff{rem-cousin} and obtain a family of bounded holomorphic
functions $\tilde F_{j}:\tilde V_{j}\times D'(\bo,\rho)\to\CC$, $j=0,\ldots ,N$, such that
$\tilde D_{j,k}=\tilde F_{j}-\tilde F_{k}$ for $j,k=0,\ldots ,N$.
We have, moreover, for all $j\in\{0,\ldots ,N\}$
$$\norma{\tilde F_{j}}_\infty\leq C
\max\{\norma{ \tilde D_{\ell,k}}_\infty\mid \ell,k\in\{0,\ldots ,N\}\}$$
where $C$ denotes the constant associated to the cover $\{\tilde V_{j}\}_{j=0,\ldots ,N}$
in Lemma \reff{cousin}.

Now put $\tilde U_j=R^{-1}(\tilde V_{j}\times D'(\bo,\rho))$ for $j\in\{1,\ldots,N\}$,
$\tilde U_0=R^{-1}(\tilde V_{0}\times D'(\bo,\rho))\cap U_0$ 
and $F_j(p)=\tilde F_j(R(p))$ for $j\in\{0,\ldots ,N\}$, $p\in\tilde U_j$.
Then $\tilde U_j\cap E=\tilde V_j\times \{\bo\}$ for $j=0,\ldots,N$ and hence $\tilde U_j$,
$j=0,\ldots,N$, form
an open cover of $E$. By construction, we have $\tilde U_j\subset U_j$ for $j=0,\ldots,N$
and 
$F_j(p)-F_k(p)=D_{jk}(p)$ for $p\in\tilde U_j\cap\tilde U_j$.

\end{proof}\begin{nota}\labl{rem-cover}
The above Lemma can be extended to collections $(D_{ij})_{i,j=0,\ldots,N}$ depending holomorphically
upon parameters, that is $D_{ij}$ are holomorphic on $T\times U_{ij}$, $T$ an open subset of
$\CC^k$. We obtain collections $(F_i)_{i}$ of holomorphic functions on $T\times \tilde U_i$
and for $t\in T$ the estimates
$$\max_{j\in\{0,\ldots ,N\}}\,\sup_{p\in \tilde U_j}\norm{F_j(t,p)}\, \leq\, C\,
\max_{i,j\in\{0,\ldots ,N\}}\,\sup_{p\in U_i\cap U_j}\norm{D_{i,j}(t,p)}\leq\infty.$$
The proof remains essentially the same, one just has to apply Lemma \reff{rem-cousin}
to $\tilde V_j\times D'(\bo,\rho)\times T$ instead of $\tilde V_j\times D'(\bo,\rho)$, $j=0,\ldots,N.$
Details are left to the reader.
\end{nota}

\section{Classical and Monomial asymptotics and summability}\labl{monomial}
Here we recall the notions and main properties of classical Poincar\'e- and Gevrey asymptotics and
summability in one variable (see, for instance, \cite{R2,Sibuyalibro,Balserlibro,Mireillenotas}) and then the corresponding
theory for monomial asymptotics of \cite{CDMS}. Our presentation follows essentially \cite{CDMS};
the theory of monomial asymptotic expansion is presented for $d\geq2$ variables instead of
2 and is rearranged and abbreviated.

\subsection{Asymptotics in one Variable}
Let $E$ be a complex
Banach space, with norm $\| \cdot \|_E$ and ${\hat f}(x)=\sum a_nx^n \in E\lf x\rf $. A (open)
sector in $\CC$ is a set $V(a,b;r)= \{ x\in \CC \mid a<\arg
x<b,\ 0<|x|<r \}$. We will omit frequently $a$, $b$, $r$, and
speak of a sector $V$. If $f:V\rightarrow E$ is holomorphic, $f$
is said to have $\hat{f}$ as an asymptotic expansion at the
origin if for each $N\in \NN$, there exists
$C(N)>0$ such that
\begin{equation}\labl{onevar}
\norma{f(x)-\sum_{n=0}^{N-1} a_nx^n}_E \leq C(N)\cdot |x|^N
\text{ in }V.
\end{equation}
The asymptotic expansion is $s$-Gevrey if, moreover, $C(N)$
can be chosen as \linebreak
$C(N)=K\cdot A^N \cdot N!^s$, with
constants $K$, $A$. We will write
$f\sim {\hat f}$ and $f\sim_s \hat{f}$ in the $s$-Gevrey case,
respectively.
Observe that $f\sim_s\hat f$ implies that the formal series $\hat f$
is $s$-Gevrey, i.e.\ there exist $C,A>0$ such that
$|a_n|\leq C A^n n!^s$ for all $n\in\NN$. The set of all
such formal series will be denoted by $E\lf x\rf _s$.

Asymptotic expansions are unique, and respect algebraic
operations and differentiation. The so called Borel-Ritt-Gevrey theorem and
Watson's lemma are of great importance. The following result
collects them.

\begin{teorema}\labl{borel-ritt-watson}
Let $V=V(a,b;r)$, ${{\hat f}}\in E\lf x\rf _s$, and $s>0$. Then:
\be
\item If $b-a\leq s\pi$, there exists $f\in {\mathcal O}(V,E)$
such that $f\sim_s {\hat f}$.
\item If $f\in {\mathcal O}(V,E)$ is such that $f\sim_s 0$, then
there are positive constants such that
$$
\norma{f(x)}_E\leq C\cdot \exp
\left({-{A}/{|x|^{1/s}}}\right).
$$
\item If $b-a>s\pi$ and $f_1,f_2\in {\mathcal O} (V,E)$
have $\hat f$ as their $s$-Gevrey asymptotic expansion,
then $f_1=f_2$. \ee
\end{teorema}
Because of the above theorem, a function   $f\in {\mathcal O} (V;E)$
is uniquely determined by its $s$-Gevrey asymptotic expansion $\hat f$,
provided that the opening of $V$ is larger than
$s\pi$. If such a function exists for a formal series $\hat{f}$,
then it is said to be $k$-summable in $V$ with $k=1/s$
and $f$ is called the $k$-sum of $\hat f$ on $V$.
More precisely
\begin{defin}\labl{summone}Let $s>0$, $k=1/s$ and
${{\hat f}}\in E\lf x\rf _s$.
\be
\item The formal series $\hat f$ is
called $k$-summable on $V=V(a,b;r)$, if $b-a>s\pi$ and there
exists a function $f\in \OO(V;E)$ such that $f\sim_s\hat f$.
The uniquely determined function $f$ is called the $k$-sum of
$\hat f$ in the direction $\theta$.
\item
The formal series $\hat f$ is called $k$-summable in the direction
$\theta\in\RR$, if there exist $\delta,r>0$ such that $\hat{f}$
is
$k$-summable on the sector
$V(\theta-s\frac\pi2-\delta,\theta+s\frac\pi2+\delta;r)$. \item
The formal series $\hat f$ is simply called $k$-summable, if it is
$k$-summable in every direction $\theta\in\RR$ with finitely many
exceptions mod $2\pi$.\ee
\end{defin}

The above notion of $k$-summability in a direction $\theta$
does not indicate how to obtain a sum from a given series;
here the following characterization of $k$-summability helps.
\begin{propo}\labl{laplace}
Given $\hat f(x)=\sum a_n x^n\in E\lf x\rf _s$, it is $k$-summability
in a direction $\theta$ if and only if the following statements hold.
\benot\item Its formal Borel transform $g(t)=\sum {a_n}t^n/{\Gamma(1+n/k)}$ is 
analytic in a neighborhood of the origin,
\item the function $g$ can be continued analytically in some infinite sector 
$S=V(\theta-\delta,\theta+\delta;\infty)$ containing the ray $\arg t=\theta$, 
\item it has exponential growth there, i.e.\ there are
there are positive constants such that
$$
\norma{g(t)}_E\leq C\cdot \exp\left({A}/{|t|^{k}}\right)
$$
and hence the Laplace integral
$f (x) =
k\,x^{-k} \int_{\arg t=\tilde \theta} e^{-{t^k}/{x ^k}}
\ {g} (t)\,{t^{k-1}} dt$ defining the
sum of $\hat f$ converges for $x$ in a certain sector
$V=V(\theta-\frac\pi{2k}-\frac{\tilde\delta} k,\theta+\frac\pi{2k}+\frac{\tilde\delta} k;r),$
$0<\tilde\delta<\delta$, and suitably chosen $\tilde\theta$ close to $\theta$.
It satisfies $f\sim_s\hat f$ on $V$, $s=1/k$.
\ee\end{propo}

We recall also the very useful characterization
of functions having an $s$-Gevrey asymptotic expansion
due to J.P.\ Ramis and Y.\ Sibuya
(\cite{R2}, \cite{Si1}, \cite{R1}).

\begin{teorema}\labl{ramissibuya-classic}
Suppose that the sectors $V_j= V(a_j,b_j;r), 1\leq j\leq m$, form
a cover of the punctured disk $D(0;r)$.
Given $f_j: V_j \rightarrow E$ bounded and analytic, assume that
there is a constant
$\gamma>0$ such that
\begin{equation}\labl{26a}
\norma{f_{j_1}(x)-f_{j_2}(x)}_E =
O(\exp(-\gamma/|x|^{1/s}))
\end{equation}
for $x\in V_{j_1}\cap V_{j_2}$, whenever this intersection is non-empty.

Then the functions $f_j$ have common $s$-Gevrey asymptotic
expansions.

Conversely, if a function $f:V\rightarrow E$ having an $s$-Gevrey asymptotic
expansion is given, then a cover $V_j, 1\leq j\leq m$ and functions
$f_j: V_j \rightarrow E$ can be found that satisfy estimates
like (\reff{26a})  and $f=f_1$.
\end{teorema}

Such a family $f_1,\ldots ,f_m$ is sometimes called a {\em
$k$-precise quasi-function}.

In \cite{Si1}, the following complement of the above theorem can be found
\begin{teorema}\labl{compl-classic}
Suppose that the sectors $V_j= V(a_j,b_j;r), 1\leq j\leq m$, form
a cover of the punctured disk $D(0;r)$.
For couples $(j_1,j_2)$ with $V_{j_1}\cap V_{j_2}\neq\emptyset$, let
holomorphic $d_{j_1,j_2}:V_{j_1}\cap V_{j_2} \rightarrow E$ be given
that satisfy the cocycle condition
$d_{j_1,j_2}+d_{j_2,j_3}=d_{j_1,j_3}$ whenever $V_{j_1}\cap V_{j_2}\cap V_{j_3}\neq\emptyset$
and estimates
\begin{equation}
\norma{d_{j_1,j_2}(x)}_E = O(\exp(-\gamma/|x|^{1/s}))
\end{equation}
for $j_1,j_2\in\{1,\ldots ,m\}$ and $x\in V_{j_1}\cap V_{j_2}$ with some  constants
$s,\gamma>0$.

Then there exist bounded holomorphic functions $f_j:V_j\to E$ such that
$d_{j_1,j_2}=f_{j_1}-f_{j_2}$ whenever $V_{j_1}\cap V_{j_2}\neq\emptyset$; moreover
the functions $f_j$ have common $s$-Gevrey asymptotic expansions.
\end{teorema}


\subsection{Monomial Asymptotics}
In \cite{CDMS} the notion of monomial asymptotics in two variables was introduced in order to study doubly singular differential equations. We want to extend this notion to an arbitrary number of variables.

In the sequel, let $\bx^\bal =x_1^{\alpha_1}\cdots x_d^{\alpha_d}$ denote a monomial in the $d$
variables $x_1,\ldots ,x_d$.
 We begin by restating Corollary \reff{coro-P-serie} in a slightly different way :
$\CC$ is replaced
by an arbitrary $\CC$-vector space $E$ and $P=\bx^\bal$.

\begin{lema} \labl{defT-one}
For any vector space $E$, there exists a canonical isomorphism
$$T:E\lf \bx\rf \to\Delta(\bx^\bal,E)\lf t\rf $$
with the property $(Tf)(\bx^\bal)=f$ for all series $f\in E\lf \bx\rf $. Here the symbol
$(Tf)(\bx^\bal)$ means that $t$ is replaced by $\bx^\bal$ in the series $Tf$ and
$\Delta(\bx^\bal,E)$ is defined as in (\reff{delta}), $\CC$ replaced by $E$.
\end{lema}
By abuse of notation, we use the same symbol $T$ for the analogous isomorphism
$T:E\{\bx\}\to\Delta(\bx^\bal,E)\{t\}$ if $E$ is a normed vector space (and consequently, there is a notion of convergence).

For $r>0$ let $\EE_r$ denote the Banach space of all functions $f\in\OO_b(D(\bo;r))$
the series expansion of which $J(f)\in\Delta(\bx^\bal)$. { If $r'<r$ there is a natural restriction map $\EE_r\rightarrow \EE_{r'}$, linear and continuous. The image of $f\in\EE_r$ will be denoted $f|_{\EE_{r'}}$. Similarly, if $f(t)=\sum_{n=0}^\infty f_nt^n\in\EE_r \lf t\rf $ is a formal series, $f(t)|_{\EE_{r'}}$ will represent $\sum_{n=0}^\infty f_n|_{\EE_{r'}} t^n$.
}


In the subsequent lemma, we establish an analogue of the operator $T$ for functions
defined on sectors in a monomial. This lemma generalizes the construction below Lemma 3.5 of
\cite{CDMS} to an arbitrary number of variables.

We call ``sector in $\bx^\bal$", or $\bx^\bal$-sector,  a set\newcommand{\Cstar}{\CC\setminus\{0\}}
\newcommand{\pp}{\Pi}
$\pp = \pp  (a,b; \bR)\subseteq (\Cstar)^d$, $\bR=(R_1,\ldots ,R_d)\in]0,\infty]^d$,
$$
\pp = \{ \bx \in { \CC}^d \mid a<\arg(\bx^\bal)<b,\ 0<|x_j|<R_j, j=1,\ldots ,d\}\ \  .
$$
\begin{nota}
Here and throughout this work, we will only consider sectors in $\CC$, i.e. of opening not greater than $2\pi$. So, in the definition of a $\bx^\bal$-sector, and in subsequent definitions, we will assume implicitely that $b-a\leq 2\pi$.
\end{nota}

\begin{lema}\labl{T-sector}
Let $\pp=\pp(a,b;\bR)$ a sector in $\bx^\bal$ and $f:\pp\to\CC$ a holomorphic function.
Then there is a uniquely determined holomorphic function
$Tf:V(a,b;\bR^{\bal})\times D(\bo;\bR)\to\CC$ such that $J((Tf)(t,.))\in\Delta(\bx^\bal)$
for any $t$ and
$(Tf)(\bx^\bal, \bx)=f(\bx)$. Moreover,
if there is a function $K:]0,\bR^{\bal}]\to\RR_+$
such that $\norm{f(\bx)}\leq K(\norm{\bx^\bal})$ for $\bx\in\pp$, then
$$\norm{(Tf)(t,\bx)}\leq \frac{\bR^{\a}}{\norm t} K(\norm t)
\prod_{j=1}^d\left(1-\frac{\norm{x_j}}{R_j}\right)^{-1}\mbox{ for }
t\in V(a,b;\bR^{\bal}),\  \bx\in D(\bo;\bR).$$
\end{lema}
\begin{nota} \benot \item Thus for $t\in V(a,b;\bR^{\bal})$, the mapping
$\bx\mapsto(Tf)(t,\bx)$ defines an element of any $\EE_r$, $0<r<\min_jR_j$.
This element will be denoted by $Tf(t)\mid_{\EE_r}$.
Clearly $T f\mid_{\EE_r}: V(a,b;\bR^{\bal})\rightarrow \EE_r$ is
holomorphic.
\item The estimate could be improved by multiplying with
$\left(1-\frac{\norm{\bx^\bal}}{\bR^{\bal}}\right)$ on the right. We omit this factor, as it has no
advantages in applications of the lemma.\ee
\end{nota}

\begin{proof}[Proof of Lemma \ref{T-sector}.]
If $f:\pp(a,b;\bR)\to\CC$ is a holomorphic function and $\pp(a',b';\bR')$ is some proper $\bx^\bal$-subsector of $\pp(a,b;\bR)$,
that is $a<a'<b'<b$ and $0<R_j'<R_j$ for $j=1,...,d$, then we can establish a function $K:]0,\bR'^\bal]\to\RR_+$
such that $\norm{f(\bx)}\leq K(\norm{\bx^\bal})$ for $\bx\in\pp(a',b';\bR')$. We can simply put
$$K(s)=\max\{\norm{f(\bx)}\mid\bx\in{\mathcal M}(s)\}\mbox{, where }{\mathcal M}(s)=\{\bx\in\cl(\pp(a',b';\bR'))\mid\norm{\bx^\bal}=s\},$$ 
because ${\mathcal M}(s)$ is a compact subset of $\pp(a,b;\bR)$. To see this observe that for $j=1,\ldots,d$ and
$\bx\in{\mathcal M}(s)$, $|x_j|^{\alpha_j}\geq s\prod_{k\neq j}{R'}_k^{-\alpha_k}$ and hence none of the $x_j$ can be too close to $0$. 

It is therefore sufficient to prove the lemma under the additional assumption that there exists a function $K:]0,\bR^{\bal}]\to\RR_+$
such that $\norm{f(\bx)}\leq K(\norm{\bx^\bal})$ for $\bx\in\pp$.
The uniqueness implies that we can define $Tf$ for a holomorphic function on $\pp(a,b;\bR)$ by combining
all the functions $T \tilde f$ obtained for the restrictions of $f$ to proper subsectors of $\pp(a,b;\bR)$. 

The Lemma had been proved in the case of a product of two variables in \cite{CDMS}.
We give a proof for the general statement.
Suppose first that $\bal=(1,1,\ldots ,1)$, i.e.\ the monomial is the product
$\bx^\bal=x_1\cdot\ldots \cdot x_d$.
Then we show the statement with the improved estimate
\begin{equation}\labl{ineq-prod}\norm{(Tf)(t,\bx)}\leq  K(\norm t)
\prod_{j=1}^d\left(1-\frac{\norm{x_j}}{R_j}\right)^{-1}\mbox{ for }
t\in V(a,b;\bR^{\bal}),\  \bx\in D(\bo;\bR).
\end{equation}

Observe that we can assume without loss of generality that the radii coincide: $R_j=R$
for $j=1,\ldots ,d$. Otherwise put $R=R_1$ and consider the function
$\tilde f(x_1,\ldots,x_d)=f(x_1,x_2\frac{R_2}{R},\ldots ,x_d\frac{R_d}{R})$ and
$\tilde K(t)=K(t\cdot R_2\cdot\ldots\cdot R_d/R^{d-1})$; the radii are all reduced to $R$ now.

We now proceed similarly to \cite{CDMS}, but have to treat Laurent series in several variables.
Put $g(t,z_2,\ldots,z_d):=f(\frac t{z_2\cdot \ldots \cdot z_d},z_2,\ldots,z_d)$. Then for fixed
$t\in V:=V(a,b;R^d)$, $g(\bz';t)$ (with notation as in \reff{notation})
is defined on the set of all $\bz'=(z_2,\ldots,z_d)\in\CC^{d-1}$
such that $\norm{\bz'}=\max(\norm{z_2},\ldots,\norm{z_d})<R$
and $\norm{z_2\cdot\ldots\cdot z_d}>\frac{\norm t}R$. Applying several times the
theorem on Laurent series expansions, we obtain that
$\dis g(\bz';t)=\sum_{\bm\in\ZZ^{d-1}}g_\bm(t){\bz'}^\bm$
with coefficients $g_\bm(t)$ holomorphic on $V$ and that
\begin{equation}\labl{majgm}\norm{g_\bm(t)}\leq K(\norm t)r_2^{-m_2}\cdot \ldots \cdot r_d^{-m_d}
\end{equation}
whenever $0<r_2,\ldots ,r_{d}<R$ are such that their product $r_2\cdots r_{d}>\frac{\norm t}R$.
In order to get good estimates for these coefficients, we have to choose the $r_j$ in an optimal
way.

In the case that one of the $m_j$ is negative, we choose $\ell$ such that the
minimum of  ${m_2,\ldots,m_d}$ is $m_\ell<0$ and rewrite (\reff{majgm}) as
$$\norm{g_\bm(t)}\leq K(\norm t)(r_2\cdot \ldots \cdot r_d)^{-m_\ell}r_2^{m_\ell-m_2}
\cdot \ldots \cdot r_d^{m_\ell-m_d}.$$
As the differences $m_\ell-m_j\leq0$ and one of them equals 0 in case $j=\ell$, we can choose $r_j=R$
if $j\neq\ell$ and $r_\ell$ arbitrary such that $r_2\cdot \ldots \cdot r_d>\frac{\norm t}{R}$.
Going over to the limit, we can as well assume that $r_\ell$ is chosen such that
$r_2\cdot \ldots \cdot r_d=\frac{\norm t}{R}$.
Introducing the notation $\mu(\bm)=m_\ell=\min(m_2,\ldots,m_d,0)$ and
$\norm\bm_1=m_2+\cdots +m_d$, we thus obtain in this case
\begin{equation}\labl{majgm-2}
\norm{g_\bm(t)}\leq K(\norm t)\norm t^{-\mu(\bm)}R^{d\mu(\bm)-\norm\bm_1}.
\end{equation}

In the case where all $m_j$ are nonnegative, we choose $r_j=R$ for all $j$ and obtain
$\norm{g_\bm(t)}\leq K(\norm t)R^{-\norm\bm_1}$. So, (\reff{majgm-2}) is valid for all $\bm\in\ZZ^{d-1}$.

Now we put $h_\bm(t):=t^{\mu(\bm)}g_\bm(t)$ and obtain
that $h_\bm$ are holomorphic on $V$ and
$\norm{h_\bm(t)}\leq K(\norm t)R^{d\mu(\bm)-\norm\bm}$ for $t\in V$ and $\bm\in\ZZ^{d-1}$.
It is convenient to introduce $\phi:\ZZ^{d-1}\to\NN^d$ by
$\phi(\bm)=(-\mu(\bm), m_2-\mu(\bm),\ldots,m_{d}-\mu(\bm))$. Observe that
$\phi$ is a bijection between $\ZZ^{d-1}$ and the set $\MM_d$ of all
$\bn=(n_1,\ldots ,n_d)\in\NN^d$ such that at least one
of the $n_j$ vanishes; moreover $\norm{\phi(\bm)}_1=\norm\bm_1-d\mu(\bm)$.
Now we define for $t\in V$
\begin{equation}\labl{defT-many}
(Tf)(t,\bx)=\sum_{\bm\in\ZZ^{d-1}}h_\bm(t)\bx^{\phi(\bm)}.
\end{equation}
As all $\phi(\bm)$ are in $\MM_d,$ we obtain $J((Tf)(t,.))\in\Delta(\bx^\bal)$.
Next, we have to show the convergence of the series if $\norm{x_j}<R$ for all
$j$.
Using $\phi(\ZZ^{d-1})=\MM_d$, we  estimate
$$\sum_{\bm\in\ZZ^{d-1}}\norm{h_\bm(t)}\norm{\bx^{\phi(\bm)}}\leq
K(\norm t)\sum_{\bn\in\MM_d}\left(\frac{\norm\bx}R\right)^{\norm\bn_1}\leq
K(\norm t)\prod_{j=1}^d\left(1-\frac{\norm{x_j}}{R}\right)^{-1}
$$
and thus the convergence of the series and the estimate of the Theorem. This also implies
that $Tf$ is analytic for $t\in V$, $\bx\in D(\bo;R)$.
The fact that $(Tf)(\bx^\bal,\bx)=f(\bx)$ follows easily from the construction%
$$\sum_{\bm\in\ZZ^{d-1}}h_\bm(\bx^\bal)\bx^{\phi(\bm)}=
\sum_{\bm\in\ZZ^{d-1}}g_\bm(\bx^\bal)\bx^{\mu(\bm)\bal+\phi(\bm)}=
\sum_{\bm\in\ZZ^{d-1}}g_\bm(\bx^\bal)\bx'^{\bm}=f(\bx).$$

We now reduce the general case to the one treated above.
Suppose that $\a_1>1$ and let $\xi=e^{2\pi i/\a_1}$.
Observe that $\bx\in\pp$ implies $\bx^{(k)}:=(\xi^k x_1,\bx')\in\pp$ for $k=0,\ldots,\a_1-1$
and therefore
there are uniquely determined functions $F_0,\ldots,F_{\a_1-1}$ defined on the sector
$a<\arg(zx_2^{\a_2}\cdot \ldots \cdot x_d^{\a_d})<b$, $0<\norm z <R_1^{\a_1},\ 0<\norm{ x_j} <R_j,$
for $j=2,\ldots,d$  in the monomial $zx_2^{\a_2}\cdot \ldots \cdot x_d^{\a_d}$
such that
$$f(\bx)=\sum_{j=0}^{\a_1-1}x_1^jF_j(x_1^{\a_1},\bx').$$
The functions $F_j$ can be determined by the Vandermonde system
$$f(\bx^{(k)})=\sum_{j=0}^{\a_1-1}\xi^{jk}x_1^jF_j(x_1^{\a_1},\bx'),\ k=0,\ldots,\a_1-1.$$
Hence,
$$x_1^jF_j(x_1^{\a_1},\bx')=\frac1{\a_1}\sum_{k=0}^{\a_1-1}\xi^{-jk}f(\bx^{(k)}),\
  j=0,\ldots,\a_1-1$$
and therefore $\norm{x_1^{\a_1}F_j(x_1^{\a_1},x_2,\ldots,x_d)}\leq
  R_1^{\a_1-j} K(\norm{\bx^\bal})$ for $j=0,\ldots,\a_1-1$, $\bx\in\pp$.
Continuing in this way we prove that
\begin{equation}\labl{Fbeta}
f(x)=\sum_{\bo\leq\bbeta<\bal}\bx^\bbeta F_\bbeta(x_1^{\a_1},\ldots,x_d^{\a_d}),
\end{equation}
where summation is over all integer vectors $\bbeta\in\ZZ^d$, $0\leq\beta_j<\a_j$ for
all $j$, and where the functions $F_\bbeta$ satisfy
$$\norm{\bx^{\bal}F_\bbeta(x_1^{\a_1},\ldots,x_d^{\a_d})}\leq
  \bR^{\bal-\bbeta}K(\norm{\bx^\bal}).$$
Now the situation is reduced to functions $F_\bbeta$, satisfying
$$\norm{F_\bbeta(u_1,\ldots,u_d)}\leq
\bR^{\bal-\bbeta}\frac{K(\norm{u_1\cdot\ldots\cdot u_d})}
  {\norm{u_1\cdot\ldots \cdot u_d}}$$
on a $(u_1\cdot\ldots \cdot u_d)$-sector $\tilde\pp$.
Using the first part of the proof, especially (\reff{ineq-prod}) for each $F_\bbeta$, and then combining
them using (\reff{Fbeta}) implies the statement. We just have to use the formula
$$\sum_{\bo\leq\bbeta<\bal}\bx^\bbeta \bR^{-\bbeta}
\prod_{j=1}^d\left(1-\frac{\norm{x_j^{\a_j}}}{R_j^{\a_j}}\right)^{-1}=
\prod_{j=1}^d\left(1-\frac{\norm{x_j}}{R_j}\right)^{-1}.
$$

The proof of the uniqueness can be given following the same steps as in the construction
of $Tf$. Details are left to the reader.
An alternative proof is given, in the context of asymptotic expansions with respect to a germ,
at the end of the proof of Theorem \reff{T-sector-germ}, at the end of Section \reff{proof33}.
\end{proof}

\begin{ejemplo}
The following example due to S.\ Kamimoto
shows that if one of the $\alpha_i>1$,
the estimate for $Tf$ cannot be as good as in (\reff{ineq-prod}) in the case of a ``simple'' product
$x_1\cdot\ldots\cdot x_d$.

Consider the monomial $\bx^\bal=x_1^2x_2$ and a small
$\bx^\bal$-sector $\Pi=\Pi(-\delta,\delta;R)$, $\delta,R>0$. Define
$f:\Pi\to\CC$ by the principal value $x_2^{1/2}$ if $\arg x_1$ and
$\arg x_2$ are both small and extend this function to all of $\Pi$
by analytical continuation. This is possible as for any path
$\gamma:[0,1]\to\Pi$, $\gamma(s)=(\gamma_1(s),\gamma_2(s))$, we must have
$\norm{2\arg \gamma_1(s)+\arg \gamma_2(s)} < \delta$. Hence if
we start with $\arg\gamma_j(s)\approx0$ and
$\gamma_2$ has made one tour of $x_2=0$ and thus reached
$\arg\gamma_2(s)\approx2\pi$, then we have $\arg\gamma_1(s)\approx-\pi$
and are far away from the starting point of the path.
After two tours of $\gamma_2$ around $x_2=0$, we have
$\arg\gamma_1(s)\approx2\pi$ and $\arg\gamma_2(s)\approx4\pi$ and are again
(with respect to the arguments) close to the starting points of the path.
The values of $f$ obtained by analytic continuation of $x_2^{1/2}$ are also
close to the original ones as $\arg x_2$ has been changed by about $4\pi$.

Thus we have an analytic function $f:\Pi\to\CC$ that
is bounded and satisfies $f(-x_1,x_2)=-f(x_1,x_2)$ for $(x_1,x_2)\in\Pi$.
The unique function $Tf$ of Lemma \ref{T-sector} is apparently
$Tf(t,(x_1,x_2))=t^{-1/2}x_1x_2$ (with the principal value of $t^{-1/2}$)
and this function is not bounded as $t\to0$.

In the above example, the $P$-sector is connected. A simpler example
where $\Pi$ has several connected components is given in the one variable
case by the monomial $x^2$. Consider the $x^2$-sector
$\Pi(-\delta,\delta,R)$ which has the two components
$\norm{\arg x}<\delta/2$ respectively $\norm{\arg x-\pi}<\delta/2$.
A bounded holomorphic function can be defined by having the value $1$
on one component and the value $-1$ on the other. The corresponding
function $Tf$ is apparently $Tf(t,x)=t^{-1/2}x$ and also unbounded
as $t\to0$.

\end{ejemplo}

Now we are in a position to define monomial asymptotics.
\begin{defprop}  \labl{defprop36}
Let $f$ be a  bounded holomorphic function on $\pp=\pp(a,b;R)$
and  $\hat f\in \hat\OO$. We
will say
that $f$ has $\hat f$ as asymptotic expansion at the origin in
$\bx^\bal$ if there exists $0<\tilde R\leq R$ such that $T\hat f(t)=
\sum_{n=0}^\infty g_nt^n\in \EE_{\tilde R}\lf t\rf $
and one of the following equivalent conditions is
satisfied:
\begin{enumerate}
\item For every $r\in]0,\tilde R[ $ one has
        $Tf(t)\mid_{\EE_r}\sim T\hat f(t)\mid_{\EE_r}$
as $V(a,b;r^n)\ni t\rightarrow 0$ in the sense of  (\reff{onevar}).
\item 
For  every 
$0< r<\tilde R$ and every
$N$, there exists $C(N,r)$ such that for all
$\bx\in\pp(a,b;r)$
$$
\norm{{{f}}(\bx )-\sum_{n=0}^{N-1}
g_n(\bx)\bx^{n\bal}}  \leq C(N,r)\cdot
\norm{\bx^{N\bal}} .
$$
\end{enumerate}
Analogously, we define the notion of
$s$-Gevrey asymptotic expansion if
$T\hat f$ is an $s$-Gevrey formal series (with coefficients in $\EE_{\tilde R}$) and
$Tf\sim_sT\hat f$ or,
equivalently, $C(N,r)$ can be chosen as $L(r)A(r)^{N}N!^s$.
\end{defprop}
\begin{proof} It suffices to prove that the second condition implies the first, the converse
is trivial.
For that purpose, consider the function $\delta(\bx)={{f}}(\bx )-\sum_{n=0}^{N}
g_n(\bx)\bx^{n\bal}$. We can apply Lemma \reff{T-sector}
with $K(u)=u^{N+1}$ and obtain for $0<r'<r<\tilde R$
$$\norm{(Tf)(t,\bx)-\sum_{n=0}^{N}g_n(\bx)t^n}\leq C(N+1,r) \norm{t}^N
\left(1-\frac {r'}r\right)^{-d}\ .$$

So, we have
\begin{equation*}
\begin{split}
\norm{(Tf)(t,\bx)-\sum_{n=0}^{N-1}g_n(\bx)t^n} & \leq \norm{(Tf)(t,\bx)-\sum_{n=0}^{N}g_n(\bx)t^n} + \norm{g_N(\bx) t^N}  \\ & \leq
C(N+1,r) \norm{t}^N
\left(1-\frac {r'}r\right)^{-d}+ \norma{g_N(\bx)}\norm{t}^N.
\end{split}
\end{equation*}

\end{proof}

\begin{nota} \labl{caracterGevreyseries}
Let us note that, in the Gevrey case, the series $T\hat{f}(t)$ automatically turns out to be $s$-Gevrey. In fact, from the inequalities
$$\norm{(Tf)(t,\bx)-\sum_{n=0}^{N}g_n(\bx)t^n}\leq L(r) A(r)^{N+1} (N+1)!^{s} \norm{t}^N
\left(1-\frac {r'}r\right)^{-d}$$
we obtain that
\begin{align*}
\norm{g_N (\bx )} & \leq L(r) A(r)^{N+1} (N+1)!^{s} \left(1-\frac {r'}r\right)^{-d} +L(r) A(r)^{N} N!^{s} \left(1-\frac {r'}r\right)^{-d} \frac{1}{\norm{t}} \\
 & =L(r) \left(1-\frac {r'}r\right)^{-d} A(r)^{N}N!^{s} \left[ A(r) (N+1)^{s} +\frac{1}{\norm{t}} \right] .
\end{align*}
Fixing $t$ with $\norm{t} = \frac{r}{2}$ yields Gevrey bounds for $g_N (\bx)$.
\end{nota}

In the rest of this Subsection, we recall the properties of Gevrey asymptotic expansions
in a monomial  from \cite{CDMS}, but state and prove them in the general setting -- whereas \cite{CDMS}
only consider the monomial $x_1x_2$. Since we have the main Lemma \reff{T-sector}
in the general setting, the generalisation is straightforward.

As in the single variable case, functions Gevrey asymptotic to 0 in a monomial
are exponentially small.
\begin{lema} \labl{expsmall}
If $f\in{\mathcal O}(\pp;E)$ has an $s$-Gevrey asymptotic
expansion in $\bx^\bal$
where $\hat f=0$, then, for all
sufficiently small $R'>0 $ there exist $C,\ B>0$ such that
on
${\tilde \pp}$
$$
\norm{f(\bx )}\leq C\cdot \exp\left( {-\frac{B}{\norm{\bx^\bal }^{1/s}}}\right)\
\mbox{ for }\bx\in\pp,\ \norm\bx<R'.
$$
\end{lema}

\begin{proof}
As in the classical case, we choose $N$ close to the optimal value
$(A|\bx^\bal|)^{-1/s}$ in the definition of an $s$-Gevrey asymptotic expansion
in $\bx^\bal$. Stirling's formula yields the statement.
\end{proof}

Using the first condition in the
definition of an $s$-Gevrey asymptotic expansion in $\bx^\bal$
and using Lemma \reff{T-sector} with $K(u)=\exp(-\gamma/u^{1/s})$,
the theorem of Ramis-Sibuya (theorem \reff{ramissibuya-classic}) (together with
Lemma \reff{borel-ritt-watson}, (1)) immediately implies
\begin{teorema}\labl{ramissibuya-monom}
Suppose that
the sectors $\pp_j= \pp(a_j,b_j;r), 1\leq j\leq m$ in $\bx^\bal$,
form a cover of $D(\bo;r)\setminus\{\bx;\ \bx^\bal=\bo\}$.
Given $f_j: \pp_j \rightarrow  E$ bounded and analytic, assume
that for every subsector $\pp'$ of  $\pp_{j_1}\cap \pp_{j_2}$
(provided that $\pp_{j_1}\cap \pp_{j_2}\neq \emptyset$)
there is a
constant $\gamma(\pp')>0$ such that
\begin{equation}\labl{expabove}
\norm{f_{j_1}(\bx)-f_{j_2}(\bx)} =
O(\exp(-\gamma(\pp')/|\bx^\bal|^{1/s}))
\end{equation}
for $\bx\in \pp'$.
Then the functions $f_j$ have asymptotic expansions in $\bx^\bal$
with a common right hand side and
the expansions are $s$-Gevrey.

Conversely, if a function $f:\pp\rightarrow E$ having an $s$-Gevrey asymptotic
expansion in $\bx^\bal$ is given, then a cover $\pp_j, 1\leq j\leq m$ and functions
$f_j: \pp_j \rightarrow E$ can be found that satisfy estimates
like (\reff{expabove})  and $f=f_1$.
\end{teorema}
As a consequence, Gevrey asymptotics in a monomial are compatible with the elementary
operations (sum, product,\ldots ). This is not obvious from the definition, except for addition.

Also, a Watson's lemma for Gevrey asymptotics in a monomial follows from
Lemma \reff{T-sector} and the one-variable version in Lemma \reff{borel-ritt-watson}, (2).
\begin{teorema}     \labl{Watson}
Let $\pp=\pp (a,b;R)$ be a sector in $\bx^\bal$ with $b-a>s\pi$
and suppose that $f\in {\mathcal O}(\pp ;E)$ has $\hat f=0$ as its
$s$-Gevrey asymptotic expansion. Then $f\equiv 0$.
\end{teorema}

\begin{defin}\labl{summxeps}
Let $s>0,k=1/s$ and a formal series
$\hat f(\bx)=\sum_{\bm\in\NN^d} a_{\bm}x^\bm$
be given.
\be\item
We say that {\em $\hat f$ is $k$-summable in $\bx^\bal$ on $\pp=\pp(a,b;R)$}
if $b-a>s\pi$
and there exists a holomorphic bounded function $f:\pp\rightarrow
E$ such that
$f$ has $\hat f$ as its $s$-Gevrey asymptotic expansion in
$\bx^\bal$ on $\pp$ in the sense of Definition/Proposition \reff{defprop36}. Then
$f$ is called the $k$-sum of $\hat f$ in $\bx^\bal$ on $\pp$. If it exists, it
is unique, by Theorem \reff{Watson}.
\item
The formal series $\hat f$ is called $k$-summable in $\bx^\bal$ in the direction
$\theta\in\RR$, if there exist $\delta,r>0$ such that $\hat{f}$
is
$k$-summable in $\bx^\bal$ on the sector
$\pp(\theta-s\frac\pi2-\delta,\theta+s\frac\pi2+\delta;r)$ in $\bx^\bal$.
\item
The formal series $\hat f$ is simply called $k$-summable, if it is
$k$-summable in every direction $\theta\in\RR$ with finitely many
exceptions mod $2\pi$ (called singular directions).\ee
\end{defin}

The first condition in Definition/Proposition \reff{defprop36}
shows that
$\hat f$ is $k$-summable in $\bx^\bal$ on $\pp(a,b;R)$ if and only
if the
formal series $T\hat f =\sum_{n=0}^\infty g_n t^n$ has coefficients in $\EE_r$
and if it is $k$-summable on $V(a,b;r^d)$ for $r>0$ sufficiently small
as series in one variable with coefficients in a Banach space. 
This allows us to carry over classical theorems to $k$-summability in a monomial.

It would be tempting to define summability in a monomial (and also
Gevrey asymptotics in a monomial) using only a fixed radius $r>0$,
but an example in \cite{CDMS} shows that $r$ might have to be chosen smaller and smaller
if the direction $\theta$ approaches a singular direction.

\section{Asymptotics with respect to an analytic germ}

Consider a germ of analytic function $P(\bx)\in \OO=\CC\{x_1,\ldots,x_d\}$, not a
unit (i.e.\ $P(\bo)=0$)  and not identically vanishing,
defined in some neighbourhood of $\bo\in \CC^d$, say in 
$D(\bo; \rho)$.
\begin{defin}
A sequence $\{ f_n \}_{n=0}^{\infty }$ in $\OO_b (D(\bo; \rho))$ is
an {\em asymptotic sequence (for $\hat f$)} if $J(f_n)$ converges in the $\mathfrak{m}$-adic
topology of $\hat\OO=\formal$ towards an element $\hat{f}\in \hat{\OO}$.

If, moreover, $J(f_n)\equiv \hat{f} \mod P^n\cdot\hat{\OO}$ for all $n$, then we
will say that $\{ f_n\}_n$ is a {\em $P$-asymptotic sequence (for $\hat f$)}.

If $\hat f\in\hat \OO$ is the limit of some $P$-asymptotic sequence, then we say that
{\em $\hat f$ is a $P$-asymptotic series}.
\end{defin}
\begin{defin} Given $a<b$, $0<R_j\leq +\infty $, $j=1,\ldots ,d$, $\bR=(R_1,\ldots ,R_d)$,
the $P$-sector $\Pi_P (a, b; \bR)$ is the set
$$ \Pi_P (a, b; \bR)= \{ \bx\in \CC^d;\ a<\arg P(x,y)<b,\
0<\norm{x_j}<R_j\mbox{ for }j=1,\ldots ,d\} .  $$
\end{defin}
By abuse of notation, we sometimes write $\Pi_P(a, b;R)$ for
$ \Pi_P (a, b; (R,R,\ldots ,R)$.

\begin{defin}  \labl{desarrollo}
Given a $P$-sector $\Pi$, $f\in \OO (\Pi)$ and $\hat f\in\formal$, we will say that
{\em $\hat{f}$ is the $P$-asymptotic expansion of $f$ on $\Pi$} if there exist $\rho>0$ and
a $P$-asymptotic sequence $\{ f_n \}_{n=1}^\infty $ in $\OO_b (D(\bo;\rho))$ for $\hat f$,
such that, for every $n\in \NN$, there exists $K_n>0$ such that
\begin{equation}\labl{Kn}
\norm{f(\bx)-f_n (\bx)}\leq K_n \cdot \norm{P(\bx)}^n
\end{equation}
on $D(\bo;\rho)\cap\Pi$. 
We will denote this by $f\sim_{\Pi}^P \hat{f}$. Observe that $\hat f$ is a $P$-sysmptotic series in this case.
\end{defin}

\begin{nota} \labl{nota-desarrollo}
\benot
\item In Theorem \reff{st-form}, we will show that the above definition is equivalent
to statements that reduce to Proposition/Definition \reff{defprop36} in
the case of a monomial.
\item \label{unidad}
If $U$ is a unit and $Q=U\cdot P$ then it is immediate to verify that $f$ has a series
$\hat f$ as a $P$-asymptotic expansion if and only if $f$ has
the same series as $Q$-asymptotic expansion.\medskip


\item \textbf{Specialization}. Consider a disk $D(\bo; \rho' )\in \CC^m$, and $Q: D(\bo; \rho' )\rightarrow \CC^d$ such that $Q(\bo )\in D(\bo; \rho )\subseteq \CC^d$, and $P\circ Q (\bo )=0$ but $P\circ Q\not\equiv 0$. Let $f\in \OO (\Pi_P (a,b;R))$ and consider a $P$-asymptotic sequence $\{ f_n\}_{n=1}^\infty$ for $f$. There exists $R'>0$ such that, if $\by \in \CC^m$ verifies $0<\norm{y_i}<R'$ for every $i$, $1\leq i\leq m$, then $Q(\by )\in D(\bo;R) \subseteq \CC^d$. Under these conditions, $f\circ Q$ is well defined on a $P\circ Q$-sector $\tilde{\Pi}= \Pi_{P\circ Q} (a,b; R')$ and has the sequence $\{f_n\circ Q\}_n$ as $P\circ Q$-asymptotic sequence.

This applies in particular when $Q(\bo)=\bo$.
Another interesting special case of this property can be given in the context of monomial asymptotic expansions, i.e.\ $P=\bx^\bal$, and $Q:\CC\rightarrow \CC^d$ is defined by $Q(x)=(x,t_2,\ldots , t_d)$, with $(t_2,\ldots ,t_d)\in \CC^{d-1}$, $0<\norm{t_i}<R$. We obtain that monomial asymptotic expansions can be specialized, fixing the values of some of the variables.


\item The notion of $P$-asymptotic expansion agrees with the
usual notion of asymptotic
expansion in one variable if $P=x$.
Indeed, suppose that $f$ is a holomorphic function defined on
a sector $V$, and that there is a family of holomorphic functions
$\{ f_n\}_n$, defined on a common neighbourhood of the origin
$D(\bo; \rho)$, and such that  there exists $C_n$ with
$$
\norm{f(x)-f_n(x)}\leq C_n \norm{x}^n
$$
on $V\cap D(\bo ; \rho )$. The sequence $\{ f_n\}_n$ turns out to
be an asymptotic sequence. Indeed, observe that
$$
\norm{f_n(x)-f_{n+1}(x)}\leq \norm{f_n(x)-f(x)}+
\norm{f(x)-f_{n+1}(x)}\leq (C_n+C_{n+1} \norm{x})\norm{x}^n,
$$
and therefore the meromorphic functions ${(f_n(x)-f_{n+1}(x))}/{x^n}$ are
bounded on $V\cap D(\bo;\rho )$, thus holomorphic at the origin.
Therefore we have $J_{n-1}(f_n)= J_{n-1}(f_{n+1})$ for all $n$ and $J(f_n)$
converges in the ${\mathfrak m}$-adic topology of $\CC \lf x\rf $
towards some series $\hat{f}$, such that $J_{n-1}(f_m)=
J_{n-1}(\hat{f})$ whenever $m\geq n$.

As we have $\norm{f_n(x)- J_{n-1}(f_n)(x)}\leq K_n
\norm{x}^n$ for every $n\in\NN$ with some $K_n$, we finally obtain
$$
\norm{f(x)- J_{n-1}(\hat f)(x)}\leq \norm{f(x)- f_n (x)}+
\norm{f_n(x)- J_{n-1}(f_n)(x)} \leq (C_n+K_n) \norm{x}^n,
$$
on $V\cap D(\bo;\rho')$.

The converse is trivial.
\ee
\end{nota}

\begin{lema} \labl{existfhat}
1.\ If a sequence $\{ f_n\}_n$ of functions on some polydisk $D(\bo;\rho)$
and a function $f$ on some $P$-sector satisfy the inequalities (\reff{Kn}),
then $\{ f_n\}_n$ is a $P$-asymptotic sequence.\\
2.\ The $P$-asymptotic expansion of a function $f$ on a $P$-sector, if it exists, is unique.

\end{lema}
\begin{proof}
For 1.: Such a sequence satisfies for all $n\in\NN$
\begin{equation*}
\begin{split}
\norm{f_n(\bx)-f_{n+1}(\bx)} & \leq \norm{f_n(\bx)-f(\bx)}+
\norm{f(\bx)-f_{n+1}(\bx)} \\ & \leq (K_n+K_{n+1} \norm{P(\bx)})
\norm{P(\bx)}^n\leq K'_n \norm{P(\bx)}^n
\end{split}
\end{equation*}
on the $P$-sector $\Pi$ mentioned in the statement. By Lemma \reff{lemadivision}, $P(\bx)^n$
divides $f_n(\bx)-f_{n+1}(\bx)$ for all $n$. As $P(\bo)=0$,
$\{f_n\}_{n\in\NN}$ is a Cauchy sequence for the $\mathfrak m$-adic topology and converges to some
$\hat f\in\hat\OO$. Moreover, $f_n\equiv\hat f \mod P^n\hat\OO$ for all $n$ and the
statement follows.

\noindent For 2.:  Let $\{ f_n\}_n$, $\{ \tilde{f}_n \}_n $  be two asymptotic sequences on a $P$-sector $\Pi$, such that a family of constants $C_n>0$ exists satisfying
\begin{align*}
\norm{f(\bx)-f_n(\bx)} & \leq C_n\cdot \norm{P(\bx)}^n, \\
\norm{f(\bx)-\tilde{f}_n(\bx)} & \leq C_n\cdot \norm{P(\bx)}^n.
\end{align*}
Then,
$$
\norm{f_n(\bx)- \tilde{f}_n (\bx)} \leq 2C_n \norm{P(\bx)}^n
$$
on $\Pi$, and by Lemma \reff{lemadivision}, $P(\bx)^n$ divides $f_n(\bx)-\tilde{f}_n (\bx)$. So, the families $\{ f_n \}_n$, $\{ \tilde{f}_n\} $ have the same limit in the ${\mathfrak m}$-topology.
\end{proof}

Let us see now that Definition \reff{desarrollo} is independent of
the chosen
$P$-asymptotic
sequence with limit $\hat f$. Assume that $\{ f_n \}_n$ in $\OO_b (D(\bo; \rho))$
is a $P$-asymptotic sequence, $f\in \OO (\Pi )$, $\Pi$ a $P$-sector,
such that for all $n\in\NN$
$$
\norm{f(\bx)-f_n (\bx)}\leq K_n\cdot \norm{P(\bx)}^n,
$$
for $\bx\in \Pi \cap D(\bo; \rho )$,
where $K_n>0$ are certain constants.

Let $ \{ \tilde{f_n}\}_n$ be another $P$-asymptotic sequence with $\tilde{f_n}\in \OO_b
(D(\bo; \tilde\rho))$ and such that $\{ f_n\}_n$ and $\{ \tilde{f}_n\}_n$
have the same limit in the ${\mathfrak m}$-adic topology.
Without loss of generality we may assume that $\tilde \rho=\rho$.

For any given $n\in\NN$, we have $J(f_n)\equiv
J(\tilde{f}_n)  \mod P^n\hat{\OO}$. Applying Lemma \reff{WDT}  for formal and 
convergent power series, it follows that actually $J(f_n)\equiv
J(\tilde{f}_n)  \mod P^n{\OO}$ for all $n$. Applying Lemma \reff{WDT-Ob}, it follows
that there
exists some positive $\rho'<\rho$  such that for every $n\in\NN$ we can write
$f_n - \tilde{f}_n = h_n P^n$ with some $h_n\in\OO_b(D(\bo;\rho'))$.

On ${D} (\bo; \rho')\cap \Pi$ we have
\begin{equation*}
\begin{split}
\norm{f(\bx) - \tilde{f}_n (\bx)} & \leq \norm{f(\bx)-f_n (\bx)}
+ \norm{h_n (\bx)}\cdot \norm{P(\bx)}^n \\
 & \leq (K_n+C_n)\cdot \norm{P(\bx)}^n,
\end{split}
\end{equation*}
where $C_n$ denotes some bound of $h_n$ on ${D} (\bo; \rho')$.
This proves that $\{\tilde f_n\}_{n}$ also satisfies the inequalities (\reff{Kn}) and
thus can be used to define $f\sim_\Pi^P\hat f$ on $\Pi$.\medskip

Contrary to monomial asymptotics, there is no canonical expansion (like in Definition/Proposition
\reff{defprop36}).
Using Generalised Weierstrass Division in the form of Lemma \reff{WDT-Ob}, we are going to present
standard expansions in an expression, but they cannot be called canonical, as they depend
on the choice of the linear form $\ell$ or equivalently on the choice of the
leading monomial of the analytic germ.

The only case where this expansion is canonical is precisely when this leading monomial does not depend
upon the linear form $\ell$, or, in geometric terms, when the Newton polyhedron of
$P(\bx )$ has only one vertex.
In this case, $P(\bx )= \bx^{\bal} U(\bx )$ with some unit $U(\bx )$, and
the Remark \reff{nota-desarrollo} (2) reduces the situation to the monomial case.

It is convenient to construct operators $T_\ell$ (for injective linear forms $\ell:\NN^d\to\RR_+$)
analogous to the operator $T$ used in monomial asymptotics.
First we restate the Corollaries \reff{coro-P-serie} and \reff{coro-P-germ}
in a slightly different way :
$\CC$ is replaced by an arbitrary $\CC$-vector space $E$.
For an injective linear form $\ell:\NN^d\to\RR_+$, $P\in\OO\setminus\{0\}$, $P(\bo)=0$ and a vector space $E$,
let $\Delta_\ell(P,E)$ denote the subset
of $E\lf \bx\rf $ defined analogously to (\reff{delta}). We abbreviate
$\Delta_\ell(P)=\Delta_\ell(P,\CC)$.

\begin{lema} \labl{defT-poly} Let $\ell:\NN^d\to\RR_+$ an injective linear form,
$P\in\OO\setminus\{0\}$, $P(\bo)=0$.
For any vector space $E$, there exists an isomorphism
$$T_\ell:E\lf \bx\rf \to\Delta_\ell(P,E)\lf t\rf $$
with the property $(T_\ell f)(P)=f$ for all series $f\in E\lf \bx\rf $. Here the symbol
$(T_\ell f)(P)$ means that $t$ is replaced by $J(P)$ in the series $T_\ell f$.
If $E$ is a normed vector space and
$f\in E\{\bx\}$ then $T_\ell f\in\EE_\ell\{t\}$, where $\EE_\ell=\Delta_\ell(P,E)\cap E\{\bx\}$.
\end{lema}
For $r>0$ let $\EE_{\ell,r}$ denote the Banach space of all functions $f\in\OO_b(D(\bo;r))$
the series expansion of which $J(f)\in\Delta_\ell(P)$. { If $r'<r$ there is a natural restriction map $\EE_{\ell,r}\rightarrow \EE_{\ell,r'}$, linear and continuous. The image of $f\in\EE_{\ell,r}$ will be denoted $f|_{\EE_{\ell,r'}}$. Similarly, if $f(t)=\sum_{n=0}^\infty f_nt^n\in\EE_{\ell,r} \lf t\rf $ is a formal series, $f(t)|_{\EE_{\ell,r'}}$ will represent $\sum_{n=0}^\infty f_n|_{\EE_{\ell,r'}} t^n$.
}

In the subsequent theorem, we establish an analogue of the operator $T_\ell$ for functions
defined on sectors in a germ. This theorem generalizes
Lemma \reff{T-sector} to arbitrary germs.
\begin{teorema}\labl{T-sector-germ} Let $\ell:\NN^d\to\RR_+$ an injective linear form,
$P\in\OO\setminus\{0\}$, $P(\bo)=0$.
Let $\pp=\pp_P(a,b;\bR)$ a sector in $P$.
Then there exists $\rho,\sigma,L>0$ with $P(D(\bo,\rho))\subset D(\bo,\sigma)$ and
the following properties:
\be\item If $f:\pp\to\CC$ is a holomorphic function 
on $\pp$, then there exists
a uniquely determined holomorphic function
$T_\ell f:V(a,b;\sigma)\times D(\bo;\rho)\to\CC$ such that
$J((T_\ell f)(t,.))\in\Delta_\ell(P)$
for any $t$ and
$(T_\ell f)(P(\bx), \bx)=f(\bx)$ for all $\bx\in\pp$, $\norm{\bx}<\rho$.
%
\item Moreover,
given a function $K:]0,S]\to\RR_+$, $S\geq\sup_{\bx\in \pp}\norm{P(\bx)}$,
such that $\norm{f(\bx)}\leq K(\norm{P(\bx)})$ for $\bx\in\pp$ we have
$$\norm{(T_\ell f)(t,\bx)}\leq \frac{L}{\norm t} K(\norm t)\mbox{ for }
t\in V(a,b,\sigma),\  \bx\in D(\bo;\rho).$$\ee
\end{teorema}
\noindent Theorem \reff{T-sector-germ} will be proved in the next section. 
\begin{nota}\labl{notaTl}\benot \item
It is important in some applications, that the numbers $\sigma,\rho,L$ are independent of the function
$f$ to which $T_\ell$ is to be applied.
\item
Unfortunately, $T_\ell f$ is in general defined on a small set only unlike
$Tf$ in Lemma \reff{T-sector} for monomial asymptotics. As in our theory of asymptotics
in a germ, the radius of the sectors or polydisks has to be reduced frequently, this
is not crucial.
The authors were surprised that such an operator $T_\ell$ for
asymptotics in a germ exists.
\item The unicity of $T_\ell f$ in statement (1) implies that the operator $T_\ell$ is independent
of the given $P$-sector in the following sense: If $a\leq a'<b'\leq b$, $f\in \OO_b(\Pi_P(a,b;\bR))$,
$F_1=T^{a,b}_\ell f:V(a,b,\sigma)\times D(\bo;\rho)$ is the function of statement (1) and
$F_2=T^{a',b'}_\ell f\mid_{\Pi_P(a',b';\bR)}:V(a',b',\sigma')\times D(\bo;\rho')$ is the function of statement (1)
for $f$ restricted to the $P$-subsector $\Pi_P(a',b';\bR)$, then
the restrictions of $F_1$ and $F_2$ to $V(a',b',\tilde\sigma)\times D(\bo;\tilde\rho)$,
$\tilde\sigma=\min(\sigma,\sigma')$, $\tilde\rho=\min(\rho,\rho')$ coincide.
This justifies our notation and will become important later.\ee
\end{nota}

The first crucial application of the above Theorem generalizes  Proposition/Definition 
\reff{defprop36} to asymptotics with respect to an analytic germ.

\begin{teorema}\labl{st-form} Let $\ell:\NN^d\to\RR_+$ be an injective linear form, $P\in\OO\setminus\{0\}$,
$P(\bo)=0$ and let $\Delta(P)$ be defined by (\reff{delta}).
Let $\Pi$ be a $P$-sector, $f\in\OO(\Pi)$ and $\hat f\in\hat\OO$.
Then $f$ has $\hat f$ as $P$-asymptotic expansion on $\Pi$ if and only if
there exists $\rho>0$ such that $T_\ell\hat f\in\OO_b(D(\bo;\rho))\lf t\rf $ and one of the following two
equivalent conditions holds:
\be
\item $T_\ell \hat f=\sum_{n=0}^\infty g_n  (\bx )t^n$ and for every $N$ there exists $L_N>0$ such that
$$\norm{f(\bx)-\sum_{n=0}^{N-1}g_n(\bx)P(\bx)^n}\leq
   L_N\norm{P(\bx)}^N\mbox{ for }\bx\in\Pi\cap D(\bo;\rho).$$
\item The function $T_\ell f$
from Theorem \reff{T-sector-germ} is defined on $V(a,b;\sigma)\times D(\bo;\rho)\to\CC$ for some
positive $\sigma$ and satisfies
$$T_\ell f\sim T_\ell \hat f\mbox{ as }V(a,b;\sigma)\ni t\to0.$$
\ee
\end{teorema}

It is worth noting separately that series that are $P$-asymptotic expansions, i.e.\ $P$-asymptotic series, 
cannot be arbitrary. The Theorem will be proved after the subsequent Corollary and several remarks.
\begin{coro}\labl{P-adic-clos}If $\hat f$ is a $P$-asymptotic series
then there exists $\rho >0$ such that
$T_\ell \hat f \in \OO_b(D_\rho)\lf t\rf $, i.e.\ if $\hat f$ is written
according to Corollary \reff{coro-P-serie}
$$\hat f=\sum_{n=0}^\infty g_nP^n,\ \ g_n\in\Delta_\ell(P),$$
then there exists $\rho>0$
such that for all $n\in\NN$, $g_n$ defines an element of $\OO_b(D(\bo;\rho))$.
\end{coro}
\begin{nota} \labl{nota-P-clos}
\benot
\item The converse is also true. Indeed, $f_n=\sum_{k=0}^n g_k P^k$ defines a $P$-asymptotic sequence
converging to $\hat f$.
\item The set of the above series is  a subset of the completion
of $\CC\{\bx\}$ with respect to the valuation defined by
the powers of the ideal $(P)$. Observe that their union over all $\rho>0$ does
not exhaust the completion: the latter also contains series
$\sum_{n=0}^\infty g_nP^n$, where the radii of convergence of the $g_n$
tend to 0.

\item In the case of a monomial $P=\bx^\a$, Theorem \reff{st-form} and
Corollary \reff{P-adic-clos} confirm that the ``new" Definition \reff{desarrollo}
of $\bx^\a$-asymptotic expansions is equivalent to the ``classical" Definition
\reff{defprop36} from \cite{CDMS}.

\item If $f\in \OO_b(D(\bo;\rho ))$, and $\hat{f}$ is the Taylor expansion of $f$ at the origin, $f$ has $\hat{f}$ as $P$-asymptotic expansion, as $f$ can be written in powers of $P$ by Corollary \reff{coro-P-germ} of the Division Theorem \reff{WDT}.

\item With $\EE=\Delta_{\ell} (P) \cap \con$, the set  $\EE \lf t\rf $ is not an algebra, as it is not closed under multiplication. Nevertheless, from Definition \reff{desarrollo} it can be seen that the product of functions having a $P$-asymptotic expansion also has a $P$-asymptotic expansion. 
Indeed, consider functions $f,g$ on some $P$-asymptotic sector and $P$-asymptotic sequences
$\{f_n\}_{n},\,\{g_n\}_n$ satisfying (\reff{Kn}) 
corresponding to $P$-asymptotic expansions of $f,g$. Then we can write
    $$
    f(\bx) g(\bx) -f_n(\bx) g_n(\bx) = f(\bx) (g(\bx) -g_n(\bx))+ (f(\bx) -f_n(\bx ))g_n(\bx).
    $$
    So, $\{f_n(\bx) g_n (\bx )\}_n$ is a $P$-asymptotic sequence, and it converges to $\hat{f} (\bx ) \hat{g} (\bx)$. In fact, if $(T_{\ell}\hat{f} )(t,\bx ) =\sum_{n=0}^\infty a_n(\bx ) t^n$, $(T_{\ell} \hat{g} ) (t,\bx ) = \sum_{n=0}^{\infty} b_n (\bx ) t^n$, and decompose
    $$
    \sum_{k=0}^n a_k(\bx ) b_{n-k} (\bx ) =\sum_{m=0}^\infty h_{nm} (\bx ) P(\bx )^m,
    $$
    with $h_{nm}\in\EE$, we have
    $$
    T_{\ell} (\hat{f}\cdot \hat{g} ) (t,\bx )= \sum_{n=0}^\infty \left( \sum_{m=0}^n h_{n-m,m} (\bx ) \right) \cdot t^n.
    $$

\item It is not evident from Definition  \reff{desarrollo} and the characterization given in Theorem \reff{st-form} that the set of functions having a $P$-asymptotic expansion is stable by partial derivatives. Let $\Pi$ be a $P$-sector, $f\in \OO (\Pi) $ having $\hat{f}\in \hat{\OO}$ as a $P$-asymptotic expansion. Using the notation of Theorem \reff{st-form} (2), $T_{\ell} f\sim T_{\ell }{\hat{f}}$. From the equality $f(\bx )= T_{\ell} f (P(\bx), \bx )$, we deduce that
    $$
    \frac{\partial f}{\partial x_i}(\bx )= \frac{\partial P}{\partial x_i}(\bx ) \cdot \frac{\partial (T_{\ell}f)}{\partial t} (P(\bx ), \bx )+\frac{\partial (T_{\ell} f)}{\partial x_i} (P(\bx ), \bx ).
    $$
    As $\dfrac{\partial (T_{\ell} f)}{\partial t}(t,\bx )$ and $\dfrac{\partial (T_{\ell }f) }{\partial x_i}(t, \bx )$ have asymptotic expansion with respect to $t$, by Cauchy's formula, considerations about products made in (3) and (4) imply that $\dfrac{\partial f}{\partial x_i} (\bx )$ has a $P$-asymptotic expansion. Moreover, if we write
    $$
    (T_{\ell} \hat{f} ) (t,\bx ) = \sum_{n=0}^\infty f_n (\bx ) t^n,
    $$
    and expand
    $$
    \frac{\partial P}{\partial x_i} (\bx )\cdot f_n (\bx )=\sum_{m=0}^\infty g_{nm}(\bx ) P(\bx )^m,
    $$
    with $g_{nm} (\bx )\in \Delta_{\ell} (P)$, then a straightforward computation shows that
    $$
    T_{\ell} \left( \frac{\partial f}{\partial x_i} \right) (t,\bx )= \sum_{n=0}^\infty \left( \frac{\partial f_n}{\partial x_i} (\bx ) + \sum_{k=1}^{n+1} kg_{k,n-k+1} (\bx ) \right) t^n.
    $$
    Observe that $\dfrac{\partial f_n}{\partial x_i} (\bx )\in \Delta_{\ell}(P)$.

\ee
\end{nota}

\begin{proof}[Proof of Theorem \reff{st-form}]:
Assume that $\{f_n\}_{n\in\NN}$ is a $P$-asymptotic sequence defined on $D(\bo;R)$
for some positive $R$ satisfying the inequalities (\reff{Kn}) of Definition
\reff{desarrollo} with the constants $K_n$:
$$ \norm{f(\bx)-f_n (\bx)}\leq K_n \cdot \norm{P(\bx)}^n$$
for $\bx\in D(\bo;R)\cap\Pi$, $\Pi$ some $P$-sector, and such that $J(f_n) \equiv \hat{f} \mod P^n\cdot \hat{\OO}$.

According to Lemma \reff{WDT-Ob}, we can choose $\mu >0$ such that
$D_\mu \subset D(\bo;R)$ for the set $D_\mu$ of Lemma \reff{WDT-Ob}; let $Q,R$
denote the operators on $\OO_b(D_\mu )$ introduced there.
Then we can write for all $m\in\NN$
$$f_n (\bx )=\sum_{\nu=0}^{m-1} RQ^\nu(f_n)P(\bx )^{\nu}+Q^m(f_n)P(\bx )^m.$$
As in the proof of Lemma \reff{existfhat}, we find that
$f_n\equiv f_{n+1}\mod P^n\OO_b(D_\mu)$
and hence $RQ^\nu (f_n) =RQ^\nu (f_m)$, if $\nu<n\leq m$.
So, define $g_n:=RQ^n(f_{n+1})\in \OO_b (D_\mu)$. We have $J(g_n)\in \Delta_\ell (P)$ and
$g_n=RQ^n(f_{m})$ for all $m>n$.

Then for all $n$,
\begin{equation*}
\begin{split}
\norm{f_n(\bx)-\sum_{\nu=0}^{n-1} g_\nu(\bx) P^{\nu}(\bx)} = \norm{f_n(\bx)-\sum_{\nu=0}^{n-1} RQ^\nu (f_{\nu+1}) P^{\nu}(\bx)} \\
 = \norm{f_n(\bx)-\sum_{\nu=0}^{n-1} RQ^\nu (f_{n}) P^{\nu}(\bx)} = \norm{Q^n (f_n)P^n (\bx )} \leq M_n\cdot \norm{P(\bx )}^n.
 \end{split}
 \end{equation*}
for $\bx \in D_\mu$ with some constant $M_n$.
This first implies that $\hat f(\bx)=\sum_{n=0}^\infty g_n(\bx) P(\bx)^n$ and hence
$T_\ell \hat f\in\OO_b(D_\mu)\lf t\rf $.
Together with (\reff{Kn}), this yields
\begin{equation}\labl{st-esti}
\norm{f(\bx)-\sum_{\nu=0}^{n-1} g_\nu(\bx) P^{\nu}(\bx)}\leq(K_n+ M_n)\norm{P(\bx)}^n
\end{equation}
for $n\in\NN$ and $\bx\in\Pi$. Thus we have proved (1).
Application of Theorem \reff{T-sector-germ} to (\reff{st-esti}) with $K(s)=(K_n+M_n) s^n$
yields the existence of some positive $\sigma,\tilde\rho\leq\rho $ and $L$ such that
$$\norm{(T_\ell f)(t,\bx)-\sum_{\nu=0}^{n-1}g_\nu(\bx) t^\nu}\leq L (K_n+M_n)\norm t^{n-1}$$
for $(t,\bx)\in V(a,b;\sigma)\times D(\bo;\tilde\rho)$. This proves (2).

The proof of the converses is trivial.
\end{proof}

Corollary \reff{P-adic-clos} raises the question, whether all formal
series $\sum g_nP^n$, the coefficients $g_n\in\OO_b(D(\bo;\rho))$ of which satisfy
$J(g_n)\in\Delta_\ell(P)$
can be attained as $P$-asymptotic expansions of some function $f$ on an arbitrary $P$-sector.
Using the classical Borel-Ritt Theorem \reff{borel-ritt-watson} (1), it follows
easily that this ``Borel-Ritt Theorem for asymptotics in a germ" is valid.
Details are left to the reader.

\section{Proof of Theorem \protect{\ref{T-sector-germ}}}\labl{proof33}

The main problem is to find {\em any} function $F$ analytic on $V(a,b;\sigma)\times D(\bo;\rho)$
satisfying $F(P(\bx),\bx)=f(\bx)$ for small $\bx$ in $\pp_P(a,b;\bR)$ because subsequently
Corollary \reff{coro-P-germ} can be applied to $F(t,.)$.
In the construction of such a function $F$ using induction on $h(P)$, we need to
study functions $F$ satisfying $F(P(\bx),\bx)=0$. This will be done in
the two subsequent lemmas.\newcommand{\D}{{\mathcal D}}

In this section, we fix a linear form $\ell$ and a germ $P$ as in the hypothesis of the Theorem
and suppose that $P\in\OO(D(\bo;R))$ for some $R>0$. We begin with a simple observation.
\begin{lema}\labl{division}Let $a,b,r>0$ and $\D\subset D(\bo;\bR)$
some domain. For every analytic $F:V(a,b;r)\times\D\to\CC$ satisfying $F(P(\bx),\bx)=0$
for $\bx\in\D$ with $P(\bx)\in V(a,b;r)$, there exist a unique analytic function
$H:V(a,b;r)\times\D\to\CC$ such that
\begin{equation}\labl{t-P}
F(t,\bx)=(t-P(\bx))H(t,\bx)\mbox{ for all }t\in V(a,b;r),\bx\in\D.
\end{equation}
\end{lema}
\begin{proof} $H$ is determined by $H(t,\bx)=F(t,\bx)/(t-P(\bx))$ on the set of
$(t,\bx)$ with $t\neq P(\bx)$. If there is no $\bx\in\D$ with $P(\bx)\in V(a,b;r)$
then $H$ is obviously analytic on $V(a,b;r)\times\D$.

If there exists $\bx\in\D$ such that $P(\bx)\in V(a,b;r)$,
then the hypothesis implies that $\lim_{t\to P(\bx)}H(t,\bx)=
\frac{\partial F}{\partial t}(P(\bx),\bx)$ exists. Using Riemann's theorem on removable
singularities, this shows that, for any fixed
$\bx$, the function $t\mapsto H(t,\bx)$ can be analytically continued to a function holomorphic
on $V(a,b;r)$.

The simplest way to establish analyticity of this continuation with respect to
$(t,\bx)$ in the neighborhood of some ``critical '' point of the form
$(P(\bx_0),\bx_0)\in V(a,b;r)\times\D$ is to write
$$H(t,\bx)=\int_0^1 \frac{\partial F}{\partial t}(\tau t+(1-\tau) P(\bx),\bx)\,d\tau$$
for all $(t,\bx)$ in its neighborhood.
\end{proof}
\begin{lema}\labl{division2}
Let $a,b,r>0$ and $\D,\D'\subset\CC^d$
two domains such that the closure of $\D'$ is compact and contained in $\D$.
Then there exists $L>0$ with the following property:
For every analytic $F:V(a,b;r)\times\D\to\CC$ satisfying $F(P(\bx),\bx)=0$
for $\bx\in\D$ with $P(\bx)\in V(a,b;r)$
and $\sup_{\bx\in\D}\norm{F(t,\bx)}\leq K(t)$ with some $K:V(a,b;r)\to\RR_+$,
the unique analytic function
$H:V(a,b;r)\times\D\to\CC$ of Lemma \reff{division} with (\reff{t-P})
satisfies
$$\sup_{\bx\in\D'}\norm{H(t,\bx)}\leq L\,K(t)\mbox{ for }t\in V(a,b;r).$$
\end{lema}
\begin{nota}Lemma \reff{division} implies that $t-P(\bx)$ divides $F(t,\bx)$. Therefore
Lemma \reff{lemadivision2} could be applied to these functions of $(t,\bx)$. Unfortunately this
does not yield the desired result as we would have estimates for $H$ only on
$(t,\bx)$-subsets compactly contained in $V(a,b,r)\times\D$ which cannot have points with
$t=0$ on their boundary.\end{nota}
\begin{proof}It is close to that of Lemma \reff{lemadivision2}, but use of
Lemma \reff{WDT-Ob} in this special situation improves the domains of validity
of the estimates.

By a classical argument of compacity, it is sufficient to prove that for
every $t_0$ in the closure of $V(a,b,r)$ and every $\bx_0\in\D$, there exist $\delta,L>0$
and a neighborhood $\U\subset\D$ of $\bx_0$ such that for every function $F:V(a,b;r)\times\D\to\CC$
fufilling the hypothesis of the theorem with some majorant $K$, the quotient
$H$ from the previous Lemma satisfies
$$\sup_{\bx\in \U}\norm{H(t,\bx)}\leq L\,K(t)\mbox{ for }t\in V(a,b;r)\cap D(t_0,\delta).$$
For the proof of this statement, we have to distinguish two cases.

If $P(\bx_0)\neq t_0$, then $\norm{t-P(\bx)}$ is bounded below by some positive constant
if $t$ is sufficiently close to $t_0$ and $\bx$ sufficiently close to $\bx_0$. In this case,
the existence of $\delta,L,\U$ is immediate.

If $P(\bx_0) = t_0$, then we apply the results of subsection 2.3 to
$\tilde P(\bx)=P(\bx_0+\bx)-t_0$. We choose some injective linear form $\tilde\ell$ and define
$\Delta_{\tilde\ell}(\tilde P)$ accordingly (see (\reff{delta})). We choose some neighborhood
$\tilde \U$ of $\bo$ such that Lemma \reff{WDT-Ob} can be applied. This yields bounded linear
operators $\tilde Q,\tilde R:\OO_b(\U)\to\OO_b(\U)$, $\U:=\bx_0+\tilde\U$, such that
for all functions $g,q,r\in\OO_b(\U)$, we have
$g=(P-t_0)q+r$, $J_{\bx_0}(r)\in\Delta_{\tilde\ell}(\tilde P)$ if and only if
$q=\tilde Q(g)$ and $r=\tilde R(g)$. Here $J_{\bx_0}(r)$ denotes the Taylor expansion of
the function $\bx\to r(\bx_0+\bx)$.

Equation (\reff{t-P}) is equivalent to
$$(t-t_0)H(t,\bx)-F(t,\bx)=(P(\bx)-t_0)H(t,\bx)+0$$
for all $(t,\bx)$. Thus for $t\in V(a,b,r)$, the functions $h,f:\U\to\CC$ defined by
$h(\bx)=H(t,\bx)$, $f(\bx)=F(t,\bx)$ for $\bx\in\U$ satisfy
\begin{equation}\labl{star}h=(t-t_0)\tilde Q(h)-\tilde Q(f)\ .
\end{equation}
with the above operator $\tilde Q$ on $\OO_b(\U)$.
If $\delta>0$ is sufficiently small, the fixed point principle can be applied
to (\reff{star}) if $\norm{t-t_0}<\delta$ and
yields that $\norma{h}\leq\frac{\norma{\tilde Q}}{1-\delta\norma{\tilde Q}}\norma{f}$, where $\norma{\cdot}$ denotes the maximum norm.
This yields
$$\sup_{\bx\in \U}\norm{H(t,\bx)}\leq\frac{\norma{\tilde Q}}{1-\delta\norma{\tilde Q}}
  \sup_{\bx\in \U}\norm{F(t,\bx)}$$
if $t\in V(a,b,r)$, $\norm{t-t_0}<\delta$. Hence we can choose
the above $\delta$, $L=\frac{\norma{\tilde Q}}{1-\delta\norma{\tilde Q}}$ and the above
neighborhood $\U$ of $\bx_0$ to obtain the wanted statement.
This completes the proof.
\end{proof}

The main step is
\begin{lema}\labl{main-step}
Let $P\in\OO\setminus\{0\}$, $P(\bo)=0$ and let $\pp=\pp_P(a,b;\bR)$ a sector in $P$.
Then there exist
$\rho,\sigma,L>0$ with $P(D(\bo;\rho))\subset D(\bo;\sigma)$ and the following properties:
\be\item If $f:\pp\to\CC$ is a holomorphic function,
then there exists a holomorphic function
$F:V(a,b;\sigma)\times D(\bo;\rho)\to\CC$ such that
$F(P(\bx), \bx)=f(\bx)$ for all $\bx\in\pp$, $\norm{\bx}<\rho$.
\item Moreover,
given a function $K:]0,S]\to\RR_+$, $S\geq\sup_{\bx\in \pp}\norm{P(\bx)}$,
such that $\norm{f(\bx)}\leq K(\norm{P(\bx)})$ for $\bx\in\pp$, the function $F$ of statement (1)
satisfies
$$\norm{F(t,\bx)}\leq \frac{L}{\norm t} K(\norm t)\mbox{ for }
t\in V(a,b;\sigma),\  \bx\in D(\bo;\rho).$$\ee
\end{lema}
\begin{proof} 
%
%
The statements can be formally combined if we allow a function $K$ with $K(s)\equiv\infty$.
Thus we prove both statements together by induction on $h(P)$ using Lemma \reff{reduction}. 

If $h(P)=0$ then $P$ has normal
crossings  and the statement can be reduced to the monomial version, where
Lemma \reff{T-sector} even gives a better result.
Assume now that the statement is true whenever $h(Q)\leq m$  and prove it if $h(P)=m+1$.
As the statement does not change by right composition of $P$ with a diffeomorphism,
we can assume that $h(P\circ b_\xi)\leq m$ for all $\xi\in \PC$
or $h(P\circ r_k)\leq m$ for some $k\in\NN$. See Subsection \reff{normal} for notation.

We first assume that the statement is true for all $P\circ b_\xi$, $\xi\in\PC$.
Then for every $\xi\in\PC$ and every sector $\pp_\xi:=\pp_{P\circ b_\xi}(a,b;\tilde\bR)$,
there exist $\rho_\xi,\sigma_\xi, L_\xi>0$ such that
for every holomorphic $\tilde f:\pp_\xi\to\CC$ and $K:]0,\sigma_\xi]\to\RR_+\cup\{\infty\}$ with
$\norm{\tilde f(\bz)}\leq K(\norm{(P\circ b_\xi)(\bz)})$  for
$\bz\in\pp_{\xi}$
there exists a holomorphic $\tilde F:V(a,b;\sigma_\xi)\times D(\bo;\rho_\xi)\to\CC$
with
$$\tilde f(\bz)=\tilde F((P\circ b_\xi)(\bz),\bz)\mbox{ for }\bz\in\pp_\xi\cap D(\bo,\rho_\xi)$$
and $\norm{\tilde F(t,\bz)}\leq\frac{L_\xi}{\norm t}K(\norm t)$ for
$(t,\bz)\in V(a,b;\sigma_\xi)\times D(\bo;\rho_\xi)$.

Given some analytic $f:\pp\to\CC$ and $K:]0,S]\to\RR_+\cup\{\infty\}$ with
$\norm{f(\bx)}\leq K(\norm{P(\bx)})$ for $\bx\in\pp$,
let $F_\xi$ denote the holomorphic  function on $V(a,b;\sigma_\xi)\times D(\bo;\rho_\xi)$
corresponding to $\tilde f=f\circ b_\xi$.
As before we use $\phi_\xi$ to carry over these statements to neigborhoods of points of
the exceptional divisor.
So define $G_\xi$ on $V(a,b;\sigma_\xi)\times U_\xi$, $U_\xi=\phi_\xi^{-1}(D(\bo,\rho_\xi))$
such that $G_\xi(t,p)=F_\xi(t,\phi_\xi(p))$. By construction, we have
$$G_\xi((P\circ b)(p),p)=(f\circ b)(p)\mbox{ for }p\in U_\xi, \phi_\xi(p)\in \pp_\xi$$
and $\norm{G_\xi(t,p)}\leq \frac{L_\xi}{\norm t}K(\norm t)$ for
$(t,p)\in V(a,b,\sigma_\xi)\times U_\xi$.
The differences $D_{\xi\eta}(t,p):=G_\xi(t,p)-G_\eta(t,p)$ are then defined and holomorphic
for $t\in V(a,b;\sigma_{\xi\eta})$, $\sigma_{\xi\eta}:=\min(\sigma_\xi,\sigma_\eta)$,
and $p\in U_\xi\cap U_\eta$. They satisfy
$$D_{\xi\eta}((P\circ b)(p),p)=0\mbox{ for small }p\in U_\xi\cap U_\eta,\arg(P(b(p)))\in]a,b[$$
and $\norm{D_{\xi\eta}(t,p)}\leq \frac{L_\xi+L_\eta}{\norm t}K(\norm t)$ for
$t\in V(a,b;\sigma_{\xi\eta}))$, $p\in U_\xi\cap U_\eta$.

In order to apply Lemma \reff{division2} (resp.\ Lemma \reff{division} in the case $K(s)\equiv\infty$), we also consider
$\hat U_\xi=\phi_\xi^{-1}(D(0;\hat r_\xi))$ with some positive
$\hat r_\xi<r_\xi$.
Then this Lemma, applied to $D_{\xi\eta}$ -- more precisely to their right
composition with $\phi_\xi$ -- yields holomorphic functions
$Q_{\xi\eta}:V(a,b;\sigma_{\xi\eta})\times (\hat U_\xi\cap\hat U_\eta) \to\CC$
satisfying
$$D_{\xi\eta}(t,p)=(t-(P\circ b)(p))Q_{\xi\eta}(t,p)$$
and
$\norm{Q_{\xi\eta}(t,p)}\leq {C_{\xi\eta}(L_\xi+L_\eta)}\frac1{\norm t}K(\norm t)$
on the domain of $Q_{\xi\eta}$ with some constant $C_{\xi\eta}$ depending only upon
$a,b,\sigma_{\xi\eta},U_\xi\cap U_\eta$ and $\hat U_\xi\cap \hat U_\eta$.

As the $\hat U_\xi$, $\xi\in\PC$ cover the exceptional divisor
in the blow-up variety $M$, there exists a finite subcover, say corresponding to
$\xi_j$, $j=0,\ldots ,N$. We now apply Lemma \reff{cover-E} and Remark \reff{rem-cover} to the collection
$(Q_{\xi_j\xi_k})_{j,k=0,\ldots ,N}$ of holomorphic functions
$Q_{\xi_j\xi_k}: V(a,b;\tilde\sigma)\times(\hat U_{\xi_j}\cap \hat U_{\xi_k})\to\CC$, $j,k=0,\ldots ,N$,
$\tilde\sigma$ the minimum of $\sigma_{\xi_j}$, $j=0,\ldots ,N$.
We obtain a collection of holomorphic functions
$R_{\xi_j}:V(a,b;\tilde\sigma)\times \tilde U_{\xi_j}\to\CC$,  satisfying
$$Q_{\xi_i,\xi_j}(t,p)=R_{\xi_i}(t,p)-R_{\xi_j}(t,p)\mbox{ for }(t,p)\in
    V(a,b;\tilde\sigma)\times (\tilde U_{\xi_i}\cap \tilde U_{\xi_j})$$
and $\sup_{p\in\tilde U_{\xi_j}}\norm{R_{\xi_j}(t,p)}\leq \frac{\tilde C}{\norm t}K(\norm t)$
for $t\in V(a,b,\tilde\sigma)$.
Here $\tilde U_{\xi_j}\subset M$ are the open subsets of $U_{\xi_j}$ in Lemma \reff{cover-E} 
covering $E$ and with the constant $C$ from
Lemma \reff{cover-E}, the constant $\tilde C$ is the maximum of
$C\,C_{\xi_j\xi_k}(L_{\xi_j}+L_{\xi_k})$, $j,k=0,\ldots,N$. 

Now we can define holomorphic functions
$\tilde G_{\xi_j}:V(a,b;\tilde\sigma)\times \tilde U_{\xi_j}\to\CC$,
$j=0,\ldots,N$ by
$$\tilde G_{\xi_j}(t,p)=G_{\xi_j}(t,p)-(t-(P\circ b)(p))R_{\xi_j}(t,p).$$
By the construction of $R_{\xi_j}$, the family $\tilde G_{\xi_j}$ glues
together, i.e.\ $\tilde G_{\xi_i}(t,p)=\tilde G_{\xi_j}(t,p)$ whenever
$p\in \tilde U_{\xi_i}\cap \tilde U_{\xi_j}$.
As at the end of the proof of Lemma \reff{lemadivision},
this implies that there exists a positive $\bar\rho$ and a holomorphic function
$F:V(a,b;\tilde\sigma)\times D(\bo,\bar\rho)\to\CC$ such that
$F(t,b(p))=\tilde G_{\xi_j}(t,p)$ for $t\in V(a,b;\tilde\sigma)$ and
$p\in \tilde U_{\xi_j}$ with $b(p)\in  D(\bo,\bar\rho)$.
We can assume without loss in generality that $P(D(\bo;\bar\rho))\subset D(\bo;\tilde\sigma)$.
An easy calculation shows that $F(P(\bx),\bx)=f(\bx)$ for $\bx\in\pp_P(a,b;\bar\rho)$.

By their definition, we have $\sup_{p\in \tilde U_{\xi_j}}
\norm{\tilde G_{\xi_j}(t,p)}\leq \left(L_{\xi_j}+2\tilde\sigma\tilde C\right)\frac1{\norm{t}}K(\norm t)$
for $j=0,\ldots ,N$ and hence
$$\sup_{\bx\in D(0;\bar\rho)}\norm{F(t,\bx)}\leq
\left(\max_{j=0,\ldots ,N}L_{\xi_j}+2\tilde\sigma\tilde C\right)\frac1{\norm{t}}K(\norm t)$$
for $t\in V(a,b;\tilde\sigma)$.
This completes the proof of the lemma in the case of blow-ups.

The case of a ramification is much simpler and left to the reader.
\end{proof}

Now we are in a position to prove Theorem \reff{T-sector-germ}, combining the two statements as in the above
proof of Lemma \reff{main-step}.
This Lemma provides positive $\tilde\rho,\tilde\sigma,\tilde L$ and,
for given holomorphic $f:\pp\to\CC$ with an estimate $\norm{f(\bx)}\leq K(\norm{P(\bx)})$ on $\pp$, it yields
a holomorphic function $F:V(a,b;\tilde\sigma)\times D(\bo;\tilde\rho)\to\CC$
satisfying $F(P(\bx),\bx)=f(\bx)$ for $\bx\in D(\bo;\tilde\rho)$
and $\norm{F(t,\bx)}\leq \frac {\tilde L}{\norm t}K(\norm t)$ on $V(a,b;\tilde\sigma)\times D(\bo;\tilde\rho)$.
In order to apply Lemma \reff{WDT-Ob}, we restrict $F$ to
$V(a,b;\tilde\sigma)\times D_s$, where $D_s\subset D(\bo;\tilde\rho)$ is chosen such that the operators
$Q,R$  of Lemma \reff{WDT-Ob} are defined.

As in the proof of Corollary \reff{coro-P-germ}, we can write
$$F(t,\bx)=\sum_{n=0}^{N-1}((RQ^n)(F(t,.)))(\bx)P(\bx)^n+(Q^NF(t,.))(\bx)P(\bx)^N$$
for $t\in V(a,b;\tilde\sigma)$, $N\in\NN$ and $\bx\in D_s$ by repeated application of Lemma
\reff{WDT-Ob}.
If $\rho>0$ is so small that $D(\bo;\rho)\subset D_s$ and
$B=\sup\{\norm{P(\bx)};\  \bx\in D(\bo;\rho)\}  <\frac1{\norma{Q}}$
then $\norm{(Q^NF(t,.))(\bx)P(\bx)^N}\leq(B\norma Q)^N\norma{F(t,.)}\to0$
as $N\to\infty$ for $\bx\in D(\bo,\rho)$ and we obtain
$F(t,\bx)=\sum_{n=0}^{\infty}((RQ^n)(F(t,.)))(\bx)P(\bx)^n$
for these $\bx$.
Now we define the desired function\footnote{
The function $T_\ell f$ is independent of the choice of $F$ in Lemma \reff{main-step}.
This follows from the uniqueness established at the end of this proof.  Note also that the choice of $F$ does not depend on the linear form $\ell$.}
$T_\ell f$ by
$$(T_\ell f)(t,\bx)=\sum_{n=0}^{\infty}((RQ^n)(F(t,.)))(\bx)t^n$$
for $t\in V(a,b;\sigma)$, $\bx\in D(\bo;\rho)$;
here $\sigma=\min(\tilde\sigma,B)$ and we reduce $\rho$ if necessary so that
$P(D(\bo;\rho))\subset D(\bo;\sigma)$.
By construction, we then have $(T_\ell f)(P(\bx),\bx)=F(P(\bx),\bx)=f(\bx)$ for
$\bx\in\pp_P(a,b,\rho)$. As
$$\sup_{\bx\in D(\bo,\rho)}\norm{((RQ^n)(F(t,.)))(\bx)}\leq
   \norma R \norma Q ^n \sup_{\bx\in D(\bo,\tilde\rho)}\norm{F(t,\bx)},$$
we obtain that
$$\sup_{\bx\in D(\bo,\rho)}\norm{ (T_\ell f)}(t,\bx)\leq
 \frac{\norma R}{1- B \norma Q} \sup_{\bx\in D(\bo,\tilde\rho)}\norm{F(t,\bx)}
\leq \frac L{\norm t}{K(\norm t)} $$
for $t\in V(a,b,\sigma)$, where $L=\dfrac{ \tilde L\norma R}{1- B \norma Q}$.
By the definition of $R$, the expansion $J((RQ^n)(F(t,.))$ is in $\Delta_\ell(P)$ for
any $t,n$. Hence also $J((T_\ell f)(t,.))\in\Delta_\ell(P)$ for $t\in V(a,b,\sigma)$
as desired.
Hence the function $T_\ell f$ satisfies the properties wanted in the theorem.
It is defined and holomorphic on $V(a,b,\sigma)\times D(\bo,\rho)$ and the above construction of
$\sigma,\rho,L$ is independent of $f$. 

Thus it remains to show the uniqueness of $T_\ell f$.
If $G:V(a,b,\sigma)\times D(\bo,\rho)\to\CC$ is another holomorphic
function satisfying $G(P(\bx),\bx)=f(\bx)$ for sufficiently
small $\bx\in\pp_P(a,b,R)$ and $J(G(t,.))\in\Delta_\ell(P)$ for $t\in V(a,b,\sigma)$,
then $\delta=T_\ell f - G:V(a,b,\sigma)\times D(\bo,\rho)\to\CC$ satisfies
$$J(\delta(t,.))\in\Delta_\ell(P)\mbox{ for }t\in V(a,b,\sigma),\ \delta(P(\bx),\bx)=0
  \mbox{ for small }\bx\in \pp_P(a,b,\mu),$$
if $\mu>0$ is small enough.
We will show that this implies $\delta=0$.

Indeed, let $H:V(a,b,\sigma)\times D(\bo,\rho)\to\CC $
denote the function of Lemma \reff{division} with
\begin{equation}\labl{divi-delta}
\delta(t,\bx)=(t-P(\bx))H(t,\bx)\end{equation}
for all $t,\bx$. For sufficiently small positive $s$, the operators
$Q,R$ of Lemma \reff{WDT-Ob} are defined on $\OO_b(D_s)$ and
the restriction $h_t$ of $H(t,.)$ to $D_s$ is bounded, i.e.\ in $\OO_b(D_s)$.
For each fixed $t$, we can apply $Q$ to equation (\reff{divi-delta}) and obtain
$$h_t=t\,Q(h_t)$$
because $J(\delta(t,.))\in\Delta_\ell(P)$ implies $Q( \delta (t,.)\mid_{D_s})=0$.
As $Q$ is a bounded linear operator on $\OO_b(D_s)$, this is only
possible if $h_t=0$, provided $t$ is sufficiently small.
This means that $H(t,\bx)=0$ for all sufficiently small $t$ and all $\bx\in D_s$.
By the identity theorem, $H$ must vanish and hence also $\delta$. This
proves that $G=T_\ell f$ and thus the last assertion of
the theorem.

\section{Behaviour under blow-ups \green{ and ramification}}\labl{blu}
In this section, we study how the notion of asymptotics with respect to an analytic germ
behaves under blow-ups \green{ and ramification}. This will be useful to reduce
the notion to monomial asymptotics when necessary. The statements and proofs 
of this section also prepare analogous ones in the Section \reff{Gevrey}.

Consider a nonzero germ $P\in\OO=\CC\{x_1,\ldots ,x_d\}$, not a unit,
and suppose it is defined on $D(\bo;R).$ All radii of polydisks in this section are assumed 
to be smaller than $R$, but we will not mention this below.
Consider some $P$-sector $\Pi=\Pi_P(\alpha,\beta;\rho)$ and a function $f$ holomorphic on
$\Pi$.

If we suppose that some $\hat f$ is the $P$-asymptotic expansion of $f$ on $\Pi$,
then it is straightforward that $\hat f\circ b_\xi$ is the ($P\circ b_\xi$)-asymptotic
expansion of $f\circ b_\xi$ on any $(P\circ b_\xi)$-sector $\Pi_{P\circ b_\xi}(\alpha,\beta;r)$
with sufficiently small $r>0$. The converse is much more interesting, also for applications.

\begin{propo} \labl{asymp-blu}
Consider $P$ on $D(\bo;\rho)$, $\Pi=\Pi_P(\alpha,\beta;\rho)$ and $f:\Pi\to\CC$ holomorphic as above.
Suppose that for every $\xi\in\PC$, the function $f\circ b_\xi$, restricted to
$\Pi_\xi=\Pi_{P\circ b_\xi}(\alpha,\beta;r_\xi)$ with some sufficiently small $r_\xi$,
has some formal series $\hat g_\xi\in\hat\OO$ as $(P\circ b_\xi)$-asymptotic
expansion on its domain.

Then there exists a formal series $\hat f\in\hat\OO$  that is the $P$-asymptotic expansion of $f$
on $\Pi$ and it satisfies $\hat f\circ b_\xi=\hat g_\xi$ for all $\xi\in\PC$.
\end{propo}
\begin{proof} Using Definition \reff{desarrollo}, we can assume that,
for every $\xi\in\PC$, there are
sequences $\{g_n^{(\xi)}\}_n$ of bounded holomorphic functions
$g_n^{(\xi)}:D(\bo;r_\xi)\to\CC$ with
$J(g_n^{(\xi)})\to \hat g_\xi$ as $n\to\infty$ and positive constants $C_n^{(\xi)}$
such that
\begin{equation}\labl{fn-xi}
\norm{(f\circ b_\xi)(\bv)-g_n^{(\xi)}(\bv)}\leq C_n^{(\xi)}\norm{(P\circ b_\xi)(\bv)}^n
\end{equation}
for $n\in\NN$ and $\bv\in\Pi_\xi$.

As in the proof of Lemma \reff{lemadivision}, we consider the neighborhoods
$U_\xi=\phi_\xi^{-1}(D(0;r_\xi))$ of $(\xi,\bo)$ in $M$ and the functions
$G_n^{(\xi)}=g_n^{(\xi)}\circ \phi_\xi$ holomorphic on $U_\xi$.
In order to apply Lemma \reff{lemadivision2} at some point, we also consider
$\tilde U_\xi=\phi_\xi^{-1}(D(0;\tilde r_\xi))$ with some positive
$\tilde r_\xi<r_\xi$.

By (\reff{fn-xi}), we have
$$\norm{G_n^{(\xi)}-G_n^{(\zeta)}}(p) \leq (C_n^{(\xi)}+C_n^{(\zeta)})\norm{P(b(p))}^n$$
for $\xi,\zeta\in\PC$, $p\in U_\xi\cap U_\zeta$ with $\a<\arg(P(b(p)))<\beta$.
Now we apply Lemma \reff{lemadivision2} -- more precisely after right composition
with $\phi_\xi$ --  and obtain that there are bounded holomorphic functions
$H_n^{(\xi,\zeta)}: \tilde U_\xi\cap \tilde U_\zeta\to\CC$ such that
\begin{equation}\labl{Hn}
G_n^{(\xi)}-G_n^{(\zeta)} = H_n^{(\xi,\zeta)}\cdot (P\circ b)^n.
\end{equation}
More precisely, we have $\norma{H_n^{(\xi,\zeta)}}\leq L^{(\xi,\zeta)} (C_n^{(\xi)}+C_n^{(\zeta)})$,
where $L^{(\xi,\zeta)}$ denotes the constant of Lemma \reff{lemadivision2} for the domains
$U_\xi\cap U_\zeta$ and $\tilde U_\xi\cap \tilde U_\zeta$, and, as usual, $\norma{\cdot}$ denotes the maximum norm.

Now the neighborhoods $\tilde U_\xi$, $\xi\in\PC$ of $(\xi,\bo)$
cover the compact set $E=\PC\times\{\bo\}$
and therefore there is a finite subcover, say corresponding to
$\xi_0,\ldots ,\xi_K$. 

At this point, Lemma \reff{cover-E} can be applied. It yields open sets
$\bar U_{\xi_j}\subset \tilde U_{\xi_j}$ 
forming an open cover of $E$
and a constant $C$ used later in the estimates.
For $n\in\NN$, we now apply it to the family $H_n^{(\xi_i,\xi_j)}$, $i,j=0,\ldots ,K$, and obtain
bounded holomorphic functions $L_n^{(\xi_j)}:\bar U_{\xi_j}\to\CC$, $j=0,\ldots ,K$ satisfying
$$\norma{L_n^{(\xi_j)}}_\infty\leq C
\max\left\{\norma{ H_n^{(\xi_\ell,\xi_k)}}_\infty\mid \ell,k\in\{0,\ldots ,K\}\right \}$$
and $H_n^{(\xi_j,\xi_k)}=L_n^{(\xi_j)}-L_n^{(\xi_k)}$ for $j,k=0,\ldots ,K$.

%

With these functions $L_n^{(\xi_j)}$ we now define
\begin{equation}\labl{asymp-glue}
F_n^{(\xi_j)}=G_n^{(\xi_j)}- L_n^{(\xi_j)}\cdot (P\circ b)^n
\end{equation}
for $j=0,\ldots ,K$, $n\in\NN$. Then $F_n^{(\xi_j)}$ are defined on
$\bar U_{\xi_j}$ for $j=0,\ldots,K$.
The domains of the functions $F_n^{(\xi_j)}$, $j=0,\ldots,K$ again cover a neighborhood
$\tilde V$
of the exceptional divisor $E$ in $M$. By construction, the functions
satisfy
\begin{equation}\labl{majFn}
\norm{f(b(p))- F_n^{(\xi_j)}(p)}\leq \left(C_n^{(\xi_j)}+\norma{L_n^{(\xi_j)}}_\infty\right)
\norm{P(b(p))}^n\end{equation}
for $n\in\NN$, $j=1,\ldots,K$ and $p$ in the domain of $F_n^{(\xi_j)}$,
$\alpha<\arg(P(b(p)))<\beta$.
Again by construction, the functions $F_n^{(\xi_j)}$, $j=0,\ldots,K$ coincide on the
intersections of their domains, i.e.\ they glue together to functions $F_n:\tilde V\to\CC$.
As in the proof of Lemma
\reff{lemadivision}, this implies that there are holomorphic bounded
functions $f_n:V\to\CC$, $n\in\NN$,
defined on the neighborhood $V=b(\tilde V)$ of $\bo$ in $\CC^d$ such that
$F_n=f_n\circ b$ on $\tilde V$.
The inequalities (\reff{majFn}) yield that
$$\norm{f(\bx)-f_n(\bx)}\leq B_n \norm{P(x)}^n\mbox{ for }\bx\in V, \alpha<\arg(P(\bx))<\beta$$
i.e.\ on $V\cap \Pi$, where $B_n$ is the maximum of the above constants
$C_n^{(\xi_j)}+\norma{L_n^{(\xi_j)}}_\infty,$ $j=0,\ldots,K$.
Taking the above estimates into account, we find altogether
\begin{equation}\labl{bn}
B_n\leq (1+2\,L\,C)\max\{C_n^{(\xi_j)}\mid j=0,\ldots,K\},
\end{equation}
where $L$ denotes the maximum of the constants $L^{(\xi_k,\xi_\ell)}$, $k,\ell\in\{0,\ldots,K\}$
associated  in Lemma \reff{lemadivision2} to the domains
$U_{\xi_k}\cap U_{\xi_\ell}$ and $\tilde U_{\xi_k}\cap \tilde U_{\xi_\ell}$ and
$C$ denotes the constant associated to the cover $\{\tilde U_{\xi_j}\}_{j=0,\ldots,K}$
in Lemma \reff{cover-E}.

By Lemma \reff{existfhat}, (1), the sequence $\{f_n\}_n$ is a $P$-asymptotic sequence, i.e.\ there
exists $\hat f\in\OO$ satisfying $J(f_n)\equiv \hat f$ for $n\in\NN$
that is the $P$-asymptotic expansion of $f$ on $V\cap\Pi$.
As we have seen this implies for every $\xi\in\PC$ that $\hat f\circ b_\xi$ is the $(P\circ b_\xi)$-asymptotic
expansion of $f\circ b_\xi$ on any $(P\circ b_\xi)$-sector $\Pi_{P\circ b_\xi}(\alpha,\beta;\mu)$
if $\mu$ is sufficiently small. The uniqueness of the $(P\circ b_\xi)$-asymptotic expansion
(see Lemma \reff{existfhat}, (2)) now implies that $\hat f\circ b_\xi=\hat g_\xi$ and the last assertion of the theorem is proved.
\end{proof}

In the sequel, we want to improve Proposition \reff{asymp-blu} and also give a variant for $P$-asymptotic series.
For this purpose, we need some additional notation and exploit the crucial observation
that for any formal series $\hat f\in\hat\OO$, the compositions $\hat f\circ b_\xi(\bv)$, $\xi\in\PC$, are
not arbitrary formal series. Indeed, by the formulas of Subsection \reff{normal}, we see that for each term
in the series expansion
of such a composition, the exponent of $v_1$ is smaller or equal to the one of $v_2$.
Our charts were chosen so that we have this property for $\xi\in\CC$ and for $\xi=\infty$.
As a consequence, if we write $\hat f\circ b_\xi(\bv)$ as a series in powers of $v_2,\ldots,v_d$,
the coefficients are polynomials of $v_1$, i.e.\ 
\begin{equation}\labl{comppoly}\hat f\circ b_\xi(\bv)\in\CC[v_1]\lf v_2,\ldots,v_d\rf \mbox{ for }\hat f\in\hat\OO,\,\xi\in\PC.
\end{equation}

We will work in a larger set than $\CC[v_1]\lf v_2,\ldots,v_d\rf $. We consider for positive $\rho$ the algebra
$\AA_\rho=\OO_b(D(0,\rho))\lf v_2,\ldots,v_d\rf $ and regard it as a subset of $\hat\OO$.
It is endowed with the \newcommand{\mfm}{{\mathfrak m}} $\mfm'$-adic topology, where
$\mfm'=(v_2,\ldots,v_d)$. $\AA_\rho$ is complete
for this topology. It is finer than the one inherited from $\hat\OO$: Any sequence converging for
the $\mfm'$-adic topology also converges for the $\mfm$-adic topology.
As $P\circ b_\xi\in\mfm'$ for any $\xi$, the $P\circ b_\xi$-adic topology
generated by the powers of the ideal $(P\circ b_\xi(\bv))\AA_\rho$ is even finer than the $\mfm'$-adic topology.
It is readily shown that $\AA_\rho$ is also complete for the  $P\circ b_\xi$-adic topology.
\begin{lema}\labl{comp}Consider $\hat g\in\AA_\rho$ for some positive $\rho$ and $\xi\in\PC$.
Then there exists a $\rho'\in]0,\rho[$ such that for all $\zeta\in\CC$ sufficiently close
to $\xi$, the composition $\hat g\circ(\phi_\xi\circ\phi_\zeta^{-1})(\bv)$ is well defined and an element of $\AA_{\rho'}$.
\end{lema}
\begin{proof}In the case $\xi\in\CC$ this follows immediately from the formula
$\phi_\xi\circ\phi_\zeta^{-1}(\bv)=(v_1+\zeta-\xi,\bv')$.

In the case $\xi=\infty$, we have $\phi_\xi\circ\phi_\zeta^{-1}(\bv)=((\zeta+v_1)^{-1},(\zeta+v_1)v_2,\bv'')$.
If $f_\bbeta(v_1)$, $\bbeta\in\NN^{d-1}$, denotes the coefficient of $\hat f(\bv)$ in front of $(\bv')^{\bbeta}$ then
the corresponding coefficient of $\hat f\circ(\phi_\xi\circ\phi_\zeta^{-1})(\bv)$ is
of the form $(\zeta+v_1)^{\beta_2}f_\bbeta((\zeta+v_1)^{-1})$ which defines elements of $O_b(D(0,\rho'))$
for sufficiently small $\rho'$ if $\zeta\in\CC$ is sufficiently large.
\end{proof}

Another consequence of (\reff{comppoly}) is used in Generalized Weierstrass Division (see Subsection \reff{GWD}).
For a given injective linear form $\ell:\NN^d\to\RR_+$, the minimal exponent of a series $g$ is important in this context and
determines the set $\Delta_\ell(g)$ used in the statements. It will be important for us to have
an injective linear form such that the minimal exponent of $P\circ b_\xi(\bv)$ is the same for all but finitely many
$\xi\in\PC.$

In the case $d=2$, we simply consider the leading terms of $P(x_1,x_2)$ with respect to the homogeneous valuation.
Let $H(x_1,x_2)$ be their sum and $h>0$ their valuation. As $H\circ b_\xi(v_1,v_2)=H(1,\xi+v_1)v_2^h$ for $\xi\in\CC$
and $H\circ b_\xi(v_1,v_2)=H(v_1,1)v_2^h$ for $\xi=\infty$, the leading term is $v_2^h$ whatever the choice of $\ell$ provided
$H(\xi)\neq0$\ \footnote{As we have identified $\xi\equiv[1,\xi]$ in the case $\xi\in\CC$ resp.\ $\infty\equiv[0,1]$,
we use here and in the sequel $H(\xi)=H(1,\xi)$ resp.\ $H(\infty)=H(0,1)$.}.
Observe that $\xi$ with $H(\xi)=0$ correspond to the tangent directions of the curve $P(x_1,x_2)=0$ 
in the origin. Geometrically, we have shown above that the germ $P$ is monomialised by blow-ups for
all but some of these tangent directions. The tangent directions are also the intersection points of the
exceptional divisor with the strict transform of $P(x_1,x_2)=0$ under blow up.

In the case $d>2$, consider any injective linear form $\LL:\NN^{d-2}\to\RR_+$. Let $\MM$ be the set
of exponents $\bm=(m_3,\ldots,m_d)\in\NN^{d-2}$ such that there exists a nonvanishing term
$\alpha x_1^{m_1}x_2^{m_2}(\bx'')^{\bm}$ in the series of $P(\bx)$ and denote by $\ba$ the minimum of
$\MM$ with respect to the ordering $<_\LL$ induced by $\LL$. Now consider the sum $P_\LL(x_1,x_2)$ of
all terms in the series of $P(\bx)$ containg $(\bx'')^{\ba}$ and let $h\geq 0$ denote the homogenous valuation
of $P_\LL(x_1,x_2)$. Denote by $H_\LL(x_1,x_2)$ the terms of $P_\LL(x_1,x_2)$ of valuation
$h$. In order to complete $\LL$ to an injective linear form $\ell:\NN^d\to\RR_+$, choose a convenient $\ell_2$
such that 
$$\ell_2 h < \min_{\bbeta\in\MM\setminus\{\ba\}}\LL(\bbeta)-\LL(\ba)$$
and with a convenient $\ell_1$ the linear form $\ell(\bbeta)=\ell_1 \beta_1+\ell_2 \beta_2 +\LL(\bbeta'')$.
Evaluating $H_\LL\circ b_\xi(v_1,v_2)$ as above, it follows that $v_2^h(\bv'')^\ba$ is the minimal term
of $P\circ b_\xi(\bv)$ for the ordering induced by $\ell$ provided $H_\LL(\xi)\neq0$. 
Observe that $H_\LL$ can be a nonzero constant in the case $d>3$.
In order to combine the cases $d=2$ and $d>2$, we put $\LL=\emptyset$ in the former case and
choose an arbitrary injective linear form $\ell$.

We summarize the discussion in the following lemma.
\begin{lema}\labl{HLL}Consider an injective linear form $\LL:\NN^{d-2}\to\RR_+$ in the case $d>2$ and 
$\LL=\emptyset$ in the case $d=2$. There exist an injective linear form $\ell:\NN^d\to\RR_+$ completing $\LL$,
a nonnegative integer $h$ and a homogeneous polynomial $H_\LL(x_1,x_2)$ of degree $h$ such that for $\xi\in\PC$
the dominant term of $P\circ b_\xi(\bv)$ is $H_\LL(\xi)v_2^h(\bv'')^{\ba}$ provided $H_\LL(\xi)\neq0$.
\end{lema}

Our next goal is to establish a statement analogous to Proposition \reff{asymp-blu} for $P$-asymptotic
series. The first step is some kind of continuity for the notion of $P\circ b_\xi$-asymptotic sequences
with respect to $\xi\in\PC$.
\begin{lema}\labl{close}Suppose that $\xi\in\PC$ and $\{g_n\}_{n\in\NN}$ is a $P\circ b_\xi$-asymptotic
sequence in $\OO_b(D(\bo,\rho))$ for some formal series $\hat g\in\hat\OO$.
Then there exists a $\rho'\in]0,\rho[$ such that $\hat g\in\AA_{\rho'}$ and hence
there exists a neighborhood $\mathcal V$ of $\xi$ in $\PC$ such that for $\zeta \in {\mathcal V}$,
$\hat g_\zeta(\bv)=\hat g(\phi_\xi\circ\phi_\zeta^{-1}(\bv))$ defines a series in $\AA_{\rho_\zeta}$
with some $\rho_\zeta>0$ and such that
$\tilde g_n^{(\zeta)}(\bv)=g_n\left(\phi_\xi\circ\phi_\zeta^{-1}(\bv) \right)$
defines a $P\circ b_\zeta$-asymptotic sequence for $\hat g_\zeta$ in $\OO_b(D(\bo;\rho_\zeta))$
\end{lema}

\begin{proof} By definition, we have $J(g_n)\equiv \hat g\mod(P\circ b_\xi)^n\hat\OO$ and
thus $J(g_n-g_{n+1})\equiv0\mod(P\circ b_\xi)^n\hat\OO$  for all $n$. 
It is well known that this implies that   $J(g_n-g_{n+1})\equiv0\mod(P\circ b_\xi)^n \OO$ for all $n$.
This is a very special case of Artin's approximation Theorem.
In the context of our work, it also follows
using Lemma \reff{WDT}
for some arbitrary injective linear form $\ell$ in the cases $S=\CC\lf \bv\rf $ and $S=\CC\{\bv\}$.

Here we apply Lemma \reff{coro-P-germ} with the same $\ell$ for $P\circ b_\xi$ and obtain some
$D_s$, $s$ small, such that the linear operators $Q,R:\OO_b(D_s)\to \OO_b(D_s)$ are continuous.
Applying them several times, we obtain the existence of $q_n\in\OO_b(D_s)$ such that
$g_n-g_{n+1}=q_n\cdot(P\circ b_\xi)^n$ for all $n\in\NN$. If $\rho'\in]0,\rho[$ is sufficiently small,
this relation remains valid for the restrictions to $D(\bo,\rho')$.

This means that $\{g_n\}_{n\in\NN}$ is a Cauchy sequence in the $P\circ b_\xi$-adic topology of
$\AA_{\rho'}$ and thus has a limit $\hat h\in\AA_{\rho'}$ in that topology, \i.e.
\begin{equation}\labl{limgn}g_n\equiv\hat h\mod(P\circ b_\xi)^n\AA_{\rho'}\mbox{ for }n\in\NN.
\end{equation}
Therefore also
$J(g_n)\equiv \hat h\mod\mfm^n$ for all $n$. By the uniqueness of the limit we obtain that
$\hat g=\hat h$ and hence the first statement.

If $\zeta\in\PC$ is sufficiently close to $\xi$ then right composition of every term in the equations
(\reff{limgn}) with $\phi_\xi\circ\phi_\zeta^{-1}$ 
is well defined by Lemma \reff{comp}. This implies that $\tilde g_n^{(\zeta)}$ tends to $\hat g_\zeta$
in the $P\circ b_\zeta$-adic topology of $\AA_{\rho_\zeta}$ for some small $\rho_\zeta$
and hence the second statement.
\end{proof}
Now we can prove the announced statement for $P$-asymptotic series analogous to Proposition \reff{asymp-blu}.
\begin{propo}\labl{aserie-blu}
Let $\hat f\in\hat\OO$ be such that for all $\xi\in\PC$, the composition
$\hat f\circ b_\xi$ is a $P\circ b_\xi$-asymptotic series. Then $\hat f$
is a $P$-asymptotic series.
\end{propo}
\begin{proof}By definition, for every $\xi\in\PC$ there exist $\rho_\xi>0$ and a sequence $\{g_n^{(\xi)}\}_n$
in $\OO_b(D(\bo;\rho_\xi))$ such that $g_n^{(\xi)}\equiv \hat f\circ b_\xi\mod(P\circ b_\xi)^n\hat\OO$.
By the above Lemma, $g_{n\zeta}^{(\xi)}:=g_n^{(\xi)}\circ(\phi_\xi\circ \phi_\zeta^{-1})$ are well defined
for $\zeta$ in some neighborhood ${\mathcal V}_\xi$ of $\xi$ in $\PC$ and they satisfy
\begin{equation}\labl{congr}
g_{n\zeta}^{(\xi)}\equiv\hat f\circ b_\zeta\mod(P\circ b_\zeta)^n\hat\OO\mbox{ for }\zeta\in{\mathcal V}_\xi.
\end{equation}
Consider now the neighborhoods $U_\xi=\phi_\xi^{-1}(D(\bo;\rho_\xi))$ of $(\xi,\bo)$ in the blow-up variety $M$.
By reducing ${\mathcal V}_\xi$ or $\rho_\xi$, if necessary, we can asume that 
${\mathcal V}_\xi\times\{\bo\}=U_\xi\cap E$ (recall that $E\cong\PC\subseteq M$ is the exceptional divisor of the blow-up).
Since the family $U_\xi,\,\xi\in\PC,$ covers the exceptional divisor, a compact set, we can choose a finite
subcover, say $U_i,\ i=0,\ldots ,K$, where $U_i$ corresponds to a certain $\xi_i\in\PC$. For the sake of brevity of notation,
we put $g_n^i=g_n^{(\xi_i)}$, $\phi_i=\phi_{\xi_i}$, ${\mathcal V}_i={\mathcal V}_{\xi_i}$ and $G_n^{i}=g_n^i\circ \phi_i$. 
Then (\reff{congr}) can be written
\begin{equation}\labl{congr2}
G_{n}^{i}\circ \phi_\zeta^{-1}\equiv\hat f\circ b_\zeta\mod(P\circ b_\zeta)^n\hat\OO\mbox{ for }\zeta\in{\mathcal V}_i,\,i=0,\ldots,K.
\end{equation}
This implies that
\begin{equation}\labl{congrG}(G_n^{i}-G_n^{j})\circ\phi_\zeta^{-1}\equiv 0\mod(P\circ b_\zeta)^n\hat\OO
\end{equation}
for all $n$ and $\zeta\in {\mathcal V}_i\cap{\mathcal V}_j$, $i,j\in\{0,\ldots,K\}$. 
Unfortunately this does not allow us to conclude that $G_n^{i}-G_n^{j}$ can be divided by $(P\circ b)^n$ on the intersections
$U_i\cap U_j$ of their domains. These domains have to be reduced. 

For every $\xi\in\PC$, we can choose an open neighborhood $D(\xi)$ of $(\xi,\bo)$ in $M$ with the following properties.
\benot\item For every $i\in\{0,\ldots,K\}$, if $\xi\in{\mathcal V}_i$ then $\cl(D(\xi))\subset U_i$.
\item The image $\phi_\xi(D_\xi)=D_{s(\xi)}$ is a neighborhood of 0 such that Lemma \reff{WDT-Ob} is valid for
  $P\circ b_\xi$ and some arbitrarily chosen injective linear form $\ell$.
\end{enumerate}
We obtain from (\reff{congrG}) that for $i,j=0,...,K$ and $\xi\in{\mathcal V}_i\cap{\mathcal V}_j $ there exists
holomorphic functions $H_{n,\xi}^{i,j}\in\OO_b(D(\xi))$, $n\in\NN$, such that
\begin{equation}\labl{divxi}(G_n^{i}-G_n^{j})(p)=H_{n,\xi}^{i,j}(p)(P\circ b(p))^n\mbox{ for }p\in D(\xi).
\end{equation}
We just apply the operator $Q$ of Lemma \reff{WDT-Ob} several times to $(G_n^{i}-G_n^{j})\circ \phi_\xi^{-1}$ and use
(\reff{congrG}).

On nonempty intersections $D(\xi)\cap D(\zeta)$, the equations (\reff{divxi}) imply that $H_{n,\xi}^{i,j}(p)=H_{n,\zeta}^{i,j}(p)$
for $p$ in the dense subset of all $p$ such that $(P\circ b)(p)\neq0$. By continuity, we obtain that
$H_{n,\xi}^{i,j}$ and $H_{n,\zeta}^{i,j}$ coincide on such an intersection $D(\xi)\cap D(\zeta)$.
This allows us to define $U_i'=\bigcup_{\xi\in{\mathcal V}_i}D(\xi)$ and holomorphic functions $H_n^{i,j}:U_i'\cap U_j'\to\CC$ by
$$H_n^{i,j}(p)=H_{n,\xi}^{i,j}(p)\mbox{ if }p\in U_i'\cap U_j',\, p\in D(\xi).$$
By construction, the sets $U_i'$, $i=0,\ldots,K$ are open, satisfy $U_i'\cap E={\mathcal V}_i\times\{\bo\}$ and $U_i'\subset U_i$,
and the restrictions of $G_n^i$, $n\in\NN$, which we denote by the same name, satisfy
$$G_n^{i}-G_n^{j}=H_n^{i,j}(P\circ b)^n\mbox{ for all }n\in\NN,\,i,j\in\{0,\ldots,K\}.$$
Now we can proceed as in the proof of Proposition \reff{asymp-blu}.

For $i=0,\ldots,K$, we choose open sets $\tilde U_i$ such that $\cl(\tilde U_i)\subset U_i'$ and the $\tilde U_i$
still form a cover of $E$. This is possible by the compactness of $E$.  Then the functions
$H_n^{j,k}$ are bounded on $\tilde U_j\cap\tilde U_k$ for all $n\in\NN$ and all $j,k$. 
By Lemma \reff{cover-E} we find bounded holomorphic functions $L_n^{j}$, $j=0,\ldots,K$, on some open sets $\bar U_{j}\subset \tilde U_j$ 
forming an open cover of $E$
satisfying $H_n^{j,k}=L_n^{j}-L_n^{k}$ for all $n$ and all $j,k$.
Then we define $F_n^{j}=G_n^{j}-L_n^{j}(P\circ b)^n$ on $\bar U_{j}$, $j=0,\ldots,K$.

By construction, it follows that the functions $F_n^{j}$, $j=0,\ldots,K$, glue together and are hence restrictions
of some holomorphic functions $F_n:\tilde V\to\CC$, where $\tilde V$ is some neighborhood of $E$. Hence they come
from some holomorphic functions $f_n:V\to\CC$, $V=b(\tilde V)$, i.e.\ $F_n=f_n\circ b$.
Observe that $\tilde V$ and $V$ (as the neighborhoods used before) are independent of $n$.
By their construction and (\reff{congr2}),  it is easily verified that for all $\xi\in\PC$ and $n\in\NN$, we have
\begin{equation}\labl{circbxi}f_n\circ b_\xi\equiv \hat f\circ b_\xi\mod(P\circ b_\xi)^n\hat\OO .
\end{equation}
Unfortunately it is not clear how to deduce directly that $f_n\equiv\hat f\mod P^n\hat\OO$ for all $n$.

It is more convenient to consider the sequence $\{f_n-f_{n+1}\}_{n\in\NN}$. From (\reff{circbxi}) and Lemma \reff{WDT} applied
for $S=\hat\OO$ and $S=\OO$, we find that
$$(f_n-f_{n+1})\circ b_\xi\in (P\circ b_\xi)^n\OO$$
for all $n,\xi$. As in the proof of Lemma \reff{lemadivision} the implies that $f_n-f_{n+1}\in P^n\OO$ for $n$.
Hence $\{f_n\}_n$ is a Cauchy sequence for the $P$-adic topology and hence has a limit $\hat g\in\hat\OO$
in it, i.e.\ $f_n\equiv \hat g\mod P^n\hat\OO$ for all $n$. 
By (\reff{circbxi}) and the uniqueness of the limit we obtain that $\hat g\circ b_\xi=\hat f\circ b_\xi$
for all $\xi\in\PC$. This implies that $\hat g=\hat f$ and completes the proof.

Observe that if $\norm{g_n^{(\xi)}}\leq K_n^{(\xi)}$ on $D(\bo;\rho_\xi)$, then 
$\norm{H_n^{j,k}}\leq M^{j,k}(K_n^{(\xi_j)}+K_n^{(\xi_k)})$, where $M^{j,k}$
denotes the constant of Lemma \reff{lemadivision2} for $U_j'\cap U_k'$ and $\tilde U_j\cap\tilde U_k$.
This implies with 
$$\norm{L_n^{j}}\leq C\max\{\norm{H_n^{j,k}}\mid j,k=0,\ldots,K\},$$
where $C$ is the constant of Lemma \reff{cover-E} for the cover $\tilde U_j$, $j=0,\ldots,K$, that
\begin{equation}\labl{majFn2}
\norm{ F_n^{j}}\leq K_n^{(\xi_j)}+({\textstyle\sup_{\bar U_j}} \norm{ P\circ b})^n\,C\,
   \max\{   M^{k,l}(K_n^{(\xi_k)}+K_n^{(\xi_l)})\mid k,l=0,\ldots,K   \}.
\end{equation}
This will be used later.
\end{proof}

In order to improve Proposition \reff{asymp-blu}, we first note a consequence of Lemma \reff{close}.
\begin{lema}\labl{coroclose}  Suppose that $g$ is holomorphic on $\Pi_{P\circ b_\xi}(a,b,R)$ and has some 
$P\circ b_\xi$-asymptotic expansion $\hat g_\xi$ on  it.
Then for  $\zeta\in\PC$ close the $\xi$, the composition $\hat g_\xi(\phi_\xi\circ\phi_\zeta^{-1}(\bv))$ 
is well defined and it is the $P\circ b_\zeta$-asymptotic expansion of $g\circ(\phi_\xi\circ\phi_\zeta^{-1})$ on 
$\Pi_{P\circ b_\zeta}(a,b,\rho_\zeta)$ for sufficiently small positive $\rho_\zeta$.

In particular, if $f$ holomorphic on $\Pi_P(a,b,R)$ and $\xi\in\PC$ such that $f\circ b_\xi=g$ satisfies the assumption,
then the statement holds for $g\circ(\phi_\xi\circ\phi_\zeta^{-1})=f\circ b_\zeta$ and $\zeta$ close to $\xi$.
\end{lema}
\begin{proof} By definition, there exists $\rho>0$ and a $P\circ b_\xi$-asymptotic sequence
$\{g_n\}_n$ for $\hat g_\xi$ on $\OO_b(D(\bo;\rho))$ such that
\begin{equation}\labl{estcoro} \norm{f\circ b_\xi(\bv)-g_n(\bv)}\leq K_n\norm{P\circ b_\xi(\bv)}^n
\end{equation}
for $\bv \in \Pi_{P\circ b_\xi}(a,b,\rho)$. By the above Lemma, 
$\tilde g_n^{(\zeta)}(\bv)=g_n\left(\phi_\xi\circ\phi_\zeta^{-1}(\bv) \right)$
defines a $P\circ b_\zeta$-asymptotic sequence for 
$\hat g_\xi(\phi_\xi\circ\phi_\zeta^{-1}(\bv))$. Substitution of
$\phi_\xi\circ\phi_\zeta^{-1}(\bv)$ for $\bv$ in (\reff{estcoro}) yields that
$$\norm{f\circ b_\zeta(\bv)-\tilde g_n^{(\zeta)}(\bv)}\leq K_n\norm{P\circ b_\zeta(\bv)}^n$$
and the corollary is proved. Observe that for $\zeta$ close to $\xi$ and small $\bv$
with $\arg(P\circ b_\zeta(\bv))\in]a,b[$, we have $\phi_\xi\circ\phi_\zeta^{-1}(\bv)\in\Pi_{P\circ b_\xi}(a,b,\rho)$.
(Indeed, using the formulas for  $\phi_\xi\circ\phi_\zeta^{-1}(\bv)$, we find that $\phi_\xi\circ\phi_\zeta^{-1}(\bo)$ 
is small if $\zeta$ is close to $\xi$. Hence by continuity, there is a $\rho'\in]0,\rho[$ such that
$\norm{  \phi_\xi\circ\phi_\zeta^{-1}(\bv) }<\rho$ if $\norm{\bv}<\rho'$. The rest follows from
$(P\circ b_\xi)\circ(\phi_\xi\circ\phi_\zeta^{-1})=P\circ b_\zeta$.)
\end{proof}

Next, we need a statement concerning the $\xi$-dependence of $T_\ell(f\circ b_\xi)$.
\newcommand{\mcZ}{{\mathcal Z}} \newcommand{\mcD}{{\mathcal D}}
\begin{lema}\labl{FxiG}Consider $\LL,\,\ell,\,h,\,H_\LL$ as in Lemma \reff{HLL} and
let $\mcZ_\LL$ denote the set of zeros of $H_\LL$ in $\PC$.
Let a bounded function $f\in\Pi_P(a,b,R)$ be given 
Then for any $\mcD$ with $\cl(\mcD)$ compact and contained
in $\CC\setminus\mcZ_\LL$, there exist $\sigma,\rho>0$ and a function $G_\LL:V(a,b,\sigma)\times \Omega\to\CC$,
$\Omega=\mcD\times D(0;\rho)^{d-1}$, such that for $\xi\in\mcD$, the function 
$F_\xi(t,\bv)=T_\ell(f\circ b_\xi)(t,\bv)$ of Theorem \reff{T-sector-germ} defined for $P\circ b_\xi$ 
satisfies $F_\xi(t,\bv)=G_\LL(t,(v_1+\xi,\bv'))$ for $t\in V(a,b,\sigma)$ and small 
$\bv$.
\end{lema}
This means essentially that the functions $T_\ell(f\circ b_\xi)$, $\xi\in\CC\setminus\mcZ_\LL$
glue together to a single function except for shifts in the variable $v_1$.
\noindent\begin{proof}
Applying Theorem \reff{T-sector-germ} for $\xi\in\CC\setminus\mcZ_\LL$ and with $P$ replaced by $P\circ b_\xi$
yields positive $\sigma_\xi,\rho_\xi$ and functions $F_\xi:V(a,b,\sigma_\xi)\times D(\bo;\rho_\xi)\to\CC$
such that for any $t$, we have $J(F_\xi(t,.))\in\Delta_\ell(P\circ b_\xi)$ and 
$F_\xi(P\circ b_\xi(\bv),\bv)=f(b_\xi(\bv))$ for $\bv\in\Pi_{P\circ b_\xi}(a,b,\rho_\xi)$.
Observe that $F_\xi$ are uniquely determined by this property except for restrictions.

Putting $\tilde F_\zeta(t,\bv)=F(t,\phi_\xi\circ\phi_\zeta^{-1}(\bv)))=F_\xi(t,(v_1+\zeta-\xi,\bv'))$ for $\zeta$ close to $\xi$
and sufficiently small $\bv$, this function has the defining properties of $F_\zeta$.
By the essential uniqueness of the latter, we obtain
\begin{equation}\labl{eqFxi}F_\zeta(t,\bv)=F_\xi(t,{v_1+\zeta-\xi,\bv'})
\end{equation}
for $\zeta$ close to $\xi$ and sufficiently small $\bv$.
This means that $F_\xi$ and $F_\zeta$ are analytic continuations of each other, except for
a shift in the $v_1$ component.

Consider now  $\mcD$ compactly contained in $\CC\setminus\mcZ_\LL$ and choose a finite subset
$A$ of the closure of $\mcD$ such that the disks $\xi+D(0,\rho_\xi)$ cover $\mcD$.
This allows us to define
$$G(t,\bv)=F_\xi(t,v_1-\xi,\bv')\mbox{ for }t\in V(a,b,\sigma)\mbox{ and }\bv\in\Omega, $$
if $v_1\in\xi+D(0,\rho_\xi)$ where $\sigma=\min\{\sigma_\xi\mid \xi\in A\}$,
$\rho=\min\{\rho_\xi\mid \xi\in A\}$. The property (\reff{eqFxi}) implies that the value of
$G(t,\bv)$ is independent of the choice of $\xi$ with $v_1\in\xi+D(0,\rho_\xi)$.
$G$ has the wanted properties by construction.
\end{proof}
Now we are in a position to prove that the assumption on $f\circ b_\xi$ in Proposition
\reff{asymp-blu} is only needed for finitely many $\xi\in\PC$. 
\begin{teorema}\labl{asymp-blu+}Consider $\LL,\,\ell,\,h,\,H_\LL,\mcZ_\LL$ as in Lemma \reff{FxiG} and
put $\mcZ=\mcZ_\LL$ if $H_\LL$ is not a constant, $\mcZ=\{\xi\}$ with arbitrary $\xi\in\PC$ otherwise. 
If $f$ is holomorphic  and bounded on $\Pi=\Pi_P(a,b,R)$ such that for $\xi\in\mcZ$, the composition
$f\circ b_\xi$ has some $P\circ b_\xi$-asymptotic expansion $\hat g_\xi$ on
$\Pi_{P\circ b_\xi}(a,b,R)$ then
$f$ has a $P$-asymptotic expansion $\hat f$ on $\Pi$ such that $\hat f\circ b_\xi=\hat g_\xi$
for $\xi\in\mcZ$.
\end{teorema}
\begin{proof}
As we can always achieve this by a linear transformation in the $(x_1,x_2)$-space, we assume that
$\infty\in\mcZ_\LL$ if $H_\LL$ is not constant. Indeed, if we replace $(x_1,x_2)$ by $(x_1,x_2)A$ with some invertible $2\times2$ matrix
$A$, then $\mcZ_\LL$ changes to $\mcZ_\LL A^{-1}$, the set of all $[\alpha,\beta]$ such that $[\alpha,\beta]A\in\mcZ_\LL$.
The same transformation allows us to assume that  $\mcZ=\{\infty\}$ in the case that $H_\LL$ is constant.

This means that we have now $\infty\in\mcZ$ in all the cases.
Put $\mcZ'=\mcZ\setminus\{\infty\}$. This set may be empty, in which case the proof is a little simplified.

By Corollary \reff{coroclose}, $f\circ b_\xi$ has a $P\circ b_\xi$-asymptotic expansion for large $\xi\in\CC$ and
for $\xi$ close to $\mcZ'$ \footnote{If $\mcZ'\neq\emptyset$; we only give the proof in the case.}.
To fix notation assume this is the case for $\norm\xi>R/2$ resp.\ $\dist(\xi,\mcZ')<2r$.

Then we apply Lemma \reff{FxiG} with $\mcD=D(0,R)\setminus\bigcup_{\chi\in\mcZ'}\cl(D(\chi,r))$
to $f$. We obtain 
$\sigma,\rho>0$ and a holomorphic function $G_\LL:V(a,b,\sigma)\times \Omega\to\CC$,
$\Omega=\mcD\times D(0;\rho)^{d-1}$, such that for $\xi\in\mcD$, the function 
$T_\ell(f\circ b_\xi)(t,\bv)$ of Theorem \reff{T-sector-germ} 
satisfies 
\begin{equation}\labl{FG} T_\ell(f\circ b_\xi)(t,\bv)=G_\LL(t,(v_1+\xi,\bv'))\mbox{ for }
t\in V(a,b,\sigma)\mbox{ and small }\bv.\end{equation}

By Theorem \reff{st-form}, $T_\ell(f\circ b_\xi)(t,\bv)$ has a uniform asymptotic expansion
as $V(a,b,\sigma_\xi)\ni t\to 0$ if $\xi$ is in some neighborhood of $\partial \mcD$, the boundary of $\mcD$.
Using (\reff{FG}) and the compactness of $\partial \mcD$, this implies that  there exist some positive $\sigma,\rho$ such that
$G(t,\bv)$ has a uniform limit as $V(a,b,\sigma)\ni t\to 0$ for
$\bv$ with $\norm{\bv'}<\rho$, $\dist(v_1,\partial \mcD)<\rho$. 
By the Cauchy criterion, this is equivalent to
\begin{multline}\labl{cauchy}
\forall\,\eps>0\,\exists\,\delta>0\,\forall\, t_1,t_2\in V(a,b,\sigma) \left( \norm{t_1},\norm{t_2}<\delta,\norm{\bv'}<\rho,\dist(v_1,\partial \mcD)<\rho \right.
 \\ \left. \implies \norm{G(t_1,\bv)-G(t_2,\bv)}<\eps \right).
\end{multline}
Here we can apply the maximum modulus principle to $G(t,\bv)$ in the variable $v_1$ on the domain $\mcD$.
This implies that 
\begin{multline}\labl{cauchy2}
\forall\,\eps>0\,\exists\,\delta>0\,\forall\, t_1,t_2\in V(a,b,\sigma) \left( \norm{t_1},\norm{t_2}<\delta,\norm{\bv'}<\rho, v_1\in \mcD \right.
\\ \left. \implies \norm{G(t_1,\bv)-G(t_2,\bv)}<\eps \right).
\end{multline}
This means that $G(t,\bv)$ has some uniform limit as $V(a,b,\sigma)\ni t\to 0$, say $g_0(\bv)$, for $\bv\in\Omega$.

In the same manner, we show that $\frac1t(G(t,\bv)-g_0(\bv))$ has a uniform limit as $V(a,b,\sigma)\ni t\to 0$
for $\bv\in\Omega$ etc.\ and obtain that $G(t,\bv)$ has an asymptotic expansion as $V(a,b,\sigma)\ni t\to 0$, uniformly for
$\bv\in\Omega$. By (\reff{FG}) and Theorem \reff{st-form}, this means that $f\circ b_\xi$ has a $P\circ b_\xi$-asymptotic expansion
for every $\xi\in\mcD$. Now this is also the case for the remaining $\xi$ as discussed in the beginning of the proof.
Proposition \reff{asymp-blu} implies the statement of the theorem.
\end{proof}

In the last part of this Section, we want to improve Proposition \reff{aserie-blu} in a way similar to
Theorem \reff{asymp-blu+}. We first show for any formal series $\hat f\in\hat\OO$
that the coefficients series of all but finally many of the series $T_\ell(\hat f\circ b_\xi)$, $\xi\in\PC$, 
can be combined into one formal series.

\begin{lema}\labl{hatfhatG}Consider $\LL,\,\ell,h,\,\,H_\LL$ as in Lemma \reff{HLL} and
let $\mcZ_\LL$ denote the set of zeros of $H_\LL$ in $\PC$. 
Let $\mcR=\CC[v_1,1/H_\LL(1,v_1)]$ and $\BB=\mcR\lf v_2,\ldots,v_d\rf $. Finally consider $\hat f\in\hat\OO$. 
Then there exists a formal series $\hat G(t,\bv)=\sum_{n=0}^\infty \hat g_n(\bv)t^n\in\BB\lf t\rf $
with the following property. For all $\xi\in\CC\setminus\mcZ_\LL$, the series 
$$T_\ell(\hat f\circ b_\xi)(t,\bv)=\sum_{n=0}^\infty \hat f_{n\xi}(\bv)t^n\in\Delta_\ell(P\circ b_\xi)\lf t\rf $$
of Lemma \reff{defT-poly} applied with $P\circ b_\xi$ in place of $P$ satisfies
\begin{equation}\labl{hatfn}\hat f_{n\xi}(\bv)=J_\xi(\hat g_n)(\bv)\mbox{ for }n\in\NN.
\end{equation}
\end{lema}
Here $\dis J_\xi(h)(\bv)=\sum_{k=0}^\infty \frac1{k!}\frac{\partial^k h}{\partial (v_1)^k}(\xi,\bv')v_1^k$
is obtained from some element $h\in\BB$ by replacing each of its coefficients  
by its Taylor series in the point $\xi$. Since the coefficients are elements of $\mcR$
and hence rational functions of $v_1$ the denominator of which is a power of $H_\LL(1,v_1)$, this is possible
for $\xi\in\CC\setminus\mcZ_\LL$. Observe that $J_\xi$ is compatible with multiplication.
\begin{proof}
We consider the linear operator $B_0:\CC\lf \bx\rf \to\CC[v_1]\lf \bv'\rf $ determined by
$B_0(\bx^\bal)=v_1^{\alpha_2}v_2^{\alpha_1+\alpha_2}(\bv'')^{\bal''}$ and continuity with respect to
$\mfm$-adic topology of $\CC\lf \bx\rf $ and the $\mfm'$-adic topology of $\CC[v_1]\lf \bv'\rf $.

Observe that $J_\xi(B_0(\hat g))=\hat g\circ b_\xi\in\CC\lf \bv\rf $ for $\hat g\in\CC\lf \bx\rf $. 
The difference of $B_0(\hat g)$ and $\hat g\circ b_0$ is essentially that the former is in  $\CC[v_1]\lf \bv'\rf $
whereas the latter is in $\CC\lf \bv\rf $. Expansion of the polynomial coefficients in their Taylor series
at the origin maps $B_0(\hat g)$ to $\hat g\circ b_0$. The introduction of $B_0$ becomes more useful
if  $\CC[v_1]\lf \bv'\rf $ is considered as a subset of $\BB$; $\BB$ cannot always be identified with a subset
of $\CC\lf \bv\rf $.

We consider $B_0(\hat f)$ and $B_0(P)$ as elements of $\BB=\mcR\lf v_2,\ldots,v_d\rf $ and use
the injective linear form $\ell'(a_2,\ldots,a_d)=\ell_2a_2+\ldots+\ell_d a_d$ on $\NN^{d-1}$.
By Lemma \reff{HLL}, the dominant term of $B_0(P)$ is 
$H_\LL(1,v_1)v_2^h(\bv'')^{\ba}$ with $\ba=(a_3,\ldots,a_d)$ and certain nonnegative $a_j$.

Now we apply Corollary \reff{coro-P-serie}; this is possible because $H_\LL(1,v_1)$
is a unit of $\mcR$ by construction. Therefore we can write uniquely
\begin{equation}\labl{expand} B_0(\hat f)=\sum_{n=0}^\infty \hat g_n B_0(P)^n
\end{equation}
with certain $\hat g_n\in\BB\cap \Delta_{\ell'}(B_0(P))$. These can be written
$$
\hat g_n(\bv)=\sum_{\bbeta\in\NN^{d-1}} g_{n,\bbeta}(v_1) (\bv')^{\bbeta}
$$
with certain $g_{n,\bbeta}\in \mcR$; by definition of $\Delta_{\ell'}(B_0(P))$ we have
$g_{n,\bbeta}=0$ if $\bbeta\in(h,\ba)+\NN^{d-1}$.
We define 
$$\hat G(t,\bv)=\sum_{n=0}^{\infty}\hat g_n(\bv)t^n\in\BB\lf t\rf $$
and it remains to show that
$\hat G$ has the wanted properties.


We can apply $J_\xi$, $\xi\in\CC\setminus\mcZ_\LL$,  to the equality (\reff{expand}) and obtain
\begin{equation}\labl{expandJ}
\hat f \circ b_\xi(\bv)=\sum_{n=0}^\infty J_\xi(\hat g_n)(\bv) (P\circ b_\xi(\bv))^na.
\end{equation}

Observe that $\hat g_n\in\Delta_{\ell'}(B_0(P))$ implies that for $\xi\in\CC\setminus\mcZ_\LL$,
we have $J_\xi(\hat g_n)\in\Delta_\ell(P\circ b_\xi)$ because the leading terms used to define these vector spaces
are $ v_2^h v_3^{a_3}\cdot..\cdot v_d^{a_d}$ respectively  $v_1^0 v_2^h v_3^{a_3}\cdot..\cdot v_d^{a_d}$.
This implies that (\reff{expandJ}) is actually the unique way to write $\hat f\circ b_\xi$ as a series
$\hat f\circ b_\xi=\sum_{n=0}^\infty \hat f_{n\xi}(P\circ b_\xi)^n$
with $\hat f_{n\xi}\in\Delta_\ell(P\circ b_\xi)$ if $\xi\in\CC\setminus\mcZ_\LL$.
This means that $T_\ell(\hat f\circ b_\xi)(t,\bv)=\sum_{n=0}^\infty J_\xi(\hat g_n)(\bv)t^n$
thus proving  the Lemma.
\end{proof}

Now we can also improve Proposition \reff{aserie-blu}.
\begin{teorema}\labl{aserie-blu+} Consider $\LL,\,\ell,h,\,\,H_\LL,\mcZ_\LL$ and $\mcZ$ as in Theorem 
\reff{asymp-blu+}. Let $\hat f\in\OO$ be given such that $\hat f \circ b_\xi$ is a 
$P\circ b_\xi$-asymptotic series for $\xi\in\mcZ$. Then $\hat f$ is a
$P$-asymptotic series.\end{teorema}
Hence as in Theorem \reff{asymp-blu+}, consideration of $\hat f\circ b_\xi$ for finitely
many $\xi\in\PC$ is already sufficient.\\
\noindent\begin{proof}
We can again assume that $\infty\in\mcZ$; otherwise we proceed as in the beginning of the
proof of Theorem \reff{asymp-blu+}. By assumption and Lemma \reff{comp}, $\hat f\circ b_\xi$
is a $P\circ b_\xi$-asymptotic series for large $\xi$ and for $\xi$ close to
$\mcZ'=\mcZ\setminus\{\infty\}$. Therefore in the series expansions
$$T_\ell(\hat f\circ b_\xi)=\sum_{n=0}^\infty \hat f_{n\xi}(\bv)t^n$$
of Corollary \reff{P-adic-clos}, the coefficients $\hat f_{n\xi}(\bv)$ are convergent series
with a common radius of convergence, say $\rho_\xi>0$, for these $\xi$.

We can also apply the above Lemma \reff{hatfhatG} and obtain a series 
$\hat G(t,\bv)=\sum_{n=0}^\infty \hat g_n(\bv)t^n\in\BB\lf t\rf $ such that
for $\xi\in\CC\setminus\mcZ'$ we have
\begin{equation}\labl{Tlg}T_\ell(\hat f\circ b_\xi)=\sum_{n=0}^\infty J_\xi(\hat g_n)(\bv)t^n.
\end{equation}
Here $J_\xi(\hat g_n)=\hat f_{n\xi}$ are convergent series for all $n$ and 
$\xi$ large or $\xi$ close to $\mcZ'$. Therefore if we write
$$J_\xi(\hat g_n)(\bv)=\sum_{\bbeta\in\NN^{d-1}}h_{n,\bbeta}^{\xi}(v_1)(\bv')^\bbeta$$
then there exist for these $\xi$ constants $K_{n,\xi}$ and $\rho_\xi>0$ such that
$$\norm{h_{n,\bbeta}^{\xi}(v_1)}\leq K_{n,\xi}\rho_\xi^{-\norm\bbeta}\mbox{ for }\bbeta\in\NN^{d-1},\ \norm{v_1}<\rho_\xi.$$
For the coefficients $H_{n,\bbeta}(v_1)$ in the expansion
$\hat g_n(\bv)=\sum_{\bbeta\in\NN^{d-1}}H_{n,\bbeta}(v_1)\bv'^\bbeta,$
which are elements of $\mcR=\CC[v_1,1/H_\LL(1,v_1)]$ and hence can be considered as
holomorphic functions on $\CC\setminus\mcZ'$,
this means that for $\xi$ large or $\xi$ close to $\mcZ'$ there exist positive $\rho_\xi$ such that
\begin{equation}\labl{Hnbeta}\norm{H_{n,\bbeta}(\xi+v_1)}\leq K_{n,\xi}\rho_\xi^{-\norm\bbeta}
   \mbox{ if }\norm{v_1}<\rho_\xi.
\end{equation}

Consider now the domain  $\mcD=D(0,R)\setminus\bigcup_{\chi\in\mcZ'}\cl(D(\chi,r))$
where $r,R$ were chosen such that (\reff{Hnbeta}) holds for $\norm\xi>R/2$ or
$0<\dist(\xi,\mcZ')<2r$. Then $H_{n,\bbeta}$ is  holomorphic on a neighborhood of the closure of $\mcD$ 
and, by compactness, there are $K_n,\rho>0$ such that for $v_1$ on the boundary of $\mcD$
we have $\norm{H_{n,\bbeta}(v_1)}\leq K_n\rho^{-\norm\bbeta}$.
The maximum modulus principle implies here that
$$\norm{H_{n,\bbeta}(v_1)}\leq K_n\rho^{-\norm\bbeta}\mbox{ for all }v_1\in\mcD.$$
Using that  $\hat g_n(\bv)=\sum_{\bbeta\in\NN^{d-1}}H_{n\bbeta}(v_1)(\bv')^\bbeta,$
we obtain that the coefficients of $t^n$ in  (\reff{Tlg}) are convergent series on some common polydisk
{\em for all $\xi\in\mcD$}. By Remark \reff{nota-P-clos}, (1) this proves that $\hat f\circ b_\xi$
is a $P\circ b_\xi$-asymptotic series for every $\xi\in\mcD$.
As we already know this for the remaining $\xi\in\PC$, we can apply Proposition
\reff{aserie-blu} and finally obtain the statement.
\end{proof}

\green{We end this section with a discussion of the compatibility of asymptotics with respect to
an analytic germ and ramification. The proofs are much simpler here.
\begin{propo} \labl{asymp-ram}
Consider $P$ on $D(\bo;\rho)$, $\Pi=\Pi_P(\alpha,\beta;\rho)$ and $f:\Pi\to\CC$ holomorphic as above.
Suppose that for some integer $k\geq2$, the function $f\circ r_k$, restricted to
$\Pi_k=\Pi_{P\circ r_k}(\alpha,\beta;\tilde\rho)$ with some sufficiently small $\tilde\rho$,
has some formal series $\hat g\in\hat\OO$ as $(P\circ r_k)$-asymptotic
expansion on its domain.

Then there exists a formal series $\hat f\in\hat\OO$  that is the $P$-asymptotic expansion of $f$
on $\Pi$ and it satisfies $\hat f\circ r_k=\hat g$.
\end{propo}
\begin{proof} Using Definition \reff{desarrollo}, we can assume that there exists a sequence
$\{g_n\}_n$ of bounded holomorphic functions $g_n:D(\bo;\tilde\rho)\to\CC$ with $J(g_n)\to\hat g$
as $n\to\infty$ and constants $C_n$ such that
\begin{equation}\labl{maj-ram}
\norm{f\circ r_k(\bv)-g_n(\bv)}\leq C_n\norm{P\circ r_k(\bv)}^n
\end{equation}
for $n\in\NN$ and $\bv\in\Pi_k$.

By construction, $f\circ r_k$ and $P\circ r_k$ are invariant under right composition with the
rotation $R:\bv\mapsto (e^{2\pi i/k}v_1,\bv')$ and as a consequence also $\Pi_k$. 
This implies using (\reff{maj-ram}) that
$$
\norm{f\circ r_k(\bv)-g_n\circ R^j(\bv)}\leq C_n\norm{P\circ r_k(\bv)}^n
$$
for $n\in\NN$, $j=0,...,k-1$ and $\bv\in\Pi_k$.

Consider now the sequence $\{h_n\}$ in $\OO_b(\Pi_k)$ defined by $h_n=\frac1k\sum_{j=0}^{k-1}g_n\circ R^j$.
Clearly, we have 
$$
\norm{f\circ r_k(\bv)-h_n(\bv)}\leq C_n\norm{P\circ r_k(\bv)^n}\mbox{ for }n\in\NN,\ \bv\in\Pi_k
$$
and $J(h_n)\to\hat h$ as $n\to\infty$ with $\hat h=\frac1k\sum_{j=0}^{k-1}\hat g\circ R^j$.
Furthermore, $h_n$ and $\hat h$ are invariant under right composition with the rotation $R$. Hence there
exist bounded holomorphic $f_n:\tilde\Pi\to\CC$, $\tilde\Pi=\Pi_P(\alpha,\beta;r)$ with some small positive $r$
and $\hat f\in\hat\OO$ such that $f_n\circ r_k=h_n$ and $\hat f\circ r_k=\hat h$.

We obtain the wanted properties
$$
\norm{f(\bx)-f_n(\bx)}\leq C_n\norm{P(\bx)^n}\mbox{ for }n\in\NN,\ \bx\in\tilde\Pi
$$
and $J(f_n)\to\hat f$ as $n\to\infty$. This proves that $f$ has $\hat f$ as $P$-asymptotic expansion.
The last assertion of the proposition follows from the fact that $f\circ r_k$ has $\hat f\circ r_k$ and
$\hat g$ as $P\circ r_k$-asymptotic expansions.
\end{proof} 

We can also prove a statement for $P$-asymptotic series analogous to Proposition \reff{asymp-ram}.
\begin{propo}\labl{aserie-ram}
Let $\hat f\in\hat\OO$ and $k\in\NN$, $k\geq2$ be such that  the composition
$\hat f\circ r_k$ is a $P\circ r_k$-asymptotic series. Then $\hat f$
is a $P$-asymptotic series.
\end{propo}
\begin{proof} Is it analogous to the previous one. If $\{g_n\}_n$ is a $P\circ r_k$-asymptotic sequence
in $\Pi_{P\circ r_k}(\alpha,\beta;\tilde\rho)$ for some $\alpha,\ \beta,\ \tilde\rho$
with limit $\hat f\circ r_k$ then so is $\{h_n\}_n$ where $h_n=\frac1k\sum_{j=0}^{k-1}g_n\circ R^j$,
$R$ the rotation $R:\bv\mapsto (e^{2\pi i/k}v_1,\bv')$.
Again the sequence $\{f_n\}_n$ in $\Pi_{P}(\alpha,\beta;r)$, $r>0$ sufficiently small, where $f_n\circ r_k$
is a restriction of $g_n$, is a $P$-asymptotic sequence for $\hat f$. 
\end{proof}

}

\section{Gevrey asymptotic expansions and summability with respect to a germ}\labl{Gevrey}
\subsection{Gevrey asymptotic expansions with respect to an analytic germ}
We first give a definition analogous to Definition \reff{desarrollo} and then a characterization
analogous to Theorem \reff{st-form}. Theorem \reff{T-sector-germ} is again crucial in order to
establish a relation between asymptotics involving powers of the germ and single variable asymptotics
with coefficients in a certain Banach space.

We consider again a nonzero germ of analytic function $P(\bx)\in \OO=\CC\{x_1,\ldots,x_d\}$, not a
unit (i.e. $P(\bo)=0$), defined in $D(\bo; \rho)$, $\rho>0$.

\begin{defin}\labl{desarrolloGevrey}
Given  a $P$-sector $\Pi=\Pi_P(a, b; \br)$, and
$f\in\OO (\Pi)$, we will say that $\hat{f}\in  \hat{\OO}$ is
the $P$-Gevrey asymptotic expansion of order $s$ of $f$, or
more briefly its ``$P$-$s$-(Gevrey) asymptotic expansion'', if there exist
$\rho>0$, a family $\{ f_n \in \OO_b ({D} (\bo; \rho))\}_n$, converging to
$\hat{f}$ in the ${\mathfrak m}$-adic topology, and constants
$K$, $A>0$ such that\vspace{0.2cm}
\be
\item $\forall\, n\in \NN,\ \forall\,\bx\in D(\bo;\rho),\ \norm{f_n (\bx)} \leq KA^n \Gamma
(sn+1)$.\vspace{0.2cm}
\item $\forall\, n\in \NN,\ \forall\,\bx\in \Pi\cap D(\bo;\rho),\ \norm{f(\bx)-f_n (\bx)}
    \leq KA^n\Gamma (sn+1) \norm{P(\bx)}^n$.
\ee
A sequence $\{f_n\}_{n\in\NN}$ satisfying (1) and
$J(f_n)\equiv\hat f\mod P^n\hat\OO$
will be called a {\em $P$-$s$-asymptotic sequence for } $\hat f$.
\end{defin}

\begin{nota} \benot \item As for Definition \reff{desarrollo}, this definition is compatible with
changes of variables and with multiplication of $P$ by a unit $U\in\OO$. This is verified in the
same way as in Remark \reff{nota-desarrollo}.
\item Again as for Definition \reff{desarrollo}, the definition is
independent of the choice of the  $P$-$s$-asymptotic sequence
for $\hat f$.
Indeed, if $\{f_n\}_{n\in\NN}$ is as in the definition and if
$\{g_n\}_{n\in\NN}$ is another $P$-$s$-asymptotic sequence for
$\hat f$, then $J(f_n)\equiv J(g_n)\mod P^n\hat\OO$ for all $n$.
Here Lemma \reff{WDT-Ob} can be applied in $\OO_b(D_s)$, $D_s=D(\bo;(s^{\ell_1},\ldots,s^{\ell_d}))$,
for sufficiently small positive $s$ as it was done below Lemma \reff{existfhat}.
It shows that we can write
$f_n-g_n=h_n\,P^n$, $h_n=Q^n(f_n-g_n)\in \OO_b(D_s)$.

Hence, there exist some positive $\rho'<\rho$
and positive constants $L,M$ such that
$$\norm{f_n(\bx)-g_n(\bx)}\leq L M^n \Gamma(sn+1)\norm{P(\bx)}^n$$
for all $\bx\in D(\bo;\rho')$ and all $n$. This implies that property (2) also
holds for $\{g_n\}_n$ with certain $K,A$.
\item This definition agrees with the well-known definition of
Gevrey asymptotics in one variable, i.e., if $P=x$. In fact,
suppose that $\norm{f(x)-f_n(x)}\leq KA^n \Gamma (ns+1)
\norm{x}^n$ on a sector $V$, with $f_n \in \OO_b (D(0;\rho))$ satisfying
$\norm{f_n(x)}\leq KA^n \Gamma (ns+1)$. Write
$f_n(x)=J_{n-1}(f_n)+ \tilde{f}_n (x)$. Let $S$ again denote the operator defined by
$S(f)= \frac{f(x)-f(0)}{x}$ on $\OO_b (D(0; \rho))$. The Maximum Modulus Principle
implies that $\norma{S(f)}_\infty \leq \frac{2}{\rho} \cdot
\norma{f}_{\infty } $. Moreover, $\tilde{f}_n (x)= x^n S^n (f_n)$.
Hence
\begin{equation*}
\begin{split}
\norm{f(x)- J_{n-1}(f)} & \leq \norm{f(x)-f_n(x)} +
\norm{\tilde{f}_n (x)} \\
& \leq KA^n \Gamma (ns+1) \norm{x}^n + \norm{x}^n \left( \frac{2}{\rho}\right)^n
KA^n \Gamma (ns+1),
\end{split}
\end{equation*}
as wanted. The converse follows from the Gevrey property of the formal series $\hat f$
and is left to the reader.
\ee
\end{nota}

As for general asymptotic expansions in a germ (see Theorem \reff{st-form}),
we want to write Gevrey expansions in an expression in a standard form. Recall that
this standard form depends on the choice of $\ell$ in the subsequent theorem.
\begin{teorema}\labl{st-form-gevrey} Let $\ell:\NN^d\to\RR_+$ an injective linear form, $P\in\OO\setminus\{0\}$,
$P(\bo)=0$ and let 
$T_\ell$ be defined by Lemma \reff{defT-poly} resp.\ Theorem \reff{T-sector-germ}.
Let $\Pi$ be a $P$-sector, $f\in\OO(\Pi)$, $\hat f\in\hat\OO$ and $s>0$.
Then $f$ has $\hat f$ as $P$-$s$-asymptotic expansion on $\Pi$ if and only if
there exist $\rho>0$ such that $T_\ell \hat f=\sum_{n=0}^\infty g_n t^n\in\OO_b(D(\bo;\rho))\lf t\rf $ is a formal $s$-Gevrey series and one of the following two
equivalent conditions holds:
\be\item there exist constants $K$ and $A$ such that
$$\norm{f(\bx)-\sum_{n=0}^{N-1}g_n(\bx)P(\bx)^n}\leq
   K A^N \Gamma(sN+1)\norm{P(\bx)}^N$$
for $\bx\in\Pi\cap D(\bo;\rho),N\in\NN.$
\item The function $T_\ell f$
from Theorem \reff{T-sector-germ} is defined on $V(a,b;\sigma)\times D(\bo,\rho)\to\CC$ for some
positive $\sigma$ and satisfies
$$T_\ell f\sim_s T_\ell \hat f\mbox{ as }V(a,b,\sigma)\ni t\to0.$$
\ee
\end{teorema}

\begin{nota} \benot \item In the case of $P(\bx)=\bx^\bal$, statement (1) agrees with the second definition
of monomial asymptotics of Gevrey type (see Definition/Proposition \reff{defprop36}).
\item As stated for general asymptotic expansions in a germ
in Remark \reff{nota-P-clos}, products  of functions having $P$-$s$-Gevrey asymptotic expansions also have $P$-$s$-asymptotic expansions. For a proof, consider $P$-$s$-asymptotic sequences $\{ f_n\}_{n\in\NN}$, $\{ g_n\}_{n\in \NN}$ satisfying the inequalities in Definition
\reff{desarrolloGevrey} for $f$, $g$, respectively. Define
    $$
    h_n=\sum_{k=0}^n \sum_{j=0}^{k} (f_j-f_{j-1})\cdot (g_{k-j}-g_{k-j-1}),
    $$
    (where $f_{-1}=g_{-1}=0$). A straightforward computation shows that
    $$
    f(\bx) \cdot g(\bx) -h_n(\bx )=(f(\bx )-f_n (\bx))\cdot g(\bx ) + \sum_{k=0}^n (f_k (\bx )-f_{k-1} (\bx ))\cdot (g(\bx )-g_{n-k} (\bx )).
   $$
Then it is first shown that $f_n (\bx )-f_{n-1} (\bx )$ and $g_n (\bx )-g_{n-1} (\bx ) $ satisfy $s$-Gevrey estimates and then
using the inequality
\begin{equation}\labl{prodgamma}{
\sum_{\nu=0}^n\Gamma((n-\nu)s+1)
\Gamma( \nu s+1)\leq K_s 
\Gamma(n\,s+1)}\end{equation}
with some constant $K_s$ independent of $n$, $s$-Gevrey bounds for the modulus of this expression are obtained.

For a proof of (\reff{prodgamma}), first use that the $\Gamma$-function is logarithmically convex and 
hence $\Gamma(x')\Gamma(y')\leq \Gamma(x)\Gamma(y)$ if $0<x<x'\leq y'<y$ with $x'+y'=x+y$.
Therefore, if $N$ is the smallest integer $N\geq\frac1s$ and 
$n>2N$ then the left hand side of (\reff{prodgamma}) is 
smaller than
$2N\Gamma(ns+1)+(n-2N+1)\Gamma(ns-Ns+1)\Gamma(Ns+1)$. Since $1<ns-Ns+1\leq ns$ and $ns>2$, we obtain
the bound $2N\Gamma(ns+1)+n\Gamma(ns)\Gamma(Ns+1)\leq K_s\Gamma(ns+1)$ with 
$K_s=2N+\frac1s\Gamma(Ns+1)$ provided $n>2N$. For the remaining finitely many cases it is sufficent
to increase $K_s$ if necessary.

The compatibility of  $P$-$s$-asymptotic expansions with partial derivatives could also be shown using Definition \reff{desarrolloGevrey}, but we prefer to prove it using our generalization \reff{ramissibuya-poly} of the Ramis-Sibuya Theorem.

\item As is Remark \ref{caracterGevreyseries}, the series $T\hat f$ turns out to be $s$-Gevrey if we only suppose that
it is the $s$-Gevrey asymptotic expansion of $Tf$ in the sense of statement (2) of the Theorem.
\ee
\end{nota}
\noindent\begin{proof}[Proof of Theorem \reff{st-form-gevrey}] Suppose $\{ f_n \}_n$ is a
$P$-$s$-asymptotic sequence for $\hat f$ satisfying the conditions
of Definition \reff{desarrolloGevrey} on a certain $P$-sector $\Pi$
and a certain polydisk $D(\bo,\rho)$.
As in the proof of Theorem \reff{st-form}, we choose a positive $\mu$ such that
$D_\mu \subset D(\bo,\rho)$ for the set $D_\mu $ of Lemma \reff{WDT-Ob} and write
(the restrictions to $D_\mu $)
$$
{f_n}= \sum_{m=0}^{n-1} g_m P^m + Q^n(f_n)P^n,
$$
where $g_m=RQ^m ({f_{m+1}})=RQ^m ({f_{\nu}})$ for all $\nu>m$. Then
$$\begin{array}{rcl}
\norm{f_n(\bx)-\dis\sum_{\nu=0}^{n-1} g_\nu(\bx) P^{\nu}(\bx)}& = &
\norm{Q^n(f_n)(\bx)P(\bx)^n}\\&\leq &K\,\norma Q^n A^n \Gamma(sn+1)\norm{P(\bx)}^n
\end{array}$$
for $n\in\NN$ and $\bx\in\Pi\cap D_\mu $. Together with condition (2) of Definition
\reff{desarrolloGevrey}, this proves (1).

Application of Theorem \reff{T-sector-germ} with $K(t)=K\,A^N\,\Gamma(sN+1)\,t^N$
to the inequalities in (1) yields (2).

The proof of the converses is again trivial.
\end{proof}

The same proof shows the following characterisation of $P$-$s$-asymptotic sequences.
\begin{defprop}\labl{PsGevrey}Let $\ell:\NN^d\to\RR_+$ an injective linear form
and let $\Delta_\ell(P)$ be defined by (\reff{delta}).
Let $P\in\OO\setminus \{ \bo\}$, $P(\bo)=0$, and $\hat f\in\hat\OO$ a formal series.
Then the following statements are equivalent:
\be\item There exists a $P$-$s$-asymptotic sequence for $\hat f$ in the sense of
Definition \reff{desarrolloGevrey}.
\item There exist $\rho>0$ and a sequence $\{g_n\}_n$ in $\OO_b(D(\bo;\rho))^{\NN}$
with $J(g_n)\in\Delta(P)$ for all $n$ such that $\dis\hat f=\dis\sum_{n=0}^\infty J(g_n)J(P)^n$
and $T_\ell\hat f=\dis\sum_{n=0}^\infty g_nt^n\in\OO_b(D(\bo;\rho))\lf t\rf $ is a formal Gevrey-$s$ series.
\ee
If one and hence both statements are true, then $\hat f$ is called $P$-$s$-Gevrey.
\end{defprop}

As for general $P$-asymptotic expansions, $P$-$s$-asymptotic expansions are compatible with blow-ups.
\begin{propo} \labl{gevrey-blu}
Consider $P\in \OO (D(\bo ; \rho ))\setminus\{0\}$, $P(\bo)=0, $$\Pi_P=\Pi_P (a,b;r)$, and $f\in \OO (\Pi_P)$. Then, $f$ has a $P$-$s$-Gevrey asymptotic expansion on $\Pi_P$ if and only if for every $\xi \in \PC$, there exists a positive $r_\xi$ such that
$f\circ b_{\xi}$ has a $P\circ b_{\xi}$-$s$-Gevrey asymptotic expansion on $\Pi_{P\circ b_{\xi}} (a,b;r_{\xi})$.
\end{propo}

\begin{proof}
We follow the proof of Theorem \reff{asymp-blu}, taking $C_n^{(\xi)}= C_{\xi} A_{\xi}^n \Gamma (ns+1)$. By the compacity of $\PC$, it suffices 
to consider only a finite number of points in $\PC$, say $\xi_0, \xi_1,\ldots , \xi_K$, so we can omit the dependence on $\xi$ of the above constants, taking the maximum of their values. Using the notation of Theorem \reff{asymp-blu}, there exists a constant $L$ such that
$$
\norma{H_n^{(\xi_i,\xi_j)}} \leq L\cdot 2CA^n\Gamma (ns+1),
$$
and $\tilde{C}$ such that
$$
\norma{L_n^{(\xi_j)}}\leq \tilde{C} L\cdot 2CA^n\Gamma (ns+1).
$$
Similarly, the constants in (\reff{majFn}) and (\reff{bn}) are of $s$-Gevrey type, i.e.\ the construction in the proof
of  Theorem \reff{asymp-blu} yields a $P$-$s$-Gevrey asymptotic expansion if the given expansions
of $f\circ b_\xi$ are $P\circ b_\xi$-$s$-Gevrey asymptotic expansions for every $\xi$ in  $\PC$.
\end{proof}

As for ordinary $P$-asymptotic expansions, the above Proposition can be improved.
Its assumption is in fact only necessary for a finite number of $\xi$ in $\PC$.
\begin{teorema}\labl{gevrey-blu+}Consider $\LL,\,\ell,\,h,\,H_\LL,\mcZ_\LL$ as in Lemma \reff{FxiG} and
put $\mcZ=\mcZ_\LL$ if $H_\LL$ is not a constant, $\mcZ=\{\xi\}$ with arbitrary $\xi\in\PC$ otherwise. Let $s>0.$
If $f$ is holomorphic  and bounded on $\Pi=\Pi_P(a,b,R)$ such that for $\xi\in\mcZ$, the composition
$f\circ b_\xi$ has some $P\circ b_\xi$-$s$-asymptotic expansion $\hat g_\xi$ on
$\Pi_{P\circ b_\xi}(a,b,R)$ then
$f$ has a $P$-$s$-asymptotic expansion $\hat f$ on $\Pi$ such that $\hat f\circ b_\xi=\hat g_\xi$
for $\xi\in\mcZ$.
\end{teorema}
\begin{proof}
We follow the proof of Theorem \reff{asymp-blu+} and essentially add Gevrey estimates.
We again can assume that 
$\infty\in\mcZ$. We first carry Corollary \reff{coroclose} over to $P$-$s$-Gevrey asymptotics.
It suffices to use $K_n=C\,A^n\,\Gamma(sn+1)$ in its proof. As we will need this statement again,
we write it down as a lemma.
\begin{lema}\labl{gevreyclose}  Suppose that $g$ is holomorphic on $\Pi_{P\circ b_\xi}(a,b,R)$ and has some 
$P\circ b_\xi$-$s$-Gevrey asymptotic expansion $\hat g_\xi$ on  it.
Then for  $\zeta\in\PC$ close the $\xi$, the composition $\hat g_\xi(\phi_\xi\circ\phi_\zeta^{-1}(\bv))$ 
is well defined and it is the $P\circ b_\zeta$-$s$-Gevrey asymptotic expansion of $g\circ(\phi_\xi\circ\phi_\zeta^{-1})$ on 
$\Pi_{P\circ b_\zeta}(a,b,\rho_\zeta)$ for sufficiently small positive $\rho_\zeta$.

In particular, if $f$ holomorphic on $\Pi_P(a,b,R)$ and $\xi\in\PC$ such that $f\circ b_\xi=g$ satisfies the assumption,
then the statement holds for $g\circ(\phi_\xi\circ\phi_\zeta^{-1})=f\circ b_\zeta$ and $\zeta$ close to $\xi$.
\end{lema}

We obtain that
$f\circ b_\xi$ has a $P$-$s$-Gevrey asymptotic expansion for $\xi$ large or $\xi$ close to
$\mcZ'=\mcZ\setminus\{\infty\}$.

In a second step, we use again Lemma \reff{FxiG} on some domain $\mcD=D(0,R)\setminus\bigcup_{\chi\in\mcZ'}\cl(D(\chi,r))$.
It yields $\sigma,\rho>0$ and a holomorphic function $G_\LL:V(a,b,\sigma)\times \Omega\to\CC$,
$\Omega=\mcD\times D(0;\rho)^{d-1}$, such that for $\xi\in\mcD$, the function 
$T_\ell(f\circ b_\xi)(t,\bv)$ of Theorem \reff{T-sector-germ}. 
satisfies 
\begin{equation}\labl{FG2} T_\ell(f\circ b_\xi)(t,\bv)=G_\LL(t,(v_1+\xi,\bv'))\mbox{ for }
t\in V(a,b,\sigma)\mbox{ and small }\bv.\end{equation}

It had been shown in the proof of Theorem \reff{T-sector-germ} that $G_\LL(t,\bv)$ has a uniform asymptotic
expansion on $\Omega$ as $V(a,b,\rho)\ni t\to0$. Denote it by $G_\LL(t,\bv)\sim\sum_{n=0}^\infty g_n(\bv)t^n$.
As $T_\ell(f\circ b_\xi)(t,\bv)$ has a Gevrey asymptotic expansion by assumption and
Theorem \reff{st-form-gevrey}, provided $\xi$ is close to the boundary of $\mcD$, we obtain using (\reff{FG2}) and the compactness
of $\partial\mcD$ again that there are positive constants $K,A$ such that 
\begin{equation}\labl{gevreyG}\norm{G_\LL(t,\bv)-\sum_{n=0}^{N-1} g_n(\bv)t^n}\leq K\, A^N\,\Gamma(sN+1)\,\norm t^N
\end{equation}
for all $n\in\NN$, all $t\in V(a,b,\rho)$ and all $\bv$ such that $\norm{\bv'}<\rho$ and 
$\dist(v_1,\partial \mcD)<\rho$. 
Here we use again the maximum modulus principle in the variable $v_1$ on the domain $\mcD$ and obtain
that (\reff{gevreyG}) is valid for all the above $n,t$ and all $\bv\in\Omega.$
This shows using (\reff{FG2}) and Theorem \reff{st-form-gevrey} again that $f\circ b_\xi$
has a $P\circ b_\xi$-$s$-Gevrey asymptotic expansion for all $\xi\in\mcD$. As this is known for
the remaining $\xi$ already, we have it for all $\xi\in\PC$. Proposition
\reff{gevrey-blu} allows us to conclude.
\end{proof}

We can also carry over the statements of Section \reff{blu} concerning  $P$-asymptotic series to $P$-$s$-Gevrey 
asymptotic series. We first do so for Proposition \reff{aserie-blu}.
\begin{propo}\labl{gserie-blu}
Let $\hat f\in\hat\OO$ be such that for all $\xi\in\PC$, the composition
$\hat f\circ b_\xi$ is a $P\circ b_\xi$-$s$-Gevrey asymptotic series. Then $\hat f$
is a $P$-$s$-Gevrey asymptotic series.
\end{propo}
\begin{proof} The proof of Proposition \reff{aserie-blu} carries over unchanged. Just observe that
we have here special constants $K_n^{(\xi)}=C_\xi\,A_\xi^n\,\Gamma(sn+1)$ with $C_\xi,\,A_\xi$
independent of $n$ and that by (\reff{majFn2}), we obtain such Gevrey estimates also for the
$F_n^j$ and hence for the functions $f_n$.
\end{proof}
As for Proposition \reff{aserie-blu}, the above Proposition can be improved insofar as it is sufficient to
assume the Gevrey character of $\hat f\circ b_\xi$ for finitely many $\xi$ only.
\begin{teorema}\labl{gserie-blu+}Consider $\LL,\,\ell,h,\,\,H_\LL,\mcZ_\LL$ and $\mcZ$ as in Theorem 
\reff{gevrey-blu+}. Let $\hat f\in\OO$ be given such that $\hat f \circ b_\xi$ is a 
$P\circ b_\xi$-$s$-Gevrey asymptotic series for $\xi\in\mcZ$. Then $\hat f$ is a
$P$-$s$-Gevrey asymptotic series.
\end{teorema}
\begin{proof}In the proof of Theorem \reff{aserie-blu+}, only the constants have to be modified:
We have $K_{n,\xi}=C_\xi\,A_\xi^n\,\Gamma(sn+1)$ with some positive $C\xi,A_\xi$ independent of $n$
and as a consequence later $K_n=C\,A^n\,\Gamma(sn+1)$ with certain positive $C,A$. Details are left to the readers.
\end{proof}

 \begin{nota}\labl{nota-gblu}
Consider the formal series $\hat f(x_1,x_2)=\sum_{n=0}^{\infty}n!\,x_2^{2n}\,(x_1x_2)^{n}$ and 
the monomial $P(x_1,x_2)=x_1x_2$ as a special germ of an analytic function.
Whatever $\ell$, we have $(T_\ell\hat f)(t)(x_1,x_2)=\sum_{n=0}^{\infty}n!\,x_2^{2n}\,t^{n}$
and therefore $\hat f$ is a $P$-1-Gevrey asymptotic series.

Whatever $\ell$, we have $\mcZ_\ell=\{0,\infty\}$. We calculate
$$\hat f\circ b_0(v_1,v_2)=\hat f(v_2,v_1v_2)=\sum_{n=0}^{\infty}n!\, v_1^{3n}\,v_2^{4n}=\sum_{n=0}^{\infty}n!\, v_1^{n}\,(v_1v_2^2)^{2n}
       \mbox{ and }P\circ b_0(v_1,v_2)=v_1v_2^2.$$
Hence $\hat f\circ b_0$ is a $P\circ b_0$-$\frac12$-Gevrey series. We also calculate
$$\hat f\circ b_\infty(v_1,v_2)=\hat f(v_1v_2,v_2)=\sum_{n=0}^{\infty}n!\, v_1^{n}\,v_2^{4n}=\sum_{n=0}^{\infty}n!\, v_2^{2n}\,(v_1v_2^2)^{n}
       \mbox{ and again }P\circ b_\infty(v_1,v_2)=v_1v_2^2.$$
Hence  $\hat f\circ b_\infty$ is $P\circ b_\infty$-$1$-Gevrey.

Using Theorem \reff{gserie-blu+}, this confirms that $\hat f$ is a $P$-1-Gevrey asymptotic series.
It also shows that we need to consider the blow-ups $\hat f\circ b_\xi$ {\em for all $\xi\in\mcZ_\ell$}
in order to conclude: Theorem \reff{gserie-blu+} seems to be sharp with respect to the number of points
to be considered in the exceptional divisor. 
\end{nota}

\green{$P$-$s$-Gevrey asymptotic expansions and series are also compatible with
ramification. 
\begin{propo} \labl{gevrey-ram}
Consider $P$ on $D(\bo;\rho)$, $\Pi=\Pi_P(\alpha,\beta;\rho)$ and $f:\Pi\to\CC$ holomorphic.
Suppose that for some integer $k\geq2$ and positive $s$, the function $f\circ r_k$, restricted to
$\Pi_k=\Pi_{P\circ r_k}(\alpha,\beta;\tilde\rho)$ with some sufficiently small $\tilde\rho$,
has some formal series $\hat g\in\hat\OO$ as $(P\circ r_k)$-$s$-Gevrey asymptotic
expansion on its domain.

Then there exists a formal series $\hat f\in\hat\OO$  that is the $P$-$s$-Gevrey asymptotic expansion of $f$
on $\Pi$ and it satisfies $\hat f\circ r_k=\hat g$.
\end{propo}
\begin{propo}\labl{gserie-ram}
Let $\hat f\in\hat\OO$, $k\in\NN$, $k\geq2$ and $s>0$ be such that  the composition
$\hat f\circ r_k$ is a $P\circ r_k$-$s$-Gevrey asymptotic series. Then $\hat f$
is a $P$-$s$-Gevrey asymptotic series.
\end{propo}
The proofs are analogous to the ones of Propositions \reff{asymp-ram} and \reff{aserie-ram}. One just has to add
Gevrey estimates.}

As for monomial Gevrey asymptotics, functions having a $P$-$s$-asymptotic expansion with vanishing
series are exponentially small.
\begin{lema} \labl{expsmallP}If $\pp$ is a $P$-sector for a certain
$P$ and if $f\in{\mathcal O}(\pp)$ has a $P$-$s$-Gevrey asymptotic
expansion where $\hat f=0$, then for all
sufficiently small $R'>0 $ there exist $C,\ B>0$ such that
on
${\tilde \pp}$
$$
\norm{f(\bx )}\leq C\cdot \exp\left( {-\frac{B}{\norm{P(\bx)}^{1/s}}}\right)\
\mbox{ for }\bx\in\pp,\ \norm\bx<R'.
$$
\end{lema}

\begin{proof} As $\hat f=0$ implies that all $g_n$ of theorem \reff{st-form-gevrey} vanish,
we have that $\norm{f(\bx)}\leq K A^N \Gamma(sN+1)\norm{P(\bx)}^N$  for
all sufficiently small $\bx\in\pp$ and for all $N\in\NN$.
Again we choose $N$ close to the optimal value
$(A\norm{P(\bx)})^{-1/s}$
and Stirling's Formula yields the statement.
\end{proof}

\noindent As a consequence of Definition/Proposition \reff{PsGevrey} we can also construct a function
that have a prescribed  $P$-$s$-Gevrey series as its $P$-$s$ asymptotic expansion.
\begin{propo}[Borel-Ritt-Gevrey]  \labl{BRG}
Let $P\in\OO\setminus\{0\}$, $P(\bo)=0$ as before, $\hat{f}$ a $P$-$s$-Gevrey series and
$\Pi=\Pi_P (a, b; r)$ a $P$-sector of opening $b-a < s\pi$. Then there exist
$\rho>0$ and
$f\in \OO (\Pi\cap D(\bo;\rho))$ such that $f$ has $\hat{f}$ as
$P$-$s$-Gevrey asymptotic expansion on $\Pi\cap D(\bo;\rho)$.
\end{propo}

\begin{proof}
Statement (2) of Definition/Proposition \reff{PsGevrey}
yields the existence of $\rho>0$ and of a sequence
$\{g_n\}_n\in\OO_b(D(\bo;\rho))^{\NN}$
with $J(g_n)\in\Delta(P)$ for all $n$ such that $\dis\hat f=\dis\sum_{n=0}^\infty J(g_n)J(P)^n$
and $\dis\sum_{n=0}^\infty g_nT^n\in\OO_b(D(\bo;\rho))\lf T\rf $ is a formal Gevrey-$s$ series.
Now let $V=V(a,b;\mu )$, where $\mu>\norm{P(\bx)}$
for all $\bx\in D(\bo;\rho)$.  By the
Borel-Ritt-Gevrey Theorem in Banach spaces (see Theorem \reff{borel-ritt-watson}), there exists
$F\in \OO({D} (\bo; r) \times V)$ having
$\sum_{m} g_m T^m$ as $s$-Gevrey asymptotic
expansion at $T=0$. The function $f$ defined by $f(\bx)=F(\bx,P(\bx))$ gives the
result.
\end{proof}

As in the classical and monomial asymptotics, functions having Gevrey asymptotic expansions in a germ
can be characterized by completing them to a family almost covering a neighborhood of $\bo$
such that  the differences of any two of
them is exponentially small if the intersection of their domain is nonempty.
In this context, covers of polydisks $D(\bo;R)$ outside the zero
set of $P$ will be important.
\begin{defin}
A $P$-cover denotes a family ${\Pi_i}_{i\in I}$, $I$ some finite set, of
$P$-sectors that covers the open set $D(\bo;R) \setminus \{ P(\bx )=0\}$ for some $R>0$. 
Given such a $P$-cover ${\mathcal P}=\{\Pi_i\}_{i\in I}$, a
$P$-$k$-quasifunction on $\mathcal P$ is a
family $(f_i)_{i\in I}$ of bounded holomorphic functions $f_i\in \OO_b (\Pi_i)$, such
that whenever $\Pi_i\cap \Pi_j\neq \emptyset$ there exist constants $C$ and $B$
satisfying
$$
\norm{f_i(\bx)-f_j(\bx)}\leq C\cdot \exp \left(
-\frac{B}{\norm{P(\bx)}^k}\right)\mbox{ for all }\bx\in\Pi_i\cap \Pi_j.
$$
\end{defin}

Proposition \reff{BRG} and Lemma \reff{expsmallP} imply that a function having an $s$-Gevrey asymptotic expansion in a germ
can be completed to a $P$-$k$-quasifunction, $k=1/s$.
\begin{propo}\labl{existcover}
Consider a holomorphic function $f\in \OO (\Pi)$ having a $P$-$s$-Gevrey asymptotic expansion
on $\Pi$. Then there exist $\rho>0$, a $P$-cover $\Pi_1,
\Pi_2,\ldots , \Pi_r$ of $D(\bo,\rho)\setminus\{P=0\}$ with $\Pi_1=D(\bo,\rho)\cap\Pi$
and a $P$-$k$-quasifunction $(f_1,f_2,\ldots ,f_r) $ on it such that $k=1/s$ and
$f_1=f\mid_{\Pi_1}$.
\end{propo}
\begin{proof}
We need to assume that the opening of each $\Pi_i$, $i>1$, is
not greater than $s\pi$. Then there exist $f_i\in \OO (\Pi_i)$
($i>1$) having the same $P$-$s$-asymptotic expansion as $f$;
this is possible thanks to Proposition \reff{BRG} provided $\Pi_i$
are contained in a suffiently small polydisk.
If $\Pi_i\cap \Pi_j\neq \emptyset$ then $f_i-f_j$ have a $P$-$s$-asymptotic expansion
on it and $(f_i-f_j)(\bx)\sim0$. Lemma \reff{expsmallP} now implies that
the $f_i$ can be combined to a $P$-$k$-quasifunction, $k=1/s$.
\end{proof}

As for classical and monomial asymptotics, the converse is also true. This result, which
generalises the classical Ramis-Sibuya Theorem (Theorem \reff{ramissibuya-classic}) and the version 
\reff{ramissibuya-monom} for monomial asymptotics, are the most
important means to establish the existence of $P$-$s$-Gevrey asymptotic expansions.
\begin{teorema}\labl{ramissibuya-poly}
Suppose that
the $P$-sectors $\pp_j= \pp_P(a_j,b_j;r), 1\leq j\leq m$,
form a cover of $D(\bo;r)\setminus\{\bx;\  P(\bx)=0\}$.
Given $f_j: \pp_j \rightarrow  \CC$, $j=1,\ldots ,m$, bounded and analytic, assume
that there exists $\gamma>0$ such that
for every couple $(j_1,j_2)$
\begin{equation}\labl{expabove-germ}
\norm{f_{j_1}(\bx)-f_{j_2}(\bx)} =
O(\exp(-\gamma/|P(\bx)|^{1/s}))
\end{equation}
for $\bx\in\pp_{j_1}\cap \pp_{j_2}$, provided  $\pp_{j_1}\cap \pp_{j_2}\neq \emptyset$.
Then the functions $f_j$ have $P$-$s$-Gevrey asymptotic expansions
with a common right hand side.
\end{teorema}
\begin{proof} Consider some injective linear functional $\ell:\NN^d\to\RR_+$ and the
operators $T_\ell $ corresponding to it for the $P$-sectors $\pp_j$ and their nonempty
intersections according to Theorem \reff{T-sector-germ}. Here the fact that the operators are
independent of the $P$-sector in the sense of remark \reff{notaTl}, (3) is important.
Without loss of generality, we can assume that the constants $\rho,\sigma,L$ of that theorem
are the same for all these finitely many sectors. For convenience, we
identify the Banach spaces $\OO_b(V(a,b;\sigma)\times D(\bo;\rho),\CC)$ with
$\OO_b(V(a,b;\sigma),\OO_b(D(\bo;\rho),\CC))$ in the usual way.
The inequalities (\reff{expabove-germ}) imply that
$$\norma{(T_\ell f_{j_1})(t)-(T_\ell f_{j_2})(t)}_{\OO_b(D(\bo;\rho),\CC)}=
O(\tfrac1{\norm t}\exp(-\gamma \norm{t}^{-1/s}))$$
for $t\in V(a_{j_1},b_{j_1};\sigma)\cap V(a_{j_2},b_{j_2};r)$ provided this intersection
is nonempty. Here the classical Ramis-Sibuya theorem \reff{ramissibuya-classic}
applies and yields that the functions $T_\ell f_{j}$ have common Gevrey-$s$ asymptotic
expansions. We conclude using Theorem \reff{st-form-gevrey}.
\end{proof}

In the same way, the complement 
to the classical Ramis-Sibuya Theorem (Theorem \reff{compl-classic}) will now be carried over to Gevrey 
asymptotics in a germ.
\begin{teorema}\labl{compl-germ}
Suppose that the $P$-sectors $\pp_j= \pp_P(a_j,b_j;r),\, 1\leq j\leq m$,  form
a $P$-cover.
For couples $(j_1,j_2)$ with $\pp_{j_1}\cap \pp_{j_2}\neq\emptyset$, let
holomorphic $d_{j_1,j_2}:\pp_{j_1}\cap \pp_{j_2}: \rightarrow E$ be given
that satisfy the cocycle condition
$d_{j_1,j_2}+d_{j_2,j_3}=d_{j_1,j_3}$ whenever $\pp_{j_1}\cap \pp_{j_2}\cap \pp_{j_3}\neq\emptyset$
and estimates
\begin{equation}
\norm{d_{j_1,j_2}(\bx)} = O(\exp(-\gamma/|P(\bx)|^{1/s}))
\end{equation}
for $j_1,j_2\in\{1,\ldots ,m\}$ and $x\in \pp_{j_1}\cap \pp_{j_2}$ with some  constants
$s,\gamma>0$.

Then there exist $\rho>0$ and bounded holomorphic functions $f_j:\pp_j\cap D(\bo;\rho)\to E$
such that
$d_{j_1,j_2}\mid_{\pp_{j_1}\cap \pp_{j_2}\cap D(\bo;\rho)} =f_{j_1}-f_{j_2}$
whenever $\pp_{j_1}\cap \pp_{j_2}\neq\emptyset$; moreover
the functions $f_j$ have $P$-$s$-Gevrey asymptotic expansions with a common right hand side.
\end{teorema}

\begin{proof} With the same notation as before, take an injective linear functional $\ell: \NN^d \rightarrow \RR_{+}$, and the operators $T_{\ell}$. If $\Pi_{j_1} \cap \Pi_{j_2}= \Pi_P (a_{j_1j_2}, b_{j_1j_2};r)$, consider $T_{\ell } d_{j_1j_2} (t,\bx )\in \OO (V (a_{j_1j_2}, b_{j_1j_2};\sigma ) \times D(\bo;\rho ))$, for some $\sigma$, $\rho$. If $\Pi_{j_1}\cap \Pi_{j_2}\cap \Pi_{j_3}\neq \emptyset$ 
then we have that $(T_{\ell} d_{j_1j_2} )(t,\bx ) + (T_{\ell } d_{j_2j_3})(t,\bx ) = (T_{\ell } d_{j_1j_3} )(t,\bx )$
because the operators $T_\ell$ are independent of the $P$-sector (Remark \reff{notaTl}). 
The hypotheses of Lemma \reff{compl-classic} are verified, and hence, there exist holomorphic bounded functions 
$F_{j} : V(a_j,b_j; \sigma ) \times D(\bo; \rho ) \rightarrow \CC $ such that 
$F_{{j_1}} -F_{{j_2}} = T_{\ell } d_{j_1j_2}$ whenever $V(a_{j_1}, b_{j_1};  \sigma) \cap V(a_{j_2}, b_{j_2};  \sigma)
\neq \emptyset$. The functions $f_j (\bx )= F_{j} (P(\bx ), \bx )$ satisfy the statement of the Theorem.
\end{proof}

\subsection{Summability with respect to a germ} In the sequel we still consider some
analytic germ $P$ and suppose that $P\in\OO_b(D(\bo,R))\setminus\{0\}$, $P(\bo)=0$. The existence of a summability
result in a germ is based on the following Watson's lemma. It generalizes the theorem for monomial expansions 
(Theorem \reff{Watson}), 
and as before it is easily established
by carrying over the classical version Theorem \reff{borel-ritt-watson} (3) using Theorem
\reff{T-sector-germ}.
\begin{lema}\labl{watson-germ}
Let $\pp=\pp_P (a,b;R)$ a sector in $P$ with $b-a>s\pi$
and suppose that $f\in {\mathcal O}(\pp)$ has $\hat f=0$ as its
$P$-$s$-Gevrey asymptotic expansion. Then $f\equiv 0$.
\end{lema}
\noindent Now the following definition makes sense.
\begin{defin}\labl{P-sumable}
Let $\Pi=\Pi_P(a,b;r)$ be a $P$-sector with $b-a>s\pi$, $k=\frac1s$, and
$\hat{f}\in\hat \OO$.
We will
say that $\hat{f}$ is $P$-$k$-summable in $\Pi$ if there exists $f\in \OO (\Pi )$ having $\hat{f}$ as $P$-$s$-Gevrey asymptotic expansion. In this situation, $f$ is called the $P$-$k$-sum of $\hat{f}$ in $\Pi$ (and it is uniquely determined by Lemma \reff{watson-germ}).

$\hat{f}$ is called $P$-$k$-summable in direction $\theta\in \RR$ if there exists a $P$-sector $\Pi$ as before, bisected by $\theta$,
with $b-a>s\pi$, and such that $\hat{f}$ is $P$-$k$-summable in $\Pi$. $\hat{f}$ is called $P$-$k$-summable if it is summable in every direction $\theta \in \RR$ but a finite number$\mod\ 2\pi$.
\end{defin}

The above notion of $P$-$k$-summability in a direction $\theta$
does not indicate how to obtain a sum from a given series.
Theorem \reff{st-form-gevrey} 
allows us to carry over the classical characterization 
using Laplace integrals (see Proposition \reff{laplace}) to the new concept.
\begin{propo}\labl{laplace-poly}
Let $s=1/k$. Given a $P$-$s$-Gevrey series $\hat f(x)\in\hat\OO$, it is $P$-$k$-summable
in a direction $\theta$ with $P$-$k$-sum $f$ if and only if there exist $\rho>0$ and
a formal Gevrey series $\hat F(t)=\sum_{n=0}^\infty g_nt^n\in\OO_b(D(\bo;\rho))\lf t\rf _s$ with 
$\hat F(P(\bx))(\bx):=\sum_{n=0}^{\infty}J(g_n)(\bx)J(P)(\bx)^n = \hat f(\bx)$ and the following properties: 
\benot\item The formal Borel transform $G(\bx,\tau)=\sum {g_n}(\bx)\tau^n/{\Gamma(1+n/k)}$ 
of $\hat F(t)$ is analytic in a neighborhood of the origin,
\item the function $G$ can be continued analytically with to some domain $D(\bo;\rho)\times V(\theta-\delta,\theta+\delta;\infty)$,
\item it has exponential growth there, i.e.\ there are
positive constants such that
$$
\norm{G(\bx,\tau)}\leq C\cdot \exp\left({A}{|\tau|^{k}}\right) \mbox{ for }(\bx,\tau)\in S;
$$
hence the Laplace integral
$F (\bx,t) =
k\,t^{-k} \int_{\arg \tau=\tilde \theta} e^{-{\tau^k}/{t ^k}}
\,{G} (\bx,\tau)\,{\tau^{k-1}} d\tau$ converges for $\bx\in D(\bo;\rho)$  and 
$t$ in a certain sector
$V=V(\theta-\frac\pi{2k}-\frac{\tilde\delta} k,\theta+\frac\pi{2k}+\frac{\tilde\delta} k;r),$
$0<\tilde\delta<\delta$, and suitably chosen $\tilde\theta$ close to $\theta$;
\item finally $f(\bx)=F(\bx,P(\bx))$ which has $\hat f(\bx)$ as its $P$-$s$-Gevrey asymptotic expansion
on some $\Pi$-sector  $\Pi_P(\theta-{\frac\pi{2k}}-\frac{\tilde\delta} k,\theta+\frac\pi{2k}+\frac{\tilde\delta} k;\tilde \rho)$, $\tilde\rho>0$.
\ee\end{propo}

As the concepts of asymptotic expansion in a germ and Gevrey asymptotic expansions in a germ,  
the notion of summability with respect to a germ behaves well under blow-ups. We first show
\begin{propo} \labl{sumable-blu}
Consider a $P$-sector $\Pi=\Pi_P (a,b;r)$ with $b-a>s\pi$, and $\hat{f}\in \hat{\OO}$, such that for every $\xi\in \PC$, the series $\hat{g}_{\xi} = \hat{f}\circ b_{\xi}$ is $P\circ b_{\xi}$-$k$-summable in $\Pi_{\xi}:= \Pi_{P\circ b_{\xi}} (a,b;r_{\xi})$. Then, $\hat{f}$ is $P$-$k$-summable in $\Pi$.

A formal series $\hat{f}\in\hat{\OO}$ is $P$-$k$-summable in a direction $d$ if and only if for every $\xi\in \PC$, $\hat{f}\circ b_{\xi}$ is $P\circ b_{\xi}$-$k$-summable in direction $d$.
\end{propo}
\begin{proof}
 We only prove the second statement; the proof of the first is analogous. 
Also, we only prove the nontrivial implication.

Suppose that for every  $\xi\in \PC$, the series $\hat{f}\circ b_{\xi}$ is 
$P\circ b_{\xi}$-$k$-summable on some $P\circ b_\xi$-sector 
$\Pi_\xi=\Pi_{P\circ b_\xi} (d-\varphi_\xi,d+\varphi_\xi;r_\xi)$ with $\varphi_\xi>\pi/k$.
This means that for every $\xi\in\PC$, there exist uniquely determined functions 
$g_\xi:\Pi_\xi\to\CC$ that have $\hat f\circ b_\xi$ as their $P\circ b_{\xi}$-$s$-Gevrey 
asymptotic expansion, $s=1/k$. 
%

With the notation of Subsection \reff{normal}, let 
$U_\xi=\phi_\xi^{-1}(D(\bo,\rho_\xi))\subset M$, a neighborhood of $(\xi,\bo)\in M$.
Then $G_\xi=g_\xi\circ \phi_\xi$ are defined on 
$U_\xi\cap \{p\mid\norm{\arg (P(b(p)))-d}<\varphi_\xi\}$. 

Observe that by Lemma \reff{gevreyclose}, the composition $g_\xi\circ(\phi_\xi\circ\phi_\chi^{-1})$
has the ${P\circ b_\chi}$-$s$-Gevrey asymptotic expansion $\hat f\circ b_\chi$ on
$\Pi_{P\circ b_\chi}(d-\varphi_\xi,d+\varphi_\xi,\mu)$ and hence it is the ${P\circ b_\chi}$-$k$-sum 
of $\hat f\circ b_\chi$ on this sector for some small $\mu$ if $\chi$ is sufficiently
close to $\xi$. Without loss in generality we can assume that this is already the case for $\chi\in U_\xi$.
By Watson's Lemma \reff{watson-germ}, we then have
\begin{equation}\labl{coincg}g_\xi\circ(\phi_\xi\circ\phi_\chi^{-1})(\bv)=g_\chi(\bv)
\end{equation}
for $\chi\in U_\xi$ and small $\bv$ with $\norm{\arg (P\circ b_\chi )(\bv)-d}<\min(\varphi_\xi,\varphi_\chi)$.

Using this property with some $\chi\in U_\xi\cap U_\zeta$ and extending it because of the analyticity
of the functions, we obtain
\begin{equation}\labl{coinci}
G_\xi(p)=G_\zeta(p)\mbox{ for }\xi,\zeta\in\PC,\ p\in U_\xi\cap U_\zeta
\mbox{ with }\norm{\arg (P(b(p)))-d}<\min(\varphi_\xi,\varphi_\zeta).
\end{equation}
Using the compactness of $\PC$ as before, a finite number of members
of the family $U_\xi$, $\xi\in\PC$, covers the 
exceptional divisor $E=\PC\times\{\bo\}$ in $M$, say $E\subset U= \cup_{i=1,\ldots,n}U_{\xi_i}$.
Then property (\reff{coinci}) allows us to define $G$ for $p\in U$ with
$\norm{\arg (P(b(p)))-d}<\varphi=\min_i \varphi_{\xi_i}$ by $G(p)=G_\xi(p)$ if
$p\in U_{\xi}$, $\norm{\arg (P(b(p)))-d}<\varphi$.
In turn, we obtain a function $g$ defined for small $\bx\in \CC^d$ with 
$\norm{\arg (P(\bx))-d}<\varphi$ by setting $g(b(p))=G(p)$.

Then $g$ has $\hat f$ as $P$-$s$-Gevrey asymptotic expansion on 
$\Pi_P(d-\varphi,d+\varphi;\rho)$ for sufficiently small positive $\rho$. Indeed,
using property (\reff{coinci}), we find that
for every $\xi\in\PC$, the function  $g\circ b_\xi=G_\xi\circ\phi_\xi^{-1}$ is 
the restriction of
$g_\xi$ to $\Pi_{P\circ b_\xi}(d-\varphi,d+\varphi;\rho)$ and it has
$\hat f\circ b_\xi$ as its $P\circ b_\xi$-$s$-Gevrey asymptotic expansion.
By Theorem  \reff{gevrey-blu}, $g$ has some
$P$-$s$-Gevrey asymptotic expansion $\hat g$ on $\Pi_P(d-\varphi,d+\varphi;\rho)$.
Obviously, we have $\hat g\circ b_\xi=\hat f\circ b_\xi$ for every $\xi\in\PC$
and thus $\hat g=\hat f$. This means that $g$ is the $P$-$k$-sum of $\hat f$ in direction
$d$.
\end{proof}

As for Propositions \reff{gevrey-blu} and \reff{gserie-blu}, the above Proposition can be improved;
again, it is sufficient to assume summability of $\hat f\circ b_\xi$ for a finite number of well chosen
$\xi.$ 
\begin{teorema}\labl{sumable-blu+}Consider $\LL,\,\ell,h,\,\,H_\LL,\mcZ_\LL$ and $\mcZ$ as in Theorem 
\reff{gevrey-blu+}. Let $\hat f\in\hat\OO$ be given such that $\hat f \circ b_\xi$ is 
$P\circ b_\xi$-$k$-summable in the direction $d$ for $\xi\in\mcZ$. 
Then $\hat f$ is also $P$-$k$-summable in direction $d$.
\end{teorema}
\begin{coro}Consider $\LL,\,\ell,h,\,\,H_\LL,\mcZ_\LL$ and $\mcZ$ as above and $\hat f\in\hat\OO$.
If $\hat f \circ b_\xi$ is  $P\circ b_\xi$-$k$-summable for $\xi\in\mcZ$
then $\hat f$ is $P$-$k$-summable.\end{coro}
\begin{proof}By the Theorem, the singular directions of $\hat f$, i.e.\ the directions $d$ for which it is not
$P$-$k$-summable, are contained in the finite union of the sets of exceptional
directions for $\hat f \circ b_\xi$, $\xi\in\mcZ$.\end{proof}
\begin{proof}[Proof of the Theorem] 
We can again assume that $\infty\in\mcZ$; otherwise we proceed as in the beginning of the proof of
Theorem \reff{asymp-blu+}.

As summable series are also Gevrey, application of Theorem \reff{gserie-blu+} yields that $\hat f$
is a $P$-$s$-Gevrey series, $s=1/k$. As a consequence $\hat f\circ b_\xi$ are $P\circ b_\xi$-$s$-Gevrey series
for every $\xi\in\PC$.

By Lemma \reff{hatfhatG} there exists a formal series
$\hat G(t,\bv)=\sum_{n=0}^\infty \hat g_n(\bv)t^n\in\BB\lf t\rf $, 
$\BB=\CC[v_1,1/H_\LL(1,v_1)]\lf \bv'\rf $
with the following property. For all $\xi\in\CC\setminus\mcZ_\LL$, the series 
$$T_\ell(\hat f\circ b_\xi)(t,\bv)=\sum_{n=0}^\infty \hat f_{n\xi}(\bv)t^n\in\Delta_\ell(P\circ b_\xi)\lf t\rf $$
of Lemma \reff{defT-poly} applied with $P\circ b_\xi$ in place of $P$ satisfies
\begin{equation}\labl{hatfn2}\hat f_{n\xi}(\bv)=J_\xi(\hat g_n)(\bv)\mbox{ for }n\in\NN.
\end{equation}
Since all $\hat f\circ b_\xi$ are $P\circ b_\xi$-$s$-Gevrey series, we obtain from Definition/Proposition \reff{PsGevrey} that the series
$\hat f_{n,\xi}(\bv)$ are convergent for all $n$ and $\xi\in\CC\setminus \mcZ_\LL$, 
that for every $\xi\in\CC\setminus \mcZ_\LL$
there exists $\rho_\xi>0$ such that $\hat f_{n,\xi}(\bv)$ defines an element of $\OO_b(D(\bo,\rho_\xi))$
that we denote $f_{n,\xi}$ and there exist $K_\xi,A_\xi>0$ such that for all $n$
\begin{equation}\labl{majfn}\norm{f_{n,\xi}(\bv)}\leq K_\xi\, A_\xi^n\,\Gamma(sn+1)\mbox{ for }\norm\bv<\rho_\xi.
\end{equation}
We can now define $g_n$ by $g_n(\bv)=f_{n\xi}(v_1-\xi,\bv')$ if $\norm{v_1-\xi}$ and $\norm{\bv'}$ are sufficiently
small. The value of $g_n(\bv)$ does not depend on $\xi$ because of (\reff{hatfn2}).
This defines functions $g_n$ on some common neighborhood $\Omega$ of $(\CC\setminus\mcZ_\LL)\times\{\bo\}\subset\CC^d$.
By (\reff{majfn}), the formal Borel transform $\tilde G(\tau,\bv)=\sum_{n=0}^\infty g_n(\bv)\tau^n/\Gamma(sn+1)$
defines a holomorphic function $\tilde G$ on some neighborhood of $\Omega\times\{0\}$ in $\CC^{d+1}$.

By assumption and Lemma \reff{gevreyclose}, $\hat f\circ b_\xi$ is $P\circ b_\xi$-summable for all $\xi$
large and $\xi$ close to $\mcZ'=\mcZ\setminus\{\infty\}$.
Consider again the domain 
$\mcD=D(0,R)\setminus\bigcup_{\chi\in\mcZ'}\cl(D(\chi,r))$
where $r,R$ were chosen such that summability holds for $\norm\xi>R/2$ or
$\dist(\xi,\mcZ')<2r$.
By Proposition \reff{laplace-poly}, (\reff{hatfn2}) and compactness, we can find positive $\rho,\delta$ such that
$\tilde G(\tau,\bv)$ can be continued analytically  to the set $W\times A$ where $W$ is the set of all
$\tau$ with $\tau\in D(0,\rho)\cup V(d-\delta,d+\delta,\infty)$, $A$ the union of ``tori'' of all $\bv$
with $\dist(v_1,\partial \mcD)<\rho$, $\norm{\bv'}<\rho$.
Also by Proposition \reff{laplace-poly}, $\tilde G$ has at most exponential growth as $\tau\to\infty$: There exist $K,L$ such that
$\norm{\tilde G(\tau,\bv)}\leq K\exp(L\norm\tau^k)$ for $\bv\in A$, $\tau\in W$, $\norm\tau\geq1$.

We can assume that also the set $\Omega'$ of all $\bv$ with $\dist(v_1,\mcD)<\rho$ and $\norm{\bv'}<\rho$ is contained in
$\Omega$ and that $\tilde G$ is holomorphic on  $D(0,\rho)\times\Omega'$.
Now $\tilde G$ is holomorphic on the union of $W\times A$ and of $D(0,\rho)\times\Omega'$.
This is a ``U-shaped'' domain and hence Hartogs' Lemma can be applied. It yields that $\tilde G$ can be continued analytically
to $W\times \Omega'$. 

{For the convenience of the reader we give a short proof using Cauchy's formula.
For $\tau\in D(0,\rho)$ and $\bv\in\Omega'$, we have
\begin{equation}\labl{cauchy7.2}
2\pi i\, \tilde G(\tau,\bv)=\left(\oint_{\norm z=R+\tilde\rho}-\sum_{\chi\in\mcZ'}\oint_{\norm{z-\chi}=r-\tilde\rho}\right)
\frac{\tilde G(\tau,(z,\bv'))}{z-v_1}\,dz
\end{equation}
for $\tilde\rho\in]0,\rho[$ sufficiently close to $\rho$ \footnote{More precisely, we must have
$\rho>\tilde \rho>\dist(v_1,\mcD)$.}. As the right hand side of (\reff{cauchy7.2}) only
uses values $\tilde G(\tau,(z,\bv'))$ where $(z,\bv')\in A$, it is defined for {\em any}
$\tau\in W$, hence the right hand side can be continued to an analytic function on  $W\times \Omega'$.
As they coincide except for a constant factor on some open subset of  $W\times \Omega'$, the same is true for $\tilde G$.
}

Now the maximum modulus principle applied in the variable $v_1$
permits to carry over the exponential estimate of $\tilde G$ to $W\times\Omega'$:  we
have $\norm{\tilde G(\tau,\bv)}\leq K\exp(L\norm\tau^k)$ for $\bv\in\Omega'$ and  $\tau\in W$, $\norm\tau\geq1$.

As  before, this implies that $\hat f\circ b_\xi$ is $P\circ b_\xi$-$k$-summable in direction $d$
for all $\xi\in\mcD$. As this is already known for the remaining $\xi\in\PC$, we have it for all
$\xi\in\PC$. We conclude with Proposition \reff{sumable-blu}.
\end{proof}

\green{As for $P$- and $P$-$s$-Gevrey asymptotic expansion, we complete the theory of
$P$-$k$-summability with statements concerning the compatibility with
ramification.
\begin{propo} \labl{sumable-ram}
Consider a $P$-sector $\Pi=\Pi_P (a,b;r)$ with $b-a>s\pi$, $k=1/s$, $\hat{f}\in \hat{\OO}$ and a positive integer $m$.
 
If the series $\hat{g}_{m} = \hat{f}\circ r_m$ is $P\circ r_m$-$k$-summable in $\Pi_m:= \Pi_{P\circ r_m} (a,b;\tilde r)$
then $\hat{f}$ is $P$-$k$-summable in $\Pi$.

$\hat{f}$ is $P$-$k$-summable in some direction $d$ 
if and only if $\hat{f}\circ r_m$ is $P\circ r_m$-$k$-summable in direction $d$.

$\hat{f}$ is $P$-$k$-summable
if and only if $\hat{f}\circ r_m$ is $P\circ r_m$-$k$-summable.
\end{propo}
\begin{proof} The first statement follows immediately from the definition of summability in a sector and
Proposition \reff{gserie-ram}. The remaining two then follow from the definitions of summability in a direction
respectively summability.
\end{proof}
}

\subsection{Consequences and further properties}

Theorem \reff{ramissibuya-poly} may be used, as in
classical asymptotics, to show properties about the composition
of functions having $P$-$s$-Gevrey asymptotic expansion, and
analytic functions. More precisely:

\begin{teorema}
Consider $P(\bx )$ as before, with $P(\bo)=0$, and
$\Pi=\Pi_P(a,b;r)$ a $P$-sector. Let $f_1(\bx),\ldots , f_n(\bx)
\in \OO (\Pi)$ be functions having series $\hat{f}_1(\bx),\ldots
, \hat{f}_n (\bx)$, respectively, as $P$-$s_i$-asymptotic
expansions, $i=1,\ldots ,n$, with $\hat{f}_i (\bo)=0$. Let $D$ be
a disk around the origin in $\CC^{d+n}$, and $F(\bx,\by
)=F(x_1,\ldots , x_d,y_1,\ldots , y_n)\in \OO (D)$. Then, if
$s=\max \{ s_1,\ldots ,s_n\}$, we have:

\be
\item $F(\bx,f_1 (\bx),\ldots , f_n (\bx))$ is defined in a
$P$-sector $\tilde{\Pi}=\Pi_P(a,b;\tilde{r})$, with
$\tilde{r}\leq r$ small enough.
\item $F(\bx,f_1(\bx), \ldots , f_n (\bx))$ has a $P$-$s$-Gevrey
asymptotic expansion in $\tilde{\Pi}$.
\ee
\end{teorema}
\begin{proof}
The conditions $\hat{f}_i (\bo )=0$ imply that
$\lim\limits_{\Pi\ni\bx \rightarrow \bo} f_i(\bx) =0$ and the
first statement follows. As a $P$-$s_i$-Gevrey asymptotic expansion 
also is a  $P$-$s$-Gevrey asymptotic expansion, we can assume that
all $s_i=s$. For simplicity of notation, we combine
$\bof(\bx)=(f_1(\bx),\ldots,f_n(\bx))$, $\hat \bof(\bx)=(\hat f_1(\bx),\ldots,\hat f_n(\bx))$.
By Proposition \reff{existcover}, there exist a $P$-cover
$\{ \Pi=\lista{\Pi}{r} \}$, with $\Pi_i= \Pi_P (a_i,b_i; \tilde{r})$,
and  functions $\bof_{i} \in \OO (\Pi;\CC^n)$, $1\leq
i\leq r$, such that $\bof_{i}(\bx)$ has $\hat{\bof}
(\bx)$ as $P$-$s$-asymptotic expansion. Consider the functions
$g_i(  \bx)=F(\bx,\bof(\bx))$, defined
on $\Pi_P (a_i,b_i;\tilde{r})$, reducing $\tilde{r}$ again if
necessary. 
If $\Pi_{j_1}\cap \Pi_{j_2}\neq \emptyset$ then
$$(g_{j_1}- g_{j_2})(\bx )=\bH(\bx)\cdot(\bof_{j_1}-\bof_{j_2})(\bx)\mbox{ where }
    \bH(\bx)=\int_0^1\frac{\partial F}{\partial \by}(\bx,\tau\bof_{j_1}(\bx)+(1-\tau)\bof_{j_2}(\bx))\,d\tau.$$
As with $k=1/s$
$$
\norm{\bof_{j_1} (\bx )-\bof_{j_2} (\bx )} \leq K\exp \left( -\frac{A}{\norm{P(\bx )}^k} \right)
$$
for certain $K$, $A>0$, we obtain that
$$
\norm{g_{j_1} (\bx )-g_{j_2} (\bx )} \leq \tilde{K}\exp \left( -\frac{A}{\norm{P(\bx )}^k} \right)
$$
for appropriate $\tilde{K}>0$, reducing radii if necessary. These estimates and Theorem \ref{ramissibuya-poly} show 
that every $g_i( \bx) $, $i=1,\ldots,r$, has a $P$-$s$-asymptotic expansion, and the result follows.
\end{proof}

\begin{nota}
This result provides an alternative proof that the product of functions having $P$-$s$-Gevrey asymptotic expansions has a 
$P$-$s$-Gevrey asymptotic expansion. Indeed, just take  $F(\bx,y_1,y_2)= y_1\cdot y_2$ in the above Theorem.
\end{nota}

Concerning the partial derivatives of a function having $P$-$s$-Gevrey asymptotic expansion, 
we can proceed as above first embedding the function in a $P$-$k$-quasifunction using Proposition
\reff{existcover} and applying the Ramis-Sibuya Theorem \reff{ramissibuya-poly} to the derivatives.
For this approach, we have to show that, if $f(\bx )\in \OO (\Pi_P )$ verifies an estimate
\begin{equation}
|f(\bx ) | \leq C\cdot \exp \left( -\frac{A}{|P(\bx )|^k} \right), \labl{cotaexpP}
\end{equation}
then their derivatives satisfy similar estimates. 

To show this, consider $\Pi_P=\Pi_P(a,b; r)$, $f\in \OO (\Pi _p)$ verifying (\reff{cotaexpP}).  
Lemma \reff{main-step} shows the existence of a bounded holomorphic function $F: V(a,b;\sigma )\times D(\bo; \rho )\rightarrow \CC$ with $F(P(\bx ), \bx )=f(\bx )$ and
$$
F(t,\bx )\leq \frac{CL}{|t|} \cdot \exp \left( -\frac{A}{|t|^k} \right).
$$
Choosing some positive $A'<A$, there exist $C'>0$ such that
$$
|F(t,\bx )| \leq C'\exp \left( -\frac{A'}{|t|^k} \right).
$$
Taking derivatives, we obtain 
$$
\frac{\partial f}{\partial x_i } (\bx )= \frac{\partial F}{\partial t} (P(\bx ), \bx ) \cdot \frac{\partial P}{\partial x_i} (\bx )+\frac{\partial F}{\partial x_i}(P(\bx ), \bx ).
$$
Now choose $a<\alpha<\beta<b$, $0<\sigma'<\sigma$ and $0<\rho'<\rho$. Applying Cauchy's formula for  the derivative of a holomorphic function,
we obtain in a well known way the existence of certain positive $D,B$ such that
$$\norm{\frac{\partial F}{\partial t} (t, \bx )}\leq  D\exp \left( -\frac{B}{|t|^k} \right)\mbox{ and }
\norm{\frac{\partial F}{\partial x_i}(t,\bx )}\leq  D\exp \left( -\frac{B}{|t|^k} \right)$$
for $(t,\bx)\in V(\alpha,\beta; \sigma') \times D(\bo; \rho')$.
We obtain the existence of positive $\tilde{C}$, $\tilde{A}$ such that
$$
\norm{\frac{\partial f}{\partial x_i} (\bx )} \leq \tilde{C}\cdot \exp \left( -\frac{\tilde{A}}{|P(\bx )|^k} \right).
$$
as needed. Thus we have proved 
\begin{propo}
If $\Pi= \Pi_P (a,b; r)$ is a $P$-sector and $f\in \OO_b (\Pi )$ has a $P$-$s$-Gevrey asymptotic expansion, then for $\alpha$, $\beta$ such that $a<\alpha<\beta<b$, $\frac{\partial f}{\partial x_i} (\bx )$ has a $P$-$s$-Gevrey asymptotic expansion in $\Pi_P (\alpha,\beta; r)$.
\end{propo}

\newcommand{\dif}[2]{\frac{\partial #1}{\partial #2}}
\newcommand{\gk}[1]{\left(#1\right)}

\section{Examples}\labl{examp}
\subsection{First example}We consider the following singular ordinary differential equation
depending upon a small parameter $\eps$
\newcommand{\ddif}[2]{{\frac{d\,#1}{d\,#2}}}
\begin{equation}\labl{ode1}
P(x,\eps)^2 \ddif yx = (P'(x,\eps)+A(x,\eps))y+B(x,\eps)+y^2f(x,\eps,y),
\end{equation}
where $P(x,\eps)=x^n+...$ is a homogeneous polynomial having $n\geq2$ simple roots if $\eps\neq0$, $P'=\ddif Px$,
$A,B,f$ are holomorphic near the origin in $\CC^2$ resp.\ $\CC^3$, $A$ has a homogeneous valuation
$w(A)\geq n$ and $w(B)\geq 2n-1$ if $f$ is not identically 0. 
If $\eps\neq0$ then (\reff{ode1}) has $n$ finite irregular singular points of Poincar\'e rank 1. If 
$\eps=0$, then $x=0$ is an irregular singular point of Poincar\'e rank $n$ after the reduction
$y=x^{n-1}z$. This means that our equation has $n$ coalescing irregular singular points.
\begin{teorema}Suppose that (\reff{ode1}) has a formal solution $y(x,\eps)\in\CC[[x,\eps]]$. Then it is
$P$-1-summable in every direction $d\not\equiv0\mod2\pi\ZZ$.
\end{teorema}
\begin{nota} The existence of such a formal solution is quite exceptional. It is nevertheless possible
to have a divergent formal solution, as the simple example
$y=\sum_{m=0}^\infty m!\,P^{m+1}$ solution of $P^2\ddif yx=P'y-P\,P'$ shows.
\end{nota}
\begin{proof}

We apply Theorem \reff{sumable-blu+} with $x_1$ replaced by $x$, $x_2$ replaced by $\eps$.
Then the set $\mcZ$ is the set of zeroes of the polynomial $P(u,1)$ -- it contains exactly $n$ elements.
We have to show that $y\circ b_\xi$ is $P\circ b_\xi$-1-summable in every direction $d\not\equiv0\mod2\pi\ZZ$
for any $\xi\in\mcZ$.

For simplicity, we consider only the case $\xi=0\in\mcZ$; the modification for arbitrary $\xi\in\mcZ$ is left to the reader.
Here, we perform the blow-up $x=\eps u$ and at the same time change the dependent variable by putting
$y=\eps^{n-1}z$.
We obtain the doubly singular equation
\footnote{As there exists a formal solution, we can assume that $w(B)\geq2n-1$ also in the linear case.} 
\begin{equation}\labl{ode2}
\eps^n u^2\,\ddif zu=\tfrac1{Q(u)^2}\gk{(P'(u,1))+\eps \tilde A( u,\eps))\,z+\eps\tilde B(u,\eps)+z^2f(\eps u,\eps,\eps^{n-1}z)},
\end{equation}
where $Q(u)=P(u)/u$ is a polynomial satisfying $Q(0)=P'(0,1)$ and $\tilde A$, $\tilde B$ are analytic near the origin.

The main result of \cite{CDMS} applies to this equation and yields that its unique formal solution $z(u,\eps)$ is $\eps^nu$-1-summable
in every direction not in $-\arg(P'(0,1))\mod2\pi\ZZ$.
It remains to multiply $\eps^nu$ by the unit $Q(u)$. Observing that $Q(0)=P'(0,1)$ we obtain the
$z(u,\eps)$ is $\eps^nP(u,1)$-1-summable for every direction not in $2\pi\ZZ$.
\end{proof}

\subsection{Second example}
\newcommand{\bu}{{\boldsymbol u}}
We consider the following singular partial differential equation
\begin{equation}\labl{pde1}
\gk{x_2\dif P{x_2} + \alpha P^{k+1}+P\,A}x_1\dif f{x_1} - \gk{x_1\dif P{x_1} + \beta P^{k+1}+P\,B}x_2\dif f{x_2}=h,
\end{equation}
where $P$ is a quasi-homogeneous polynomial for the valuation 
$w$ determined by $w(x_1)=a$, $w(x_2)=b$, where $k\in\NN^*$, $h,A,B$ are convergent power series and $w(A),w(B)>k\,g$, $g=w(P)$.
\begin{teorema} If (\reff{pde1}) has a formal solution $f(\bu)\in\CC[[\bu]]$ then it is $P$-$k$-summable provided $\alpha,\,\beta$
satisfy the following conditions.
\begin{enumerate}[itemsep=0.3em]
\item $\alpha+\beta\neq0$ if $P$ is not a monomial,
\item $a\mu_0\beta\neq(g-b\mu_0)\alpha$ if $x_2$ is a factor of $P$ of multiplicity 
     $\mu_0>0$,
\item $b\mu_\infty\alpha\neq(g-a\mu_\infty)\beta$ if $x_1$ is a factor of $P$ 
     of multiplicity $\mu_\infty>0$.

\end{enumerate}
\end{teorema}
\begin{nota}\benot\item The existence of a formal solution is a very strong hypothesis, comparable to the necessary condition
(see \cite{coe}, \cite{Lak})
for Ackerberg-O'Malley resonance \cite{AO}. The investigation of conditions for the existence of formal solutions, 
also for non-linear equations, will be done in a future article.
\item A formal solution of (\reff{pde1}) is not necessarily convergent, as the simple example
$f(\bx)=x_1\dis\sum_{n=0}^\infty n!\, P(\bx)^{n+1}$ with any quasi-homogeneous polynomial $P$ shows.
It is solution of (\reff{pde1}) with $k=1$, $\alpha=0$, $\beta=1,\ A=B=0$ and $h(\bx)=x_2\dif{P}{x_2}P.$
\item In the case of a monomial $P$, equations like (\reff{pde1}) have been studied by Pingli Li \cite{Li}
using a different method.
\ee\end{nota}

\begin{proof}In a first step, we prove the statement in the case of monomials $P$. Then we reduce the statement for
homogeneous and then quasi-homogeneous polynomials $P$ to the former one using 
the Theorems \reff{sumable-blu+} and \reff{sumable-ram}.

In the case of a monomial $P=x_1^cx_2^d$, $c,d$ positive integers,
the choice of $a$ and $b$ is arbitrary. The conditions (2) and (3) are both equivalent to
the condition that $\alpha/d\neq\beta/c$.
The right hand side of equation \reff{pde1}, moreover, is  divisible by $P$, because the left hand side is. We obtain the equation
\begin{equation}\labl{pde2}
\gk{d + \alpha P^{k}+A}x_1\dif f{x_1} - \gk{c + \beta P^{k}+B}x_2\dif f{x_2}= h/P.
\end{equation}
For later use, it is convenient to relax the conditions on $A,B$ slightly. We we also allow terms
in $A,B$ that are divisible by $x_1^{kc+1}$ or $x_2^{kd+1}$. Let $\mathcal I$ denote the ideal of all
convergent series satisfying this condition.

We can now apply a result of Martinet and Ramis \cite{MR2} and obtain
a formal change of coordinates $\bx=\tilde\bx+\psi(\tilde\bx)$ where both components of $\psi$
are in $\mathcal I$ such that the vector field
$$\gk{d + \alpha P^{k}+A}x_1\dif \,{x_1} - \gk{c + \beta P^{k}+B}x_2\dif \,{x_2}$$
in (\reff{pde2}) is reduced its formal normal form $\gk{d + \alpha P^{k}}\tilde x_1\dif \,{\tilde x_1} - 
\gk{c + \beta P^{k}}\tilde x_2\dif \,{\tilde x_2}$
except for a unit factor $1+\phi(\tilde\bx)\in1+{\mathcal I}$.
It is proved in \cite{MR2}, furthermore, that $\phi$ and the components of $\psi$ are $P$-$k$-summable under our
condition $\alpha/d\neq\beta/c$.

It remains to solve a partial differential equation
\begin{equation}\labl{pde3}
\gk{d + \alpha P^{k}}x_1\dif f{x_1} - \gk{c + \beta P^{k}}x_2\dif f{x_2}= h,
\end{equation}
where $P=x_1^cx_2^d$ and $h$ is a $P$-$k$-summable power series such that (\reff{pde3}) admits a formal solution $f$.

If $c>1$ then replacing $x_1$ by $x_1\,e^{2\pi i/c}$ we find that$\tilde f(x_1,x_2)=f(x_1\,e^{2\pi i/c},x_2)$ is a solution of
(\reff{pde3}) with right hand side $\tilde h(x_1,x_2)=h(x_1\,e^{2\pi i/c},x_2)$.
This leads us to split $f$ and $h$ into $c$ series: $f(x_1,x_2)=\sum_{j=0}^{c-1}x_1^jf_j(x_1^c,x_2)$ and a similar formula
for $h$ and $h_j$, $j=0,...,c-1$. We obtain $c$ equations for the $f_j=f_j(u_1,x_2)$.
\begin{equation}\labl{pde4}
\gk{d + \alpha \bar P^{k}}u_1\dif {f_j}{u_1} - \gk{1 + \frac\beta c \bar P^{k}}x_2\dif {f_j}{x_2}+\frac jc (d+\alpha \bar P^k)f_j= \frac1c h_j,\ j=0,...,c-1,
\end{equation}
where $\bar P(u_1,x_2)=u_1\,x_2^d$.
In a similar way, splitting $f_j(u_1,x_2)=\sum_{m=0}^{d-1}x_2^mf_{jm}(u_1,x_2^d)$ we obtain
equations for each of the $f_{jm}$. We simplify the notation by omitting the indices and the constant factor $\frac1{cd}$
and by introducing 
$\tilde\alpha=\alpha/d$, $\tilde\beta=\beta/c$, $\tilde P(u_1,u_2)=u_1\,u_2$. This yields the equation
\begin{equation}\labl{pde5}
\gk{1 + \tilde\alpha \tilde P^{k}}u_1\dif {f}{u_1} - \gk{1 + \tilde\beta \tilde P^{k}}u_2\dif {f}{u_2}
+\gk{\gamma(1+\tilde\alpha \tilde P^k)-\mu(1+\tilde\beta \tilde P^k)}f=  h,
\end{equation}
with $\gamma\in\MM=\{0,\frac1c,\ldots,1-\frac1c\}\subset[0,1[$,
$\mu\in\mcN=\{0,\frac1d,\ldots,1-\frac1d\}\subset[0,1[$, where now $h=h(u_1,u_2)$ is $\tilde P$-$k$-summable.


Now we use Proposition \reff{laplace-poly} to establish the $\tilde P$-$k$-summability of a formal solution $f$
of (\reff{pde5}) provided $\tilde\alpha\neq\tilde\beta$. 

Consider any direction $\theta$ for which $h$ is $\tilde P$-$k$-summable and which is not
congruent to $-\frac1k\arg(\tilde\alpha-\tilde\beta)\mod\frac\pi k\ZZ$.
Then there is a function $H(\bu,\tau)$ holomorphic 
for $\bu\in D(\bo;\rho)$ and $\tau\in D(0,\delta)$ or $\tau\in V(\theta-\delta,\theta+\delta;\infty)$ 
with the properties (1), (2), (3) stated in Proposition \reff{laplace-poly}.
In particular, there are positive constants $C,A$ such that 
\begin{equation}\labl{estH}\norm{H(\bu,\tau)}\leq C\,\exp\gk{A\norm\tau^k}\mbox{ if }\norm\bu<\rho,\ 
\tau\in (D(0,\delta)\cup V(\theta-\delta,\theta+\delta;\infty)),\end{equation}
for every $\tau$, we have $J_\bu (H(.,\tau)) \in \CC\{u_1\}+\CC\{u_2\}$, the Laplace Transforms
$$\LL(H)(\bu,t)=k\,t^{-k}\int_{\arg\tau=\tilde\theta} e^{-\tau^k/t^k} H(\bu,\tau)\tau^{k-1}\,d\tau$$
converge for $\tilde\theta$ near $\theta$
and we have $h(\bu)$ as $\tilde P$-$k$-Gevrey asymptotic expansion of $\LL(H)(\bu,\tilde P(\bu))$ in some $\tilde P$-sector
of angular opening larger that $\frac{2\pi}k$ bisected by $\arg(\tilde P(\bu))=\theta$.

\newcommand{\ta}{{\tilde\alpha}}\newcommand{\tb}{{\tilde\beta}}
Suppose for a moment that there is a function $F=F(\bu,\tau)$ with properties similar to $H$ such that
$\LL(F)(\bu,\tilde P(\bu))$ is a solution of (\reff{pde5}) having $f(\bu)$ as its $\tilde P(\bu)$-asymptotic expansion.
Using classical properties of the Laplace transform we obtain 
that $u_r\dif{f}{u_r}(\bu,\tilde P(\bu))=\LL(u_r\dif{F}{u_r}+\tau\dif{F}\tau)(\bu,\tilde P(\bu))$, $r=1,2$, and $t^k\LL(G)(\bu,t)=\LL(kI(G))(\bu,t)$
for $G$ with properties like $F$, where $I(G)(\bu,\tau)=\int_0^\tau\sigma^{k-1} G(\bu,\sigma)\,d\sigma$.
Inserting into (\reff{eqFH}) this yields that $F$ satisfies the following equation
\begin{equation}\labl{eqFH}\begin{array}{l}
\dis u_1\dif{F}{u_1}-u_2\dif{F}{u_2}+
{k(\ta-\tb)(\tau^k F-k\,I(F))+k\ta u_1\dif{I(F)}{u_1}-k\tb u_2\dif{I(F)}{u_2}}+\\
\hspace*{6cm}\dis\ \ \ +\,(\gamma-\mu)F+k(\gamma\ta-\mu\tb)I(F)=H.
\end{array}\end{equation}
Conversely, the existence of a solution $F$ of (\reff{eqFH}) with properties analogous to (1), (2), (3)
proves the  $\tilde P$-$k$-summability of a formal solution $f$ of (\reff{pde5}), because
(\reff{pde5} has a unique formal solution except for its constant term in the special case $\ell=\gamma=\mu=0$. 
This last assertion follows
by consideration of the terms of $f$ of lowest valuation. Details are left to the reader.

In order to solve (\reff{eqFH}), we expand $H$ and $F$ into series with respect to $u_1,u_2$. We have
$H(\bu,\tau)=\sum_{\ell=0}^\infty H_\ell(\tau)u_1^\ell+\sum_{\ell=1}^\infty H_{-\ell}(\tau)u_2^\ell$ and
want
$F(\bu,\tau)=\sum_{\ell=0}^\infty F_\ell(\tau)u_1^\ell+\sum_{\ell=1}^\infty F_{-\ell}(\tau)u_2^\ell$.
Equation (\reff{eqFH}) is equivalent to a sequence of equations for the functions $G_\ell=I(F_\ell)$
\begin{equation}\labl{eq-seq}
\gk{\ell+\gamma-\mu+k(\ta-\tb)\tau^k}G_\ell'+k\tau^{k-1}\gk{\chi\ell-k(\ta-\tb)+\gamma\ta-\mu\tb}G_\ell=
\tau^{k-1}H_\ell, \ G_\ell(0)=0,
\end{equation}
for $\ell\in\ZZ$, where $\chi=\ta$ if $\ell\geq0$, $\chi=\tb$ otherwise.

If $\ell+\gamma-\mu\neq0$, i.e.\ $\ell\neq0$ or $\gamma=\mu$, we can write the solutions of (\reff{eq-seq})
using the solutions of the corresponding homogeneous equations
$$U_\ell(\tau)=\gk{1+\tfrac{k(\ta-\tb)\tau^k}{\ell+\gamma-\mu}}^{-(\chi\ell-k(\ta-\tb)+\gamma\ta-\mu\tb)/(k(\ta-\tb))}.$$
We obtain
$$G_\ell(\tau)=U_\ell(\tau)\int_0^\tau \tfrac{ \tau^{k-1}H_\ell(s)}{\gk{\ell+\gamma-\mu+k(\ta-\tb)s^k}U_\ell(s)}\,ds,
\ F_\ell(\tau)=\tfrac{H_\ell(\tau)-k\gk{\chi\ell-k(\ta-\tb)+\gamma\ta-\mu\tb}G_\ell}
                  {{\ell+\gamma-\mu+k(\ta-\tb)\tau^k}}.  $$

By our hypothesis on $\theta$, we can assume that $q=\tfrac{k(\ta-\tb)\tau^k}{\ell+\gamma-\mu}$ is in the sector
$\norm{\arg q}<\pi-\delta/2$ and therefore $\frac1q\log(1+q)$ is bounded by some constant. This implies that there are constants
$K,M$ such that $\norm{U_\ell(\tau)}$ and $1/\norm{U_\ell(\tau)}$ are bounded by $M\,e^{K\norm\tau^k}$ for
$\tau\in D(0,\delta)\cup V(\theta-\delta/2,\theta+\delta/2;\infty)$ and all $\ell\in\ZZ,$ $\gamma\in\MM,\mu\in\mcN$,
$\ell\neq0,\,\gamma\neq\mu$.
By (\reff{estH}), we have $\norm{H_\ell(\tau)}\leq C\,\delta^{-\norm{\ell}}\,e^{A\norm{\tau}^k}$ for all $\ell,\gamma,\mu$
and all $\tau$. The above formula for $F_\ell$ implies that there is a constant $D$ such that
$\norm{F_\ell(\tau)}\leq D\,\delta^{-\norm{\ell}}\,e^{(A+2K+1)\norm{\tau}^k}$ for all $\ell,\gamma,\mu,\tau$
in the case $\ell\neq0,\,\gamma\neq\mu$. 

Observe that the formula for $F_\ell$ remains valid in the exceptional case $\ell=0$, $\gamma=\mu$, if we replace
$U_0$ by $\tilde U_0(\tau)=\tau^{k-\gamma}$.
It is straightforward to prove that $\norm{F_0(\tau)}\leq D\,e^{(A+2K+1)\norm{\tau}^k}$
for a certain $D$ in the finitely many cases $\gamma\in\MM,\,\mu\in\mcN$, $\gamma=\mu$.

This implies that 
$$\norm{F(\bu,\tau)}\leq \frac {4D}{2-\delta}\,e^{(A+2K+1)\norm{\tau}^k}$$
for $\bu\in D(\bo;\delta/2)$, $\tau\in  D(0,\delta)\cup V(\theta-\delta/2,\theta+\delta/2;\infty)$.
Thus the $\tilde P(\bu)$-$k$-summability of the formal solution $f$ of (\reff{pde5}) is proved.

Going back to the formal solution of equation (\reff{pde1}) in the case of the monomial $P=x_1^cx_2^d$, we
obtain its $P$-$k$-summability provided $\alpha/d\neq \beta/c$. In the terms of the Theorem, this
condition is equivalent to conditions (2) and (3) and the Theorem is proved for monomials $P$.

In a second step, we consider (\reff{pde1}) with a homogeneous polynomial $P$ that is not a monomial, i.e.\ the case $a=b$.
We can assume that $a=b=1$.
We apply Theorem \reff{sumable-blu+} to the formal solution $f$ of (\reff{pde1}) that exists according to the
assumption of the Theorem. Let $\MM$ denote the set of zeroes of $P$ in $\PC$. If we show 
that $f\circ b_\xi$ is $P\circ b_\xi$-$k$-summable for $\xi\in\MM$ then the statement follows in the present case.

It is sufficient to consider the  case where $\xi\in\CC$,
the case $\xi=\infty$ is reduced to the case $\xi=0$ by exchanging $v_1$ and $v_2$,
$\alpha$ and $\beta$, $a$ and $b$ respectively.

Now we consider $\tilde f=f\circ b_\xi$ and the analogously constructed
$\tilde h$, $\tilde P$, $\tilde A$, $\tilde B$. We calculate
\begin{equation}
\dif f {x_2}(v_2,(\xi+v_1)v_2)=\frac1{v_2}\dif {\tilde f}{v_1}(v_1,v_2),\ \ 
\dif f {x_1}(v_2,(\xi+v_1)v_2)=\gk{\dif {\tilde f}{v_2}-\frac{\xi+v_1}{v_2}\dif {\tilde f}{v_1}}
(v_1,v_2)
\end{equation}
and obtain analogous formulas for $\tilde P$. This leads to the following 
partial differential equation for $\tilde f$.
\begin{equation}\labl{pdexi}\begin{array}l
\dis\gk{(\xi+v_1)\dif {\tilde P}{v_1}+\alpha \tilde P^{k+1}+\tilde P\tilde A}v_2\dif {\tilde f}{v_2}-\\\dis
\hspace*{10em}-\gk{v_2\dif {\tilde P}{v_2}+
 (\alpha+\beta) \tilde P^{k+1}+\tilde P(\tilde A+\tilde B)}(\xi+v_1)
\dif {\tilde f}{v_1}=\tilde h
\end{array}
\end{equation}
where now the $v_2$-valuation of $\tilde A$ and $\tilde B$ is larger than $kg$, 
$g=w(P)$, and $\tilde P(v_1,v_2)=P(v_2,(\xi+v_1)v_2)=v_2^gQ(v_1)$ with some 
polynomial $Q$ vanishing at
the origin.
 
In order to reduce $\tilde P$ to a monomial, we consider a holomorphic function $\phi$
such that $Q(\phi(v_1))=v_1^c$, where $c$ is the multiplicity of $v_1=0$ as a zero of 
$Q$. Then replacing $v_1$ by $\phi(v_1)$, multiplying by $v_1\phi'(v_1)$ and dividing
by $\xi+\phi(v_1)$, we obtain a new equation for
$\bar f(v_1,v_2)=\tilde f(\phi(v_1),v_2)$ with 
$\bar P=v_1^cv_2^d$.
\begin{equation}\labl{pdexi2}\begin{array}l
\dis
\gk{v_1\dif {\bar P}{v_1}+\frac{v_1\phi'(v_1)}{\xi+\phi(v_1)}
  (\alpha\bar P^{k+1}+\bar P \bar A)}
  v_2\dif {\bar f}{v_2}-\\
\dis
\hspace*{10em}-\gk{v_2\dif {\bar P}{v_2}+(\alpha+\beta)\bar P^{k+1}+
    \bar P(\bar A+\bar B)}v_1\dif {\bar f}{v_1}=\bar h           
\end{array}
\end{equation}
where $\bar h(v_1,v_2)=\frac{v_1\phi'(v_1)}{\xi+\phi(v_1)}\tilde h(\phi(v_1),v_2)$ and
$\bar A$, $\bar B$ are certain convergent series containing only terms
of $v_2$-valuation larger than $kd$.
Choosing the quasi-homogeneous valuation appropriately, more precisely $a=1$ and
$b$ sufficiently large, we arrive at an equation of the form
(\reff{pde1}) with the monomial $P=x_1^cx_2^g$.
In the case $\xi\neq0$, the factors $\alpha,\beta$ replaced by $\alpha+\beta,0$
and $f\circ b_\xi$ is $P\circ b_\xi$-$k$-summable if $\alpha+\beta\neq0$;
this is satisfied by condition (1).
In the case $\xi=0$, these factors have to be replaced by 
$(\alpha+\beta,\alpha)$. As $c=\mu_0>0$ and $d=g$ here, we obtain
that $f\circ b_0$ is $P\circ b_0$-$k$-summable if $(\alpha+\beta)/g\neq \alpha/\mu_0$.
This condition is equivalent  to (2).
In the case $\xi=\infty$ the reduction to the case $\xi=0$, that is 
exchange $x_1$ and $x_2$ etc., proves that $f\circ b_\infty$ is $P\circ b_\infty$-$k$-summable if 
$(\alpha+\beta)/g\neq \beta/\mu_\infty$. This is condition (3).
Thus the Theorem is proved for homogeneous non-monomial $P$.

In a third and last step, we reduce a quasi-homogeneous non-monomial case 
to a homogeneous case by the ramifications $x_1=v_1^a$, $x_2=v_2^b$. 
It suffices to use  Theorem \reff{sumable-ram} twice.
The ramifications preserve the form (\reff{pde1}) of the equation,
only $\alpha,\beta$ are replaced by $b\alpha,a\beta$,
it is homogeneous for $w$ satisfying $w(v_1)=w(v_2)=1$
and $\mu_0,\mu_\infty$ are replaced by $b\mu_0,a\mu_\infty$, but
$g$ remains unchanged.
This yields the conditions of the Theorem for
$P$-$k$-summability in the present case.
\end{proof}

\end{document}